\documentclass[a4paper]{amsart}
\usepackage{hyperref,esnault-sabbah-yu}

\begin{document}
\title[$E_1$-degeneration of the irregular Hodge filtration]{$E_1$-degeneration of\\ the irregular Hodge filtration\\\smaller\smaller(with an appendix by Morihiko Saito)}

\author[H.~Esnault]{Hélène Esnault}
\address[H.~Esnault]{Freie Universität Berlin\\
FB Mathematik und Informatik\\
\hbox{Arnimallee 3}, Zimmer~110\\
14195 Berlin\\
Germany}
\email{esnault@math.fu-berlin.de }
\urladdr{http://userpage.fu-berlin.de/esnault/}
\thanks{(H.E.) This work was partially supported by ERC Advanced Grant 226257, the SFB/TR 45, the Chaire d'Excellence 2011 and the Einstein Foundation.}

\author[C.~Sabbah]{Claude Sabbah}
\address[C.~Sabbah]{UMR 7640 du CNRS\\
Centre de Mathématiques Laurent Schwartz\\
\'Ecole polytechnique\\
F--91128 Palaiseau cedex\\
France} \email{sabbah@math.polytechnique.fr}
\urladdr{http://www.math.polytechnique.fr/~sabbah}
\thanks{(C.S.) This research was supported by the grant ANR-08-BLAN-0317-01 of the Agence nationale de la recherche.}

\author[J.-D.~Yu]{Jeng-Daw Yu}
\address[J.-D.~Yu]{Department of Mathematics\\
National Taiwan University\\
Taipei 10617\\
Taiwan}
\email{jdyu@math.ntu.edu.tw}
\urladdr{http://www.math.ntu.edu.tw/~jdyu/}
\thanks{(J.-D.Y.) This work was partially supported by the Golden-Jade fellowship of Kenda Foundation, the NCTS (Main Office) the NSC, Taiwan, and SFB/TR 45.}


\begin{abstract}
For a regular function $f$ on a smooth complex quasi-projective variety, J.-D.\,Yu introduced in \cite{Yu12} a filtration (the irregular Hodge filtration) on the de Rham complex with twisted differential $\mathrm d+\mathrm df$, extending a definition of Deligne in the case of curves. In this article, we show the degeneration at~$E_1$ of the spectral sequence attached to the irregular Hodge filtration, by using the method of \cite{Bibi08}. We also make explicit the relation with a complex introduced by M.\,Kontsevich and give details on his proof of the corresponding~$E_1$-degeneration, by reduction to characteristic $p$, when the pole divisor of the function is reduced with normal crossings. In Appendix E, M.\,Saito gives a different proof of the latter statement with a possibly non reduced normal crossing pole divisor. 
\end{abstract}

\maketitle
\thispagestyle{empty}

\enlargethispage{\baselineskip}%
{\let\\\relax\let\footnotemark\relax\tableofcontents}

\mainmatter
\section*{Introduction}
Let $U$ be a smooth complex algebraic curve and $X$ be a smooth projectivization of it. Let $f:X\to\PP^1$ be a rational function which is regular on $U$ and let $(V,{}^V\!\nabla)$ be an algebraic bundle with connection on $U$ such that ${}^V\!\nabla$ has a regular singularity at each point of the divisor $D=X\setminus U$. Assuming that the monodromy of $(V,{}^V\!\nabla)$ is unitary, Deligne defines in \cite{Deligne8406} a filtration on the twisted de~Rham cohomology $H^*_{\DR}(U,\nabla)$, where $\nabla:={}^V\!\nabla+\rd f\id_V$. He shows various properties of it, among which $E_1$-degeneration of the associated spectral sequence. This is the first occurrence of a filtration having some kind of Hodge properties in the realm of connections with irregular singularities.

Taking the opportunity of the recently defined notion of polarized twistor $\cD$\nobreakdash-module, the second author has extended in \cite{Bibi08} the construction of Deligne, and shown the $E_1$-degeneration property correspondingly, for the case where $(V,{}^V\!\nabla)$ underlies a variation of polarized complex Hodge structure on the curve~$U$.

On the other hand, the third author has extended in \cite{Yu12} the construction of such a filtration in higher dimensions for $(V,{}^V\!\nabla)=(\cO_U,\rd)$ and any regular function $f:U\to\Afu$, by using a projectivization of $U$ with a normal crossing divisor~$D$ at infinity. He succeeded to prove the degeneration at $E_1$ in various special cases.

It is then a natural question to ask whether or not the generalized Deligne filtration as defined by Yu has the property that the induced spectral sequence on the hypercohomology of the twisted de Rham complex  degenerates in $E_1$. 

This question has a positive answer. This is our main theorem (Theorem \ref{th:main}).

Approximately at the same time the preprint \cite{Yu12} was made public, and independently of Deligne's construction, but in the setting of a function $f:X\to\PP^1$, Kontsevich introduced in letters to Katzarkov and Pantev \cite{Kontsevichl12} a family of complexes $(\Omega_f^\cbbullet,u\rd+v\rd f)$, $(u,v)\in\CC^2$, and sketched a proof of the independence of the hypercohomology with respect to $u,v$, giving rise in particular to the $E_1$\nobreakdash-degeneration property (see also the recent preprint \cite{K-K-P14}). We give details on the proof sketched by Kontsevich in \S\ref{subsec:implies} and Appendix \ref{sec:Kontsevich}. However the method suggested by Kontsevich requests the pole divisor $P=f^{-1}(\infty)$ to be reduced, an annoyance, as this property is not stable under blow up along the divisor, while the result is. In \S\ref{subsec:Kontsevichcomplex}, we give details on the relation between two kinds of filtered complexes, namely $(\Omega_f^\cbbullet,\rd+\rd f)$ with the stupid filtration, and $(\Omega^\cbbullet_X(\log D)(*P),\rd+\rd f)$ equipped with the filtration introduced by the third author in \cite{Yu12}. In particular, due to the results in \cite{Yu12}, our Theorem \ref{th:main} implies the $E_1$-degeneration property for $(\Omega_f^\cbbullet,\rd+\rd f)$. The $E_1$-degeneration for other values $u,v$ follows by already known arguments. On the other hand, the degeneration property for the Kontsevich complex $(\Omega_f^\cbbullet,\rd+\rd f)$ would implies the $E_1$-degeneration of our Theorem \ref{th:main} for integral values of the filtration introduced in \cite{Yu12}. (\Cf\S\ref{subsec:implies}.)

More recently, M.\,Saito proposed a new proof of the $E_1$-degeneration for $(\Omega_f^\cbbullet,\rd)$ which relies on older results of Steenbrink \cite{Steenbrink76,Steenbrink77}. His proof gives a new proof of our Theorem \ref{th:main} for integral values of the filtration introduced in \cite{Yu12}. Moreover, when $P$ is reduced, he is also able to give an interpretation of the hypercohomology of this complex in term of the Beilinson functor applied to the complex $\bR j_*\CC_U$, $j:U\hto X$. This is explained in Appendix \ref{app:saito} by M.\,Saito.

In\marginpars{The introduction is changed starting form here} conclusion, the methods of Kontsevich and M.\,Saito focus first on the complex $(\Omega_f^\cbbullet,\rd)$, for which they extend known properties in Hodge theory. These methods give then results for the complex $(\Omega_f^\cbbullet,\rd+\rd f)$ by applying standard techniques not relying on Hodge theory, as explained in Appendix \ref{app:proof152}. On the other hand, by applying Theorem \ref{th:main}, the proof of which relies on the Fourier-Laplace transformation and twistor $\cD$-modules, one treats the complex $(\Omega_f^\cbbullet,\rd+\rd f)$ first, and specializes to the complex $(\Omega_f^\cbbullet,\rd)$.

Let us quote P.\,Deligne \cite[Note p.\,175]{Deligne8406}: ``Le lecteur peut se demander à quoi peut servir une filtration `de Hodge' ne donnant pas lieu à  une structure de Hodge.''

$\bbullet$
Deligne suggests that this Hodge filtration could control $p$-adic valuations of Frobenius eigenvalues. This is related to the work of Adolphson and Sperber \cite{A-S89} bounding from below the $L$-polygon of a convenient non-degenerate polynomial defined over $\ZZ$ by a Hodge polygon attach to it, that they expect to be related to the limiting mixed Hodge structure at infinity of the polynomial. Appendix \ref{app:saito} should give the relation between the expectation of Deligne and that of Adolphson and Sperber.

$\bbullet$
In \cite{Kontsevich09}, Kontsevich defines the category of extended motivic-exponential $\cD$-modules on smooth algebraic varieties over a field~$k$ of characteristic zero as the minimal class which contains all $\cD_X$-modules of the type $(\cO_X,\rd+\rd f)$ for $f\in\cO(X)$ and is closed under extensions, sub-quotients, push-forwards and pull-backs. When $k=\CC$, the natural question is to define a Hodge filtration. Our work may be seen as a first step towards it.

$\bbullet$
For some Fano manifolds (or orbifolds), one looks for the mirror object as regular function $f:X\to\PP^1$ (Landau-Ginzburg potential). Kontsevich \cite{Kontsevichl12} conjectures that the Hodge numbers of the mirror Fano manifold (or orbifold) can be read on some ``Hodge filtration'' of the cohomology $H^k_\dR(X,\rd+\rd f)$. This filtration should be nothing but the irregular Hodge filtration.

\subsubsection*{Acknowledgements}
We thank Maxim Kontsevich for his interest and for showing us his complex $\Omega^\cbbullet_f$ and relating it in discussions with Yu's filtration. In addition, the second author thanks Morihiko Saito and Takuro Mochizuki for useful discussions.

\section{The irregular Hodge filtration in the normal crossing case}\label{sec:irregHodgenc}

\subsection{Setup and notation}\label{subsec:setup}
Let $X$ be a smooth complex projective variety with its Zariski topology and let $f:X\to\PP^1$ be a morphism. We will denote by $\Afu_t$ (\resp $\Afu_{t'}$) an affine chart of $\PP^1$ with coordinate $t$ (\resp $t'$) so that $t'=1/t$ in the intersection of the two charts. Let $U$ be a nonempty Zariski open set of $X$ such that
\begin{itemize}
\item
$f$ induces a regular function $f_{|U}:U\to\Afu_t$,
\item
$D:=X\setminus U$ is a normal crossing divisor.
\end{itemize}
Let us denote by $P$ the pole divisor $f^{-1}(\infty)$. Then the associated reduced divisor $P_\red$ has normal crossings and is the union of some of the components of $D$. The union of the remaining components of $D$ is denoted by $H$ (``horizontal'' components). We have a commutative diagram
\[
\xymatrix{
U\ar[d]_{f_{|U}}\ar@<-2pt>@{^{ (}->}[r]^-j&X\ar[d]^f\\
\Afu\ar@<-2pt>@{^{ (}->}[r]&\PP^1
}
\]
For each $k\geq0$, we will denote by $\Omega^k_X$ the sheaf of differential $k$-forms on $X$, by $\Omega^k_X(\log D)$ that of differential $k$-forms with logarithmic poles along $D$ and by $\Omega^k_X(*D)$ that of differential $k$-forms with arbitrary poles along $D$. Given any real number $\alpha$, $[\alpha P]$ will denote the divisor supported on $P_\red$ having multiplicity~$[\alpha e_i]$ on the component $P_i$ of $P_\red$, if $e_i$ denotes the corresponding multiplicity of~$P$. We will then set $\Omega^k_X(\log D)([\alpha P]):=\cO_X([\alpha P])\otimes_{\cO_X}\Omega^k_X(\log D)$.

When considering the various de~Rham complexes on $X$, we will use the analytic topology and allow local analytic computations as indicated below. However, all filtered complexes are already defined in the Zariski topology, and standard GAGA results (\cf\cite[\S II.6.6]{Deligne70}) allow one to compare both kinds of  hypercohomology on the projective variety $X$, so the results we obtain concerning hypercohomology also hold in the Zariski topology. We will not be more explicit on this point later.

Given any complex point of $f^{-1}(\infty)$, there exist an analytic neighbourhood $\Delta^\ell\times\Delta^m\times\Delta^p$ of this point with coordinates $(x,y,z)=(x_1,\dots,x_\ell,y_1,\dots,y_m,z_1\dots,z_p)$ and $\bme=(e_1,\dots,e_\ell)\in(\ZZ_{>0})^\ell$ with the following properties:
\begin{itemize}
\item
$f(x,y,z)=x^{-\bme}:=\prod_{i=1}^\ell x_i^{-e_i}$,
\item
$D=\bigcup_{i=1}^\ell\{x_i=0\}\cup\bigcup_{j=1}^m\{y_j=0\}$, $P_\red=\bigcup_{i=1}^\ell\{x_i=0\}$.
\end{itemize}

In this local analytic setting, we will set $g(x,y,z)=1/f(x,y,z)=x^{\bme}$.
The divisor~$H$ has equation $\prod_{j=1}^m y_j=0$. Finally, we set $n=\dim X$.

Set $\cO=\CC\{x,y,z\}$ and $\cD=\cO\langle\partial_x,\partial_y,\partial_z\rangle$ to be the ring of linear differential operators with coefficients in $\cO$, together with its standard increasing filtration $F_\bbullet\cD$ by the total order w.r.t.\,$\partial_x,\partial_y,\partial_z$:
\begin{equation}\label{eq:FpD}
F_p\cD=\sum_{|\alpha|+|\beta|+|\gamma|\leq p}\cO\partial_x^\alpha\partial_y^\beta\partial_z^\gamma,
\end{equation}
where we use the standard multi-index notation with $\alpha\in\NN^\ell$,  etc. Similarly we will denote by $\cO[t']$ the ring of polynomials in $t'$ with coefficients in $\cO$ and by $\cD[t']\langle\partial_{t'}\rangle$ the corresponding ring of differential operators.

Consider the left $\cD$-modules
\begin{align*}
\cO(*P_\red)&=\cO[x^{-1}],\\
\cO(*H)&=\cO[y^{-1}],\\
\cO(*D)&=\cO[x^{-1},y^{-1}]
\end{align*}
with their standard left $\cD$-module structure. They are generated respectively by $1/\prod_{i=1}^\ell x_i$, $1/\prod_{j=1}^m y_j$ and $1/\prod_{i=1}^\ell x_i\prod_{j=1}^m y_j$ as $\cD$-modules. We will consider on these $\cD$-modules the increasing filtration $F_\bbullet$ defined as the action of $F_\bbullet\cD$ on the generator:
\begin{align*}
F_p\cO(*P_\red)&=\sum_{|\alpha|\leq p}\cO\cdot\partial_x^\alpha(1/\prod\nolimits_{i=1}^\ell x_i)=\sum_{|\alpha|\leq p}\cO x^{-\alpha-\bf1},\\
F_p\cO(*H)&=\sum_{|\beta|\leq p}\cO\cdot\partial_y^\beta(1/\prod\nolimits_{j=1}^m y_j)=\sum_{|\beta|\leq p}\cO y^{-\beta-\bf1},\\
F_p\cO(*D)&=\sum_{|\alpha|+|\beta|\leq p}\cO\cdot\partial_x^\alpha\partial_y^\beta(1/\prod\nolimits_{i=1}^\ell x_i\prod\nolimits_{j=1}^m y_j)=\sum_{|\alpha|+|\beta|\leq p}\cO x^{-\alpha-\bf1}y^{-\beta-\bf1},
\end{align*}
so that $F_p=0$ for $p<0$. These are the ``filtration by the order of the pole'' in \cite[(3.12.1) p.\,80]{Deligne70}, taken in an increasing way. Regarding $\cO(*H)$ as a $\cD$-submodule of $\cO(*D)$, we have $F_p\cO(*H)=F_p\cO(*D)\cap\cO(*H)$ and similarly for $\cO(*P_\red)$. On the other hand it clearly follows from the formulas above that
\begin{equation}\label{eq:FpPH}
F_p\cO(*D)=\sum_{q+q'\leq p}F_q\cO(*H)\cdot F_{q'}\cO(*P_\red),
\end{equation}
where the product is taken in $\cO(*D)$.

\subsection{The irregular Hodge filtration}\label{subsec:FYu}
Our main object is the twisted meromorphic de~Rham complex
\[
(\Omega_X^\cbbullet(*D),\nabla)=\{\cO_X(*D)\To{\nabla}\Omega^1_X(*D)\To{\nabla}\cdots\To{\nabla}\Omega^n_X(*D)\},\quad\nabla=\rd+\rd f.
\]
This complex is equipped with the irregular Hodge filtration defined in \cite{Yu12}: this is the decreasing filtration indexed by $\RR$ (with possible jumps only at rational numbers) defined by the formula
\begin{multline}\label{eq:FYu}
F^{\Yu,\lambda}(\Omega_X^\cbbullet(*D),\nabla)=F^\lambda(\nabla)\\
:=\Big\{\cO_X([-\lambda P])_+\To{\nabla}\Omega^1_X(\log D)([(1-\lambda)P])_+\To{\nabla}\cdots\\[-5pt]
\To{\nabla}\Omega^n_X(\log D)([(n-\lambda)P])_+\Big\},
\end{multline}
where, for $\mu\in\RR$,
\[
\Omega^k_X(\log D)([\mu P])_+=
\begin{cases}
\Omega^k_X(\log D)([\mu P])&\text{if }\mu\geq0,\\
0&\text{otherwise}.
\end{cases}
\]
We will also consider the associated increasing filtration
\[
F^\Yu_\mu(\Omega_X^\cbbullet(*D),\nabla):=F^{\Yu,-\mu}(\Omega_X^\cbbullet(*D),\nabla),
\]
and, for each $\alpha\in[0,1)$, the decreasing (\resp increasing) $\ZZ$-filtration $F^{\Yu,\bbullet}_\alpha$ (\resp $F^\Yu_{\alpha+\bbullet}$) defined by
\[
F^{\Yu,p}_\alpha(\Omega_X^\cbbullet(*D),\nabla):=F^{\Yu,-\alpha+p}(\Omega_X^\cbbullet(*D),\nabla)= F^\Yu_{\alpha-p}(\Omega_X^\cbbullet(*D),\nabla)
\]
(\resp $F^\Yu_{\alpha+p}(\Omega_X^\cbbullet(*D),\nabla)$). The filtration exhausts the subcomplex
\[
(\Omega_X^\cbbullet(\log D)(*P_\red),\nabla)\hto(\Omega_X^\cbbullet(*D),\nabla).
\]
The quotient complex is quasi-isomorphic to zero, according to \cite[Cor.\,1.4]{Yu12}. We refer to \cite[Cor.\,1.4]{Yu12} for a detailed study of this filtered complex and its hypercohomology. Let us only recall that, setting $\gr^\lambda_F=F^\lambda/F^{>\lambda}$, the graded $\cO_X$\nobreakdash-complex $\gr_{F^\Yu}^\lambda(\nabla)$ is supported on~$P_\red$ for $\lambda\not\in\ZZ$ and is quasi-isomorphic to~$0$ for $\lambda\leq0$ (\cf\cite[Cor.\,1.4]{Yu12}). Our main objective is to prove in general (\cf\S\ref{subsec:proofE1degen}) the conjecture made in \cite{Yu12}, and already proved there in various particular cases.

\begin{theorem}\label{th:main}
For each $\alpha\in[0,1)$, the spectral sequence of hypercohomology of the filtered complex $F^{\Yu,p}_\alpha(\Omega_X^\cbbullet(*D),\nabla)$ ($p\in\ZZ$) degenerates at $E_1$, that is, for each $\lambda\in\RR$ and $k\in\NN$, the morphism
\[
\bH^k\big(X,F^{\Yu,\lambda}(\nabla)\big)\to\bH^k\big(X,(\Omega_X^\cbbullet(*D),\nabla)\big)=:H^k_\dR(U,\nabla)
\]
is injective.
\end{theorem}

The image of this morphism is denoted by $F^{\Yu,\lambda}H^k_\dR(U,\nabla)$ and does not depend on the choice of the projective morphism~$f:X\to\PP^1$ extending~$f_{|U}:U\to\Afu$ satisfying the properties of the setup above (\cf\cite[Th.\,1.8]{Yu12}). We thus have
\begin{equation}\label{eq:grH}
\gr^\lambda_{F^\Yu}H^k_\dR(U,\nabla)=\bH^k\big(X,\gr_{F^\Yu}^\lambda(\nabla)\big),
\end{equation}
and $F^{\Yu,\lambda}H^k_\dR(U,\nabla)=H^k_\dR(U,\nabla)$ for $\lambda\leq0$.

Let us recall that this filtration was introduced (and the corresponding $E_1$\nobreakdash-degeneration was proved) by Deligne \cite{Deligne8406}, in the case where $U$ is a curve and where the twisted de~Rham complex is also twisted by a unitary local system. The generalization to the case of a variation of a polarized Hodge structure on a curve was considered in \cite{Bibi08}.

\subsection{The Kontsevich complex}\label{subsec:Kontsevichcomplex}
M.\,Kontsevich has considered in \cite{Kontsevichl12} the complexes $(\Omega_f^\cbbullet,u\rd+v\rd f)$ for $u,v\in\CC$, where
\begin{align*}
\Omega^p_f&=\big\{\omega\in\Omega^p_X(\log D)\mid\rd f\wedge\omega\in\Omega^{p+1}_X(\log D)\big\}\\
&=\ker\Big\{\Omega^p_X(\log D)\To{\rd f}\Omega^{p+1}_X(\log D)(P)/\Omega^{p+1}_X(\log D)\Big\}.
\end{align*}
Note that we have
\[
\Omega^0_f=\cO_X(-P),\quad\Omega^n_f=\Omega^n_X(\log D),
\]
and
\[
\Omega^p_f=\ker\Big\{\Omega^p_X(\log D)\To{\nabla}\Omega^{p+1}_X(\log D)(P)/\Omega^{p+1}_X(\log D)\Big\},
\]
since $\rd\Omega^p_X(\log D)\subset\Omega^{p+1}_X(\log D)$, hence $\nabla(\Omega^p_f)\subset\Omega^{p+1}_f$. Moreover, in the local analytic setting of \S\ref{subsec:setup}, using that $\rd f=f\cdot\rd\log f$ and setting $x'=(x_2,\dots,x_\ell)$, one checks that
\begin{equation}\label{eq:Omegafloc}
\Omega^p_f=\CC\{x,y,z\}\cdot\frac{\rd x^{\bme}}{x^{\bme}}\wedge\bigwedge^{p-1}\Big\{\frac{\rd x'}{x'},\frac{\rd y}{y},\rd z\Big\}+\CC\{x,y,z\}x^{\bme}\cdot\bigwedge^p\Big\{\frac{\rd x'}{x'},\frac{\rd y}{y},\rd z\Big\},
\end{equation}
so in particular $\Omega^p_f$ is $\cO_X$-locally free of finite rank.

The following result was conjectured by M.\,Kontsevich \cite{Kontsevichl12}:

\begin{theorem}\label{th:Kontsevich}
For each $k\geq0$, the dimension of $\bH^k\big(X,(\Omega_f^\cbbullet,u\rd+v\rd f)\big)$ is independent of $u,v\in\CC$ and is equal to $\dim H^k_\dR(U,\nabla)$. In particular,
\numstareq\begin{multline*}\tag{\theequation}\label{eq:Kontsevich}
\text{the spectral sequence $E_1^{p,q}=H^q(X,\Omega_f^p) \Rightarrow \bH^{p+q}\big(X,(\Omega_f^\cbbullet,\rd)\big)$}\\ \text{degenerates at $E_1$.}
\end{multline*}
\end{theorem}

The finite dimensionality of $H^q(X,\Omega_f^p)$ for each pair $p,q$ implies that \eqref{eq:Kontsevich} is equivalent to $\dim\bH^k\big(X,(\Omega_f^\cbbullet,\rd)\big)=\dim\bH^k\big(X,(\Omega_f^\cbbullet,0)\big)$ for each~$k$. We will explain two kinds of proofs of this theorem in \S\ref{subsec:implies}.\marginpars{\hbox{Explanation shortened}}

\begin{itemize}
\item
The argument sketched by Kontsevich in \cite{Kontsevichl12,Kontsevichl12b} starts by reducing the proof to \eqref{eq:Kontsevich} (\cf Proposition \ref{prop:Kontsevich2b} whose proof is detailed in Appendix \ref{app:proof152}). Then, when $P=P_\red$, the method of Deligne-Illusie \cite{D-I87} is used for proving \eqref{eq:Kontsevich} (\cf Appendix~\ref{sec:Kontsevich} for details on this proof).
\item
In Appendix \ref{app:saito}, Morihiko Saito provides a direct proof of  \eqref{eq:Kontsevich} without the restricting assumption $P=P_\red$.\marginpars{Reference to strict NC deleted.}

\item
On the other hand, we will apply Theorem \ref{th:main} (with $\alpha=0$ but without the restricting assumption $P=P_\red$) in order to get the $E_1$-degeneration for the differential $\nabla=\rd+\rd f$. We then apply Proposition \ref{prop:nablad}.
\end{itemize}

\begin{remark}
A consequence\marginpars{\hbox{Barannikov-Kontsevich} is not used any more in the proof. It was used in Old Prop.\,1.5.1.} of Theorem \ref{th:main} for $\alpha=0$, or equivalently of Theorem \ref{th:Kontsevich}, is the equality $\dim\bH^k\big(X,(\Omega_f^\cbbullet(*P_\red), \rd f)\big)=\dim H^k_\dR(U,\nabla)$, which is due to Barannikov and Kontsevich (\cf \cite[Cor.\,0.6]{Bibi97b} and \cite[Cor.\,4.27]{O-V07}).
\end{remark}

\subsection{Comparison of the filtered twisted meromorphic de~Rham complex and the filtered Kontsevich complex}\label{subsec:comparisonYuKontsevich}

For any coherent sheaf $\cF$ and for $\mu\in\RR$, we set $\cF([\mu P])=\cO_X([\mu P])\otimes\cF$. We\marginpars{Formulation changed.} define then
\begin{equation}\label{eq:omegafalpha}
\begin{aligned}
\Omega^k_f(\mu)&:=\ker\Big(\Omega^k_X(\log D)([\mu P])\\[-5pt]
&\hspace*{1.8cm}\To{\rd f}\Omega^{k+1}_X(\log D)([(\mu+1)P])/\Omega^{k+1}_X(\log D)([\mu P])\Big)\\
&\phantom{:}=\ker\Big(\Omega^k_X(\log D)([\mu P])\\[-5pt]
&\hspace*{1.8cm}\To{\nabla}\Omega^{k+1}_X(\log D)([(\mu+1)P])/\Omega^{k+1}_X(\log D)([\mu P])\Big),
\end{aligned}
\end{equation}
where the second equality follow from $\rd\big(\Omega^k_X(\log D)([\mu P])\big)\subset\nobreak\Omega^{k+1}_X(\log D)([\mu P])$. Since $\nabla(\Omega^k_f(\mu))\subset\Omega^{k+1}_f(\mu)$, we can also consider the complex
\[
(\Omega^\cbbullet_f(\mu),\nabla):=\big(\Omega_f^0(\mu)\To{\nabla}\Omega^1_f(\mu)\To{\nabla}\cdots\big)
\]
together with its stupid filtration
\[
\sigma^p(\Omega^\cbbullet_f(\mu),\nabla):=\big(\Omega_f^p(\mu)\To{\nabla}\Omega^{p+1}_f(\mu)\To{\nabla}\cdots\big)[-p].
\]
For $\mu\leq\mu'$ we thus have natural morphisms of filtered complexes $(\Omega^\cbbullet_f(\mu),\nabla,\sigma^p)\to(\Omega^\cbbullet_f(\mu'),\nabla,\sigma^p)$. For any $\lambda\in\RR$, define
\begin{multline*}
(F^0(\lambda),\nabla)=\Big( \cO_X([-\lambda P]) \To{\nabla} \Omega^1_X(\log D)([(1-\lambda)P])\to\cdots\\
\to \Omega^k_X(\log D)([(k-\lambda)P])\to\cdots \Big).
\end{multline*}
Then for any $\lambda\leq\lambda'$ the natural inclusion
\[
(F^0(\lambda'),\nabla) \to (F^0(\lambda),\nabla)
\]
is a quasi-isomorphism (\cite[Prop.1.3]{Yu12}).

\begin{proposition}\label{Prop:Kont-log}
Fix $\mu\in\RR$.
Consider the filtration $\tau^p$ on $F^0(-\mu)$ defined by
\[
\tau^p=\begin{cases}
F^0(-\mu) & \text{if $p\leq 0$,} \\
\big(\Omega^p_X(\log D)([\mu P]) \To{\nabla} \Omega^{p+1}_X(\log D)([(1+\mu)P])\to\cdots\big)[-p]& \text{if $p\geq 0$}.
\end{cases}
\]
Then the natural inclusion
\[
\big(\Omega_f^\cbbullet(\mu),\nabla,\sigma^p\big) \to\big(F^0(-\mu),\nabla,\tau^p\big)\quad(p\in\ZZ)
\]
is a filtered quasi-isomorphism. The same holds true if one replaces the connection $\nabla$ by the left multiplication with $\rd f$ in both complexes.
\end{proposition}

\begin{proof}
Since both filtrations satisfy $\sigma^p,\tau^p=0$ for $p>\dim X$ and are constant for $p\leq0$, it is enough to prove the isomorphism at the graded level, that is, to prove that the vertical morphism of complexes below is a quasi-isomorphism:
\[
\xymatrix@C=.6cm{
0\ar[r]&\Omega_f^p(\mu)\ar[d]\ar[r]&0\ar[d]\ar[r]&\cdots\\
0\ar[r]&\Omega^p_X(\log D)([\mu P])\ar[r]^-\nabla&\Omega^{p+1}_X(\log D)([(\mu+1) P])/\Omega^{p+1}_X(\log D)([\mu P])\ar[r]&\cdots
}
\]
According to \eqref{eq:omegafalpha}, this amounts to showing that the second row has zero cohomology in degrees $\geq p+1$. This follows from \cite[Prop.\,1.3]{Yu12}, which implies that the complex
\begin{multline*}
\cdots\to\Omega^p_X(\log D)([\mu P])\Big/\Omega^p_X(\log D)([(\mu-1)P])\\
\To{\nabla}\Omega^{p+1}_X(\log D)([(\mu+1) P])\Big/\Omega^{p+1}_X(\log D)([\mu P])\to\cdots
\end{multline*}
is quasi-isomorphic to zero.

Since $\nabla=\rd f$ on the graded objects $\gr_\sigma^p$ and $\gr_\tau^p$, the second assertion can be proved by the same argument.
\end{proof}

\begin{corollary}\label{cor:Kontsevich2}
The\marginpars{Old Cor.\,1.4.6. The remaining part of \S\ref{subsec:comparisonYuKontsevich} now contains statements of Old \S1.5.} two inclusions
\[
\big(\Omega_f^\cbbullet,\nabla\big) \to \big(\Omega_f^\cbbullet(*P_\red),\nabla\big)
\quad\text{and}\quad
\big(\Omega_f^\cbbullet, \rd f\big) \to \big(\Omega_f^\cbbullet(*P_\red), \rd f\big)
\]
are quasi-isomorphisms of complexes on $X$, as well as the inclusion $\big(\Omega_f^\cbbullet(*P_\red),\nabla\big)\to\big(\Omega_X^\cbbullet(*D),\nabla\big)$.\marginpars{This supplementary statement was used in Old Prop.\,1.5.1.}
\end{corollary}

\begin{proof}
Forgetting the filtrations in Proposition \ref{Prop:Kont-log}, one obtains the vertical quasi-isomorphisms in the commutative diagram
\[\xymatrix{
(\Omega_f^\cbbullet,\star) \ar[r]\ar[d]_\wr &
(\Omega_f^\cbbullet(mP),\star) \ar[d]^\wr \\
(F^0(0),\star) \ar[r]^(.45)\sim & (F^0(-m),\star). }\]
Here the arrows are natural inclusions, $\star=\nabla$ or $\rd f$, and $m$ is any non-negative integer. Since $\Omega_f^k(*P_\red)=\bigcup_m \Omega_f^k(mP)$, the first assertion follows. The second assertion is proved as in \cite[\S3.2]{Bibi97b}, since $\Omega_f^\cbbullet(*P_\red)=\Omega_X^\cbbullet(\log H)(*P_\red)$, as follows from the expression \eqref{eq:Omegafloc}.\marginpars{Argument added.}
\end{proof}

\begin{remark}
We\marginpars{This remark gathers \hbox{Old Rems 1.5.2 and 1.5.4.}} have $\bH^k\big(X,\Omega_X^\cbbullet(*D),\nabla\big)=H^k_\dR(U,\nabla)$. Let us recall (\cf\cite[Rem.\,04]{Bibi97b}) that, if we set $Y=f^{-1}(\Afu)=X\setminus P_\red$, then $\dim H^k_\dR(U,\nabla)$ is equal to the dimension of the $(k-1)$st hypercohomology of the vanishing cycles of $f_{|Y}:Y\to\nobreak\Afu$ with coefficients in the complex $\bR\iota_*\CC_U$, where $\iota:U\hto Y$ denotes the inclusion.

On the other hand, in Theorem \ref{th:C3} of Appendix \ref{app:saito}, M.\,Saito gives, when $P=P_\red$, an identification of the complex $\big(\Omega_f^\cbbullet, \rd\big)$ with the Beilinson complex attached to $\bR j_*\CC_U$, and therefore with the nearby cycle complex $\psi_g\bR j_*\CC_U$ (recall that $g=\nobreak1/f$) of $f$ along $f^{-1}(\infty)$. 
\end{remark}

\begin{corollary}\label{Cor:Kont-log}
For\marginpars{Old Cor.\,1.4.3.} any $\alpha\in[0,1)$, the natural inclusions
\[
\big(\Omega_f^\cbbullet(\alpha),\nabla,\sigma^p\big) \to
\big(F^0(\nabla), F_\alpha^p\big) \to
F_\alpha^{\Yu,p}(\Omega^\cbbullet(*D),\nabla)
\quad(p\in\ZZ)
\]
are quasi-isomorphisms of filtered complexes on $X$. Here (\cf\eqref{eq:FYu})
\[
F_\alpha^p\big(F^0(\nabla)\big)=\begin{cases}
F^0(\nabla) & \text{if $p\leq 0$,} \\
F^{-\alpha+p}(\nabla) & \text{if $p\geq1$}.
\end{cases}
\]
\end{corollary}

\begin{proof}
That the second arrow is a quasi-isomorphism is the statement of
\cite[Cor.\,1.4]{Yu12}. The first one follows from Proposition \ref{Prop:Kont-log}. Indeed in this case we have
\[
\tau^p=\begin{cases}
F^0(-\alpha) \xleftarrow{\sim} F^0(\nabla) & \text{if $p\leq 0$} \\
F_\alpha^p & \text{if $p\geq 1$}.
\end{cases}\qedhere
\]
\end{proof}

\begin{corollary}\label{cor:omegafdegen1}
For\marginpars{Old Cor.\,1.4.4.} $\alpha\in[0,1)$, the hypercohomology $\bH^k\big(X,(\Omega^\cbbullet_f(\alpha),\nabla)\big)$ does not depend on the choice of the smooth compactification $X$ of $U$ such that $X\setminus U$ has normal crossings.
\end{corollary}

\begin{proof}
Forgetting the filtration in Corollary \ref{Cor:Kont-log}, we have $\bH^k\big(X,(\Omega^\cbbullet_f(\alpha),\nabla)\big)\simeq\bH^k\big(X,F^{\Yu,-\alpha}(\nabla)\big)$ and the assertion follows from \cite[Th.\,1.8]{Yu12}.
\end{proof}

\begin{remark}
The statement of Corollary \ref{cor:omegafdegen1} is also a consequence of Proposition \ref{prop:F+strict} below, through the various identifications that we make in \S\S\ref{subsec:implies}--\ref{subsec:compYuDel}.
\end{remark}

From the properties of the filtration $F^\Yu$ (\cf\S\ref{subsec:FYu}) we obtain, as a consequence of Theorem \ref{th:main}:

\begin{corollary}[of Theorem \ref{th:main}]\label{cor:omegafdegen2}
For\marginpars{Old Cor.\,1.5.6.} $\alpha\in[0,1)$ fixed, and for $p,q\geq0$, we have
$$
\bigoplus_{\lambda\in[-\alpha+p,-\alpha+p+1)}\gr^\lambda_{F^\Yu}H^{p+q}_\dR(U,\nabla)\simeq
\gr_{F_\alpha^\Yu}^pH^{p+q}_\dR(U,\nabla)\simeq H^q(X,\Omega^p_f(\alpha))
$$
and therefore a decomposition $H^k_\dR(U,\nabla)\simeq \bigoplus_{p+q=k}H^q(X,\Omega^p_f(\alpha))$.
\end{corollary}

\begin{proof}
The right-hand term is the $E_1$-term in the spectral sequence attached to $\big(\Omega_f^\cbbullet(\alpha),\nabla,\sigma^p\big)$, hence that of the spectral sequence attached to $F_\alpha^{\Yu,p}(\Omega^\cbbullet(*D),\nabla)$, according to Corollary \ref{Cor:Kont-log}, which is equal to the middle term, according to Theorem \ref{th:main}.
\end{proof}

\begin{remark}
It\marginpars{Old Rem.\,1.5.7.} is a natural question to ask for a geometric interpretation of $H^q(X,\Omega^p_f(\alpha))$. When $\alpha=0$ and $P=P_\red$, such an interpretation is furnished by Theorem \ref{th:C3} in Appendix \ref{app:saito}. On the other hand, when $D=P_\red$, it is natural to expect that the complex $(\Omega_f,\rd+\rd f)$ is quasi-isomorphic to the $L^2$ complex on $X\setminus P_\red$ with the same differential and relative to a complete metric on $X\setminus D$ which is equivalent to the Poincaré metric near each point of $D$. The corresponding Hodge decomposition should be proved as in \cite{Fan11} (we owe this~$L^2$ interpretation to T.\,Mochizuki).
\end{remark}

\subsection{Relation between Theorem \ref{th:main} for $\alpha=0$ and Theorem \ref{th:Kontsevich}}\label{subsec:implies}

\begin{proof}[Proof that Theorem \ref{th:main} for $\alpha=0$ implies Theorem \ref{th:Kontsevich}]

\begin{proposition}\label{prop:nablad}
If\marginpars{Old Cor.\,1.5.5 and Rem.\,1.5.8.} $\dim\bH^k\big(X,(\Omega^\cbbullet_{cf},\rd+\rd (cf))\big)=\dim\bH^k\big(X,(\Omega^\cbbullet_{cf},0)\big)$ for all~$k$ and any $c\in\CC^*$, then the conclusion of Theorem \ref{th:Kontsevich} holds.
\end{proposition}

\begin{proof}
For $u,v\neq0$, we have $\Omega_f^k=\Omega_{vf/u}^k$, and thus
\begin{align*}
\dim\bH^k\big(X,(\Omega_f^\cbbullet,u\rd+v\rd f)\big)
&=\dim\bH^k\big(X,(\Omega_{vf/u}^\cbbullet,\rd+\rd (vf/u))\big)\\
&=\dim\bH^k\big(X,(\Omega_{vf/u}^\cbbullet,0)\big)\quad\text{(assumption for $vf/u$)}\\
&=\dim\bH^k\big(X,(\Omega_f^\cbbullet,0)\big).
\end{align*}

We will now treat the case $v=0$, $u\neq0$, that is, we will prove the equality $\dim\bH^k\big(X,(\Omega_f^\cbbullet,\rd)\big)=\dim\bH^k\big(X,(\Omega_f^\cbbullet,0)\big)$. The case $v\neq0$, $u=0$ is obtained similarly. We will use a standard semi-continuity argument.

On the one hand, $H^k\big(X,(\Omega_f^\cbbullet,\rd)\big)$ being the abutment of the spectral sequence attached to the filtered complex $\sigma^p(\Omega_f^\cbbullet,\rd)$, we have
\begin{equation}\label{eq:ineqk}
\dim\bH^k\big(X,(\Omega_f^\cbbullet,\rd)\big)\leq\dim\bH^k\big(X,(\Omega_f^\cbbullet,0)\big)\quad \forall k,
\end{equation}
the latter term being equal to $\dim\bH^k\big(X,(\Omega_f^\cbbullet,\nabla)\big)$, according to the assumption for~$f$. Let us show the equality by considering the complex $(\Omega_f^\cbbullet[\tau],\rd_X+\nobreak\tau\rd f)$, where~$\tau$ is a new variable and $\rd_X$ differentiates with respect to~$X$ only. Since each~$\Omega^p_f$ is $\cO_X$-coherent (even locally free), each $H^q(X,\Omega^p_f[\tau])=\CC[\tau]\otimes_\CC H^q(X,\Omega^p_f)$ is a free $\CC[\tau]$-module of finite type, and thus each $\bH^k_\tau:=\bH^k\big(X,(\Omega_f^\cbbullet[\tau],\rd_X+\nobreak\tau\rd f)\big)$ is a $\CC[\tau]$-module of finite type, by a spectral sequence argument with respect to the stupid filtration~$\sigma^p$. We claim first:
\begin{itemize}
\item
$\dim_{\CC(\tau)}\CC(\tau)\otimes_{\CC[\tau]}\bH^k_\tau=\dim\bH^k\big(X,(\Omega_f^\cbbullet,\nabla)\big)$.
\end{itemize}
Indeed, since $\Omega_f^\cbbullet[\tau]$ is $\CC[\tau]$-free and since $\bH^k_\tau$ has finite type over $\CC[\tau]$, we have \hbox{$\CC[\tau]/(\tau-v)\otimes_{\CC[\tau]}\bH^k_\tau=\bH^k\big(X,(\Omega_f^\cbbullet,\rd+v\rd f)\big)$} for $v$ general enough. We know that the dimension of the latter space is independent of $v\neq0$ and equal to $\bH^k\big(X,(\Omega_f^\cbbullet,\nabla)\big)$ by the first part of the proof, hence the assertion.

Let us now consider the long exact sequence
\[
\cdots\to\bH^k_\tau\To{\tau}\bH^k_\tau\to\bH^k\big(X,(\Omega_f^\cbbullet,\rd)\big)\to\cdots
\]
We will prove for all $k$:
\begin{enumerate}
\item[$(1)_k$]
``$\tau$ at the level $\geq k$ is injective'',
\item[$(2)_k$]
$\bH^k\big(X,(\Omega_f^\cbbullet,\rd)\big)=\bH^k_\tau/\tau\bH^k_\tau$,
\item[$(3)_k$]
$\dim\bH^k\big(X,(\Omega_f^\cbbullet,\rd)\big)=\dim\bH^k\big(X,(\Omega_f^\cbbullet,\nabla)\big)$,
\end{enumerate}
by showing $(1)_{k+1}\implique(2)_k\implique(3)_k\implique(1)_k$. Note that $(3)_k$ is the desired equality.

The assertion $(1)_{k+1}$ implies that $\bH^k_\tau\to\bH^k\big(X,(\Omega_f^\cbbullet,\rd)\big)$ is onto, and thus $(2)_k$ holds, so
\[
\dim\bH^k\big(X,(\Omega_f^\cbbullet,\rd)\big)\geq\dim_{\CC(\tau)}\CC(\tau)\otimes_{\CC[\tau]}\bH^k_\tau=\dim\bH^k\big(X,(\Omega_f^\cbbullet,\nabla)\big),
\]
where the latter equality holds by the claim above. Hence \eqref{eq:ineqk} implies $(3)_k$ and thus, localizing at $\tau=0$, $\CC[\tau]_{(0)}\otimes_{\CC[\tau]}\bH^k_\tau$ is $\CC[\tau]_{(0)}$-free, so  $(1)_k$ holds.

Since $(1)_k$ holds for $k$ large, it holds for all $k$, as well as $(3)_k$.
\end{proof}

In order to obtain Theorem \ref{th:Kontsevich} from Theorem \ref{th:main} for $\alpha=0$, it remains to apply Corollary \ref{cor:omegafdegen2} with $\alpha=0$ (which uses Theorem \ref{th:main}), together with Corollary \ref{cor:Kontsevich2}, to $cf$ for any $c\in\CC^*$.
\end{proof}

\begin{proof}[Proof that \eqref{eq:Kontsevich} implies Theorem \ref{th:Kontsevich} and Theorem \ref{th:main} for $\alpha=0$]

\begin{proposition}[Kontsevich \cite{Kontsevichl12}]\label{prop:Kontsevich2b}
For all $k$ we have
\[
\dim\bH^k\big(X,(\Omega_f^\cbbullet, \rd)\big)=\dim H^k_\dR(U,\nabla).
\]
\end{proposition}

The proof of Proposition \ref{prop:Kontsevich2b} is postponed to Appendix \ref{app:proof152}. It does not use any Hodge-theoretic argument. If we assume \eqref{eq:Kontsevich} for $f$, it holds for $cf$ for all $c\in\CC^*$ since $\Omega_{cf}^\cbbullet=\Omega_f^\cbbullet$. Then, according to Proposition \ref{prop:Kontsevich2b} and Corollary \ref{cor:Kontsevich2}, we obtain $\dim\bH^k\big(X,(\Omega_{cf}^\cbbullet, \rd+\rd(cf))\big)=\dim\bH^k\big(X,(\Omega_{cf}^\cbbullet, 0)\big)$, which is equivalent to Theorem \ref{th:main} for $\alpha=0$, according to Corollary \ref{Cor:Kont-log}, and on the other hand implies the other statements of Theorem \ref{th:Kontsevich} according to Proposition \ref{prop:nablad}.
\end{proof}

\subsection{Deligne's filtration: the $\cD$-module approach}\label{subsec:Dapproach}
In this subsection, the results will be of a local nature, and we will make use of the local setting of \S\ref{subsec:setup}.

\subsubsection*{First construction}
Let us denote by $E^{1/g}$ the $\cO$-module $\cO(*P_\red)$ with the twisted $\cD$-module structure, so that the corresponding flat connection is~\hbox{$\nabla=\rd+\rd(1/g)$}. We will denote by $\reg$ the generator $1$, in order to make clear the twist of the connection on the $\cO$-module $\cO(*P_\red)$. The behaviour of the connection with respect to the filtration $F_\bbullet\cO(*P_\red)$ (defined in \S\ref{subsec:setup}) is as follows:
\[
\nabla(F_p\cO(*P_\red))\subset\Omega_{x,y,z}^1(\log P_\red)\otimes \big(F_p\cO(*P_\red)\big)(P)\subset\Omega_{x,y,z}^1\otimes \big(F_{p+1}\cO(*P_\red)\big)(P).
\]

For each $\alpha\in[0,1)$ and all $p\in\NN$, we consider the increasing filtration by coherent $\cO$-submodules indexed by $\NN$ defined as
\begin{equation}\label{eq:FEfprime}
F_{\alpha+p}(E^{1/g}):=F_p\cO(*P_\red)([(\alpha+p)P])\otimes\reg,
\end{equation}
where $[(\alpha+p)P]$ is the integral part of $(\alpha+p)P$, that is, locally defined by $x^{[\bme(\alpha+p)]}=x^{[\bme\alpha]}g^{p}$, and we set $F_{\alpha+p}(E^{1/g})=0$ for $p\in\ZZ_{<0}$. We therefore get a filtration $F_\mu(E^{1/g})$ indexed by $\RR$, with $F_\mu(E^{1/g})=0$ for $\mu<0$ and jumps for $\mu\in\QQ_{\geq0}$ at most.

We will mainly work with $E^{1/g}(*H)$, which is equal to $\cO(*D)$ as an $\cO$-module. Its filtration is defined, in a way analogous to \eqref{eq:FpPH}, by
\begin{equation}\label{eq:FEfprimeT}
F_{\alpha+p}(E^{1/g}(*H)):=\sum_{q+q'\leq p}F_{q}\cO(*H)\cdot F_{\alpha+q'}(E^{1/g}),
\end{equation}
where the product is taken in $\cO(*D)$.

In both cases, these filtrations satisfy the Griffiths transversality property with respect to the connection $\nabla$ on $E^{1/g}$ or on $E^{1/g}(*H)$, that is, they are $F$-filtrations with respect to the standard order filtration on $\cD$.

\subsubsection*{Second construction, step one: adjunction of a variable}
Denote by $M'$ the $\cD$\nobreakdash-module-theoretic push-forward of $(\cO(*H),\rd)$ by the graph inclusion of $g$. If~$t'$ denotes the new coordinate produced by this inclusion, $M'$ is a left $\cD[t']\langle\partial_{t'}\rangle$-module. Let us make it explicit. We consider a new variable $\tau'$, and we have by definition
\[
M'=\CC\{x,y,z\}[y^{-1},\tau']\quad\text{as a $\CC\{x,y,z\}$-module}.
\]
It will be convenient to denote by $\delta$ the element $1/\prod_{j=1}^m y_j$ of $M'$. The remaining part of the left action of $\cD[t']\langle\partial_{t'}\rangle$ is defined as follows on $\delta$ (and extended to $M'$ by using Leibniz rule):
\begin{equation}\label{eq:action}
\left\{\begin{aligned}
\partial_{t'}\delta&=\tau'\delta,\\
t'\delta&=g\delta=x^{\bme}\delta,\\
\partial_{x_i}\delta&=-\frac{\partial g}{\partial x_i}\,\tau'\delta=-e_ix^{\bme-1_i}\tau'\delta,\\
\partial_{y_j}\delta&=-\frac{1}{y_j}\,\delta,\\
\partial_{z_k}\delta&=0.
\end{aligned}\right.
\end{equation}
We note that $(\cO(*H),\rd)$, as a left $\cD$-module, is recovered as the cokernel of the injective morphism of left $\cD$-modules $\partial_{t'}:M'\to M'$ with the induced $\cD$-module structure.

We denote by $E^{1/t'}$ the left $\cD[t']\langle\partial_{t'}\rangle$-module $\CC\{x,y,z,t'\}[\tpm]\retp$ whose generator $\retp$ satisfies $\partial_{t'}\retp=-(1/t^{\prime2})\retp$. The twisted $\cD[t']\langle\partial_{t'}\rangle$-module $M'\otimes E^{1/t'}$ is the left $\cD[t']\langle\partial_{t'}\rangle$-module
\[
M'\otimes E^{1/t'}=\CC\{x,y,z\}[x^{-1},y^{-1},\tau'](\delta\otimes\retp)\quad\text{as a $\CC\{x,y,z\}$-module}.
\]
We have used here that, with respect to the $\CC[t']$-action on $M'$ defined above, we have $\CC[t',\tpm]\otimes_{\CC[t']}M'=M'[1/g]=\CC\{x,y,z\}[x^{-1},y^{-1},\tau']$. The remaining part of the left action of $\cD[t']\langle\partial_{t'}\rangle$ is defined as follows on the generator $\delta\otimes\retp$ (and extended to $M'\otimes E^{1/t'}$ by using Leibniz rule):
\begin{equation}\label{eq:twaction}
\left\{\begin{aligned}
\partial_{t'}(\delta\otimes\retp)&=(\tau'-g^{-2})(\delta\otimes\retp)=(\tau'-x^{-2\bme})(\delta\otimes\retp),\\
t'(\delta\otimes\retp)&=g(\delta\otimes\retp)=x^{\bme}(\delta\otimes\retp),\\
\partial_{x_i}(\delta\otimes\retp)&=-\frac{\partial g}{\partial x_i}\,\tau'(\delta\otimes\retp)=-e_ix^{\bme-1_i}\tau'(\delta\otimes\retp),\\
\partial_{y_j}(\delta\otimes\retp)&=-\frac{1}{y_j}\,(\delta\otimes\retp),\\
\partial_{z_k}(\delta\otimes\retp)&=0.
\end{aligned}\right.
\end{equation}
Due to the previous formulas, the decomposition
\[
M'\otimes E^{1/t'}=\bigoplus_{k\geq0}\cO(*D)\tau^{\prime k}(\delta\otimes\retp)
\]
can be transformed to a decomposition
\[
M'\otimes E^{1/t'}=\bigoplus_{k\geq0}\cO(*D)\partial_{t'}^k(\delta\otimes\retp),
\]
which shows that $M'\otimes E^{1/t'}$ is a free $\cO(*D)[\partial_{t'}]$-module of rank one with generator $\delta\otimes\retp$. 

Let $\iota$ denote the inclusion associated to the graph of $g$. The $\cD[t']\langle\partial_{t'}\rangle$-module $\iota_+ E^{1/g}(*H):=\bigoplus_kE^{1/g}(*H)\otimes\partial_{t'}^k\delta$ is also a free $\cO(*D)[\partial_{t'}]$-module of rank one with generator $(1\otimes\reg\otimes\delta)$. 

The unique $\cO(*D)[\partial_{t'}]$-linear isomorphism $\iota_+ E^{1/g}(*H)\isom M'\otimes E^{1/t'}$ sending $(1\otimes\reg\otimes\delta)$ to $\delta\otimes\retp$ is in fact $\cD[t']\langle\partial_{t'}\rangle$-linear. Let us check for instance that it is $\partial_{x_i}$-linear:
\begin{align*}
\partial_{x_i}(1\otimes\reg\otimes\delta)&=-\frac{\partial g/\partial x_i}{g^2} \otimes\reg\otimes\delta-(\partial g/\partial x_i) \otimes\reg\otimes\partial_{t'}\delta\\
&\mto -\frac{\partial g/\partial x_i}{g^2}\delta\otimes\retp-\partial_{t'}[(\partial g/\partial x_i)\delta \otimes\retp]\\
&=-\tau'(\partial g/\partial x_i)\delta \otimes\retp=\partial_{x_i}(\delta\otimes\retp),
\end{align*}
according to \eqref{eq:twaction}.

It is then clear that, on the other hand, one recovers $E^{1/g}(*H)$ from $M'\otimes E^{1/t'}$ as its push-forward by the projection $\pi$ along the $t'$ variable. So we find
\begin{align}
M'\otimes E^{1/t'}&\simeq\iota_+ E^{1/g}(*H),\label{eq:iotaE}\\
E^{1/g}(*H)&=\pi_+(M'\otimes E^{1/t'})=\coker\big[\partial_{t'}:M'\otimes E^{1/t'}\to M'\otimes E^{1/t'}\big].\label{eq:cokerdtp}
\end{align}

\subsubsection*{Second construction, step two: the Deligne filtration}
For the sake of simplicity, the filtrations will be taken increasing. One can consider them as decreasing by changing the sign of the indices.

Let $F_\bbullet M'$ be the filtration $F_\bbullet\cD[t']\langle\partial_{t'}\rangle\cdot \delta$ on $M'$. It is the filtration by $\deg_{\tau'}+\ord_H$, where $\ord_H$ is the order of the pole along $H$ such that $\ord_H\delta=0$. Let $V_\bbullet M'$ be the Kashiwara-Malgrange $V$-filtration with respect to the function~$t'$ (see \eg \cite[\S3.1]{MSaito86} or \S\ref{subsec:strictspe} below). We will only consider the steps~$V_\alpha M'$ for~$\alpha\in\nobreak[0,1)$ (the jumps possibly occur at most at $\alpha\in[0,1)\cap\QQ$). The normalization condition is that $t'\partial_{t'}+\alpha$ is nilpotent when induced on $\gr_\alpha^VM'$.

The natural generalization of the $\cD$-module-theoretic Deligne filtration defined in \cite[\S6.b]{Bibi08} is (in the increasing setting), for each $\alpha\in[0,1)$, and any $p\in\ZZ$:
\begin{equation}\label{eq:Fdelalpha}
F_{\alpha+p}^\Del(M'\otimes E^{1/t'})=\sum_{k=0}^p\partial_{t'}^k\tpm \big((F_{p-k}M'\cap V_\alpha M')\otimes\retp\big).
\end{equation}
We have $F_{\alpha+p}^\Del(M'\otimes\nobreak E^{1/t'})=0$ if $\alpha+p<0$. Note also that
\[
F_\alpha^\Del(M'\otimes E^{1/t'})=\tpm(F_0 M'\cap V_\alpha(M')\otimes\retp).
\]
For each $\alpha\in[0,1)$, it is easily checked that the filtration $F_{\alpha+p}^\Del(M'\otimes E^{1/t'})$ ($p\in\ZZ$) is a $F$-filtration for $\cD[t']\langle\partial_{t'}\rangle$, \ie satisfies
\[
F_q\cD[t']\langle\partial_{t'}\rangle\cdot F_{\alpha+p}^\Del(M'\otimes E^{1/t'})\subset F_{\alpha+p+q}^\Del(M'\otimes E^{1/t'}),
\]
where $F_q\cD[t']\langle\partial_{t'}\rangle$ consists of operators of total order $\leq q$ (w.r.t.\,$\partial_x,\partial_y,\partial_z,\partial_{t'}$). The jumps possibly occur at most at $\alpha+p\in\QQ_{\geq0}$ and we have
\begin{equation}\label{eq:FdF}
F_{\alpha+p}^\Del(M'\otimes E^{1/t'})=\partial_{t'}F_{\alpha+p-1}^\Del(M'\otimes E^{1/t'})+\tpm(F_pM'\cap V_\alpha M')\otimes\retp.
\end{equation}
Indeed by definition,
\begin{align*}
\partial_{t'}F_{\alpha+p-1}^\Del(M'\otimes E^{1/t'})&=\sum_{k\geq0}\partial_{t'}^{k+1}\tpm\big((F_{p-1-k}M'\cap V_\alpha M')\otimes\retp\big)\\
&=\sum_{k\geq1}\partial_{t'}^k\tpm\big((F_{p-k}M'\cap V_\alpha M')\otimes\retp\big).
\end{align*}

\begin{proposition}\label{prop:strictshift}
The filtration $F^\Del_\bbullet(M'\otimes E^{1/t'})$ is exhaustive and the injective $\cD$-linear morphism $\partial_{t'}:M'\otimes E^{1/t'}\to M'\otimes E^{1/t'}$ strictly shifts the Deligne filtration~\eqref{eq:Fdelalpha} by one, that is,
\[
\forall\alpha\in[0,1),\;\forall p\in\ZZ,\quad F^\Del_{\alpha+p+1}(M'\otimes E^{1/t'})\cap\partial_{t'}(M'\otimes E^{1/t'})=\partial_{t'}F^\Del_{\alpha+p}(M'\otimes E^{1/t'}).
\]
\end{proposition}

\begin{proof}
For the first point, let us denote by $V_\bbullet M'[\tpm]$ the Kashiwara-Malgrange filtration of $M'[\tpm]$ (without twist) with respect to the function $t'$. For $\mu<1$ we have $V_\mu M'[\tpm]=V_\mu M'$, while for $\mu\geq1$ we have $V_\mu M'[\tpm]=t^{\prime -[\mu]}V_{\mu-[\mu]}M'$. For each $\mu$, $\partial_{t'}-t^{\prime -2}$ sends $V_\mu M'[\tpm]$ to $V_{\mu+2} M'[\tpm]$ and the graded morphism is also that induced by~$-t^{\prime -2}$, hence is an isomorphism. It follows that any $m\in M'[\tpm]$ can be written as a finite sum $\sum_k(\partial_{t'}-t^{\prime -2})^k\tpm m_k$ with $m_k\in V_{\alpha_k}M'$, and $\alpha_k\in[0,1)$. Set $\alpha=\max_k\alpha_k$ and replace each $\alpha_k$ with $\alpha$. Since the filtration $F_\bbullet M'$ is exhaustive, there exists~$p$ such that $m_k\in F_{p-k}M'\cap V_\alpha M'$ for each $k$. Therefore, $m\otimes\retp$ belongs to $F_{\alpha+p}^\Del(M'\otimes E^{1/t'})$.

For the strictness assertion, according to \eqref{eq:FdF} and forgetting the $E^{1/t'}$ factor, it is enough to prove that, for all $\alpha\in[0,1)$ and all $p\in\ZZ$,
\[
\tpm(F_{p+1}M'\cap V_\alpha M')\cap(\partial_{t'}-t^{\prime-2})(M'[\tpm])\subset(\partial_{t'}-t^{\prime-2})\tpm(F_pM'\cap V_\alpha M'),
\]
or, by using the standard commutation rule, that
\[
(F_{p+1}M'\cap V_\alpha M')\cap(\partial_{t'}-\tpm-t^{\prime-2})(M'[\tpm])\subset(\partial_{t'}-\tpm-t^{\prime-2})(F_pM'\cap V_\alpha M').
\]
We will check separately that
\begin{itemize}
\item
$m'\in V_\alpha M'$ and $m'=(\partial_{t'}-\tpm-t^{\prime-2})m$ implies $m\in V_{\alpha-2}(M')\subset V_\alpha M'$,
\item
$m'\in F_{p+1}M'$ and $m'=(\partial_{t'}-\tpm-t^{\prime-2})m$ implies $m\in F_p M'$.
\end{itemize}

On the one hand, the operator $\partial_{t'}-\tpm-t^{\prime-2}$ induces for each $\mu$ an isomorphism $\gr_{\mu-2}^V(M'[\tpm])\isom\gr_\mu^V(M'[\tpm])$, so the first assertion is clear.

On the other hand, using the identification $M'[\tpm]=\CC\{x,y,z\}[x^{-1},y^{-1},\tau']1/y$ as a left $\CC\{x,y,z\}$-module, with the $F$-filtration induced by ``degree in $\tau'$ $+$ pole order in $y$'', the operator $\partial_{t'}-\tpm-t^{\prime-2}$ sends a term $\varphi_k(x,y,z)\tau^{\prime k}$ ($k\geq0$) to
\[
\varphi_k(x,y,z)\cdot\Big(\tau^{\prime(k+1)}+\sum_{j\leq k}\big(\textstyle\sum_{\ell\geq0}c_{j,\ell}x^{-\ell\bme}\big)\tau^{\prime j}\Big)
\]
for some coefficients $c_{j,\ell}\in\NN$ (due to the commutation rule between $\tpm$ and $\partial_{t'}$). If $m=\sum_{k=0}^q\varphi_k\tau^{\prime k}\in M'[\tpm]$ is such that $m':=(\partial_{t'}-\tpm-t^{\prime-2})m=\sum_{k=0}^{q+1}\varphi'_k\tau^{\prime k}$ belongs to $F_{p+1}M'$, then $\varphi_q=\varphi'_{q+1}$ belongs to $\CC\{x,y,z\}[y^{-1}]$ and its pole order relative to $y$ is $\leq(p+1)-(q+1)=p-q$, so $\varphi_q\tau^{\prime q}\in F_pM'$. By decreasing induction on $q$, one concludes that $m\in F_pM'$.
\end{proof}

\begin{definition}[of $F_\bbullet^\Del(E^{1/g}(*H))$]\label{def:FDelE}
The Deligne filtration $F_\bbullet^\Del(E^{1/g}(*H))$ (indexed by $\RR$) on $E^{1/g}(*H)$ is the image filtration of $F_\bbullet^\Del(M'\otimes E^{1/t'})$ by \eqref{eq:cokerdtp}.
\end{definition}

Here are some properties of $F_\mu^\Del(E^{1/g}(*H))$:
\begin{itemize}
\item
We have $F_\mu^\Del(E^{1/g}(*H))=0$ for $\mu<0$ and the jumps possibly occur at most at $\mu\in\QQ_{\geq0}$.
\item
For a fixed $\alpha\in[0,1)$ and $p\in\NN$, $F_{\alpha+p}^\Del(E^{1/g}(*H))$ is a $F_\bbullet\cD$-filtration, \ie
\[
F_q\cD\cdot F_{\alpha+p}^\Del(E^{1/g}(*H))\subset F_{\alpha+p+q}^\Del(E^{1/g}(*H)).
\]
\item
$F_{\alpha+p}^\Del(E^{1/g}(*H))=\mathrm{image}\big(\tpm (F_pM'\cap V_\alpha M')\otimes\retp\big)$ by \eqref{eq:cokerdtp}. Indeed, this directly follows from \eqref{eq:FdF}.
\end{itemize}

On the other hand, the push-forward $\iota_+E^{1/g}(*H)$ comes naturally equipped with a push-forward filtration
\begin{equation}\label{eq:iotaF}
F_{\alpha+p}\big(\iota_+E^{1/g}(*H)\big):=\bigoplus_{k\geq0}F_{\alpha+p-k}(E^{1/g}(*H))\otimes\partial_{t'}^k\delta.
\end{equation}
with $F_\bbullet(E^{1/g}(*H))$ defined by \eqref{eq:FEfprimeT}. This defines a filtration $F_\bbullet(M'\otimes E^{1/t'})$ according to \eqref{eq:iotaE}.

\subsubsection*{Comparison of both filtrations on $M'\otimes E^{1/t'}$ and $E^{1/g}(*H)$}

\begin{proposition}\label{prop:comparisonFE}
For each $\alpha\in[0,1)$ and each $p\in\ZZ$, we have
\begin{align*}
F_{\alpha+p}(M'\otimes E^{1/t'})&=F_{\alpha+p}^\Del(M'\otimes E^{1/t'})\quad\text{\cf \eqref{eq:iotaF} and \eqref{eq:Fdelalpha}},\\
F_{\alpha+p}(E^{1/g}(*H))&=F_{\alpha+p}^\Del(E^{1/g}(*H)) \quad\text{\cf \eqref{eq:FEfprimeT} and Def.\,\ref{def:FDelE}}.
\end{align*}
\end{proposition}

\begin{proof}
We will prove the first equality, since the second one obviously follows. We need here an explicit expression of $F_pM'\cap V_\alpha M'$. We will recall the computation already made in \cite{Bibi87} for the pure case ($H=\emptyset$) and recalled and generalized to the mixed case in \cite[Prop.\,4.19]{Bibi96a}, where the notation $\delta'$ is used for the present notation $\delta$. It would also be possible to use \cite[Prop.\,3.5]{MSaito87}, but the computation is written there for right $\cD$-modules, so one should first express this computation for left modules. Note also that it is enough to consider $p\in\NN$, since both filtrations are identically zero if $p\leq-1$.

We will use the notation
\begin{align*}
\bmb&=(b_1,\dots,b_\ell)\in\ZZ^\ell,&|\bmb|_+&=\sum_i\max\{0,b_i\},\\ \bmc&=(c_1,\dots,c_m)\in\ZZ^m,&|\bmc|_+&=\sum_j\max\{0,c_j\}.
\end{align*}
Then, by \eqref{eq:FEfprimeT}, we have
\begin{equation}\label{eq:FEFalphap}
F_{\alpha+p}(E^{1/g}(*H))=\sum_{|\bmb|_++|\bmc|_+\leq p} y^{-\bmc-\bf1}x^{-\bmb-\bf1}x^{-[\bme\alpha]}g^{-p+|\bmc|_+}\cdot\cO.
\end{equation}

On the other hand, for $\bma=(a_1,\dots,a_\ell)\in\ZZ^\ell$, let us set
\[
P_{\alpha,\bma}(s)=c_{\bma,\alpha}\prod_{i=1}^\ell\prod_{k=[e_i\alpha]+1}^{a_i}(s+k/e_i),\qquad \text{with }c_{\bma,\alpha}\in\CC\text{ such that }P_{\alpha,\bma}(-\alpha)=1,
\]
taking into account the convention that a product indexed by the empty set is equal to $1$. Let us also set (recall that $\lceil\beta\rceil:=-[-\beta]$)
\[
I_\alpha(\bma)=\{i\mid a_i=\lceil e_i\alpha\rceil\}\subset\{1,\dots,r\},\quad J(\bmc)=\{j\mid c_j=0\}.
\]
Then, by embedding $M'$ in $M'[x^{-1}]$, an element of $V_\alpha M'$ can be written in a unique way as the result of the action of some polynomials in $t'\partial_{t'}$ on some elements of $M'[x^{-1}]$ as follows:
\[
\sum_{\bma\geq\lceil\bme\alpha\rceil}\sum_{\bmc\geq0}\sum_{\ell\geq0}(t'\partial_{t'}+\alpha)^\ell P_{\alpha,\bma-{\bf1}}(t'\partial_{t'})h_{\bma,\bmc,\ell}(x_{I_\alpha(\bma)},y_{J(\bmc)},z)y^{-\bmc}x^{-\bma}t'\delta,
\]
with $x_{I_\alpha(\bma)}=(x_i)_{i\in I_\alpha(\bma)}$, $y_{J(\bmc)}=(y_j)_{j\in J(\bmc)}$, $h_{\bma,\bmc,\ell}\in\cO$ only depends on the indicated variables.\footnote{A condition on $h_{\bma,\bmc,\ell}$ is mistakenly added in \loccit; it is irrelevant.} The filtration $F_\bbullet M'\cap V_\alpha M'$ is the filtration by the degree in $t'\partial_{t'}$ plus the pole order in $y$. In other words, an element of $V_\alpha M'$ written as above belongs to $F_pM'\cap V_\alpha M'$ if and only if
\[
h_{\bma,\bmc,\ell}\not\equiv0\Longrightarrow \ell+|\bmc|_++\deg P_{\alpha,\bma-{\bf1}}\leq p,
\]
that is, if we set $b_i=a_i-1-[e_i\alpha]$,
\[
\ell+|\bmc|_++|\bmb|_+\leq p.
\]
Note also that the condition $a_i\geq\lceil e_i\alpha\rceil$ implies $a_i\geq[e_i\alpha]\geq0$. By using the standard commutation relations, an element of $\tpm(F_pM'\cap\nobreak V_\alpha M')$ can thus be written in a unique way as
\begin{equation}\label{eq:FpV}
\sum_{\bma\geq\lceil\bme\alpha\rceil}\sum_{\bmc\geq0}\sum_{\ell\geq0}(\partial_{t'}t'+\alpha)^\ell P_{\alpha,\bma-{\bf1}}(\partial_{t'}t')h_{\bma,\bmc,\ell}(x_{I_\alpha(\bma)},y_{J(\bmc)},z)y^{-\bmc}x^{-\bma}\delta,
\end{equation}
with the same conditions on $\bma,\bmc$ and $h_{\bma,\bmc,\ell}$.

We will use the following identity in $\cD_X[t',t^{\prime-1}]\langle\partial_{t'}\rangle$:
\begin{multline}\label{eq:partialtprime}
(\partial_{t'}t')^k=a_0^{(k)}(t')t^{\prime-k}+\cdots+(\partial_{t'}-t^{\prime-2})^jt^{\prime j-k}a_j^{(k)}(t')\\
+\cdots+(\partial_{t'}-t^{\prime-2})^ka_k^{(k)}(t'),
\end{multline}
for some polynomials $a_j^{(k)}(t')\in\CC[t']$, with $a_0^{(k)}(0)=1$. This identity can be checked easily. Then a similar identity, with coefficients still denoted by $a_j^{(k)}(t')$, holds for any polynomial of degree $k$ in $\partial_{t'}t'$, and moreover $a_0^{(k)}(0)\neq\nobreak0$.

Let us first prove the inclusion $\tpm(F_pM'\cap\nobreak V_\alpha M')\otimes\retp\subset F_{\alpha+p}(M'\otimes E^{1/t'})$. According to \eqref{eq:partialtprime} we have
\[
\big[(\partial_{t'}t')^ky^{-\bmc}x^{-\bma}\delta\big]\otimes\retp=\sum_ja_j^{(k)}(g)g^{j-k}y^{-\bmc}x^{-\bma}\otimes\reg \otimes\partial_{t'}^j\delta.
\]
If such a term occurs in \eqref{eq:FpV}, we have $k+|\bmc|_+\leq p$, hence \hbox{$j-k\geq -(p-j)+|\bmc|_+$} and, recalling that $\delta=y^{-\bf1}$, we conclude that the $j$-th coefficient belongs to $F_{\alpha+p-j}E^{1/g}(*H)$, after \eqref{eq:FEFalphap}, hence the desired inclusion, according to \eqref{eq:iotaF}.

Conversely, let us prove that
\[
F_{\alpha+p}E^{1/g}(*H)\otimes\delta\subset \tpm(F_pM'\cap\nobreak V_\alpha M')\otimes\retp+\partial_{t'}F_{\alpha+p-1}(M'\otimes E^{1/t'}).
\]
Set $\bmb=\bma-{\bf1}-[\bme\alpha]$ as above. Then $F_{\alpha+p}E^{1/g}(*H)\otimes\delta$ is generated by elements of the form $m=y^{-\bmc}x^{-\bma}g^{-p+|\bmc|_+}\reg\otimes\delta$ with $|\bmb|_++|\bmc|_+\leq p$. Setting $\ell=p-(|\bmb|_++|\bmc|_+)$, Formula \eqref{eq:partialtprime} applied to the polynomial $(\partial_{t'}t'+\nobreak\alpha)^\ell P_{\alpha,\bma-{\bf1}}(\partial_{t'}t')$ of degree $p-|\bmc|_+$ gives
\[
m\otimes\retp=c\big[(\partial_{t'}t'+\alpha)^\ell P_{\alpha,\bma-{\bf1}}(\partial_{t'}t')y^{-\bmc}x^{-\bma}\delta\big]\otimes\retp\bmod\partial_{t'}F_{\alpha+p-1}(M'\otimes E^{1/t'}),
\]
for some nonzero constant $c$.
\end{proof}

\subsection{Comparison of Yu's filtration and Deligne's filtration on the twisted de~Rham complex}\label{subsec:compYuDel}
We will introduce three definitions \eqref{eq:FDRYu}, \eqref{eq:FDelDRE} and \eqref{eq:FDelDRME} of a filtered twisted meromorphic de~Rham complex. Corollary \ref{cor:comparison} will show that they give filtered quasi-isomorphic complexes, by using Propositions \ref{prop:YuDel} and \ref{prop:strictshift}.

\subsubsection*{The filtered twisted logarithmic de~Rham complex \protect\cite{Yu12}}
Let $(\Omega^\cbbullet_{x,y,z}(\log D),\rd)$ be the logarithmic de~Rham complex (logarithmic with respect to $D$), so that in particular $\Omega^0_{x,y,z}(\log D)=\cO$. We set, for any $\alpha\in[0,1)$ and $p\in\ZZ$,
\begin{multline}\label{eq:FDRYu}
F_{\alpha+p}^\Yu\DR(E^{1/g}(*H))\\
:=\{0\to\Omega^0_{x,y,z}(\log D)([(\alpha+p)P])_+\To{\nabla}\Omega^1_{x,y,z}(\log D)([(\alpha+p+1)P])_+\To{\nabla}\cdots\}
\end{multline}
(this is the increasing version of \eqref{eq:FYu}.)

\subsubsection*{The filtered twisted meromorphic de~Rham complex}
Let us consider the usual twisted de Rham complex
\[
\{0\ra E^{1/g}(*H)\To{\nabla}\Omega^1_{x,y,z}\otimes E^{1/g}(*H)\To{\nabla}\cdots\To{\nabla}\Omega^n_{x,y,z}\otimes E^{1/g}(*H)\ra0\}.
\]
The filtration naturally induced by $F_\mu^\Del(E^{1/g}(*H))$ (as defined by \eqref{eq:FEfprimeT} or equivalently by Definition \ref{def:FDelE}, according to Proposition \ref{prop:comparisonFE}) is by definition
\begin{multline}\label{eq:FDelDRE}
F_\mu^\Del(\DR(E^{1/g}(*H)))\\
:=\big\{0\ra F_\mu^\Del E^{1/g}(*H)\ra\Omega^1_{x,y,z}\otimes F_{\mu+1}^\Del E^{1/g}(*H)\ra\cdots\\
{}\ra\Omega^n_{x,y,z}\otimes F_{\mu+n}^\Del E^{1/g}(*H)\ra0\big\}.
\end{multline}

\subsubsection*{The filtered twisted meromorphic de~Rham complex with a variable added}
We define the filtration $F_\mu^\Del\DR(M'\otimes E^{1/t'})$ on the twisted de~Rham complex $\DR(M'\otimes E^{1/t'})$ by a formula analogous to \eqref{eq:FDelDRE}, by using basically $F_\mu^\Del(M'\otimes E^{1/t'})$ as defined by \eqref{eq:Fdelalpha}:
\begin{multline}\label{eq:FDelDRME}
F_\mu^\Del(\DR(M'\otimes E^{1/t'}))\\
:=\{0\ra F_\mu^\Del(M'\otimes E^{1/t'})\ra\Omega^1_{x,y,z,t'}\otimes F_{\mu+1}^\Del(M'\otimes E^{1/t'})\ra\cdots\}.
\end{multline}

\subsubsection*{Comparison of the filtered complexes}

Notice first that, for all three complexes, we have $F_\mu\DR=0$ for $\mu<-n$ ($n={}$dimension of the underlying space), so that in the decreasing setting, $F^\lambda\DR=0$ for $\lambda>n$.

\begin{proposition}\label{prop:YuDel}
For each $\alpha\in[0,1)$ and each $p_o\in\ZZ$, the natural inclusion of complexes $F_{\alpha+p_o}^\Yu\DR(E^{1/g}(*H))\hto F_{\alpha+p_o}^\Del\DR(E^{1/g}(*H))$ is a quasi-isomorphism.
\end{proposition}

\begin{proof}[Sketch of proof]
The question is local, and we can use the local setting of \S\ref{subsec:setup}. According to Proposition \ref{prop:comparisonFE}, we can use the expression \eqref{eq:FEfprimeT} to compute $F_{\alpha+\bbullet}^\Del\DR(E^{1/g}(*H))$, that we will then simply denote by $F_{\alpha+\bbullet}\DR(E^{1/g}(*H))$. By expressing both filtered complexes as the external tensor product of complexes with respect to the variables $x$ on the one hand and $y,z$ on the other hand, we are reduced to consider both cases separately. Moreover, the $y,z$-case is that considered by Deligne \cite[Prop.\,II.3.13]{Deligne70}, so we will focus on the $x$-case, assuming that there is no $y,z$ variables. We are therefore led to proving that the following morphism of complexes is a quasi-isomorphism:
\[
\xymatrix{
\Omega^0_x(\log P_\red)([(\alpha+p_o)P])_+\ar[r]^-\nabla\ar[d]&\cdots\Omega^k_x(\log P_\red)([(\alpha+p_o+k)P])_+\cdots\ar[d]\\
(F_{p_o}\cO_x(*P_\red))\Omega^0_x([(\alpha+p_o)P])\ar[r]^-\nabla&\cdots (F_{p_o+k}\cO_x(*P_\red))\Omega_x^k([(\alpha+p_o+k)P])\cdots
}
\]
By multiplying the $k$th degree term of each complex by $x^{[(\alpha+p_o+k)\bme]}$, the differential~$\nabla$ from the $k$th to the $(k+1)$st degree is replaced by $\delta_k(p_o)=x^{\bme}\rd+\rd\log x^{-\bme}+x^{\bme}\rd\log x^{[(\alpha+p_o+k)e]}$, and we are reduced to showing the quasi-isomorphism when~$p=p_o$:
\[
\xymatrix{
\displaystyle\sigma^{\geq-p}\Big(\Omega^0_x(\log P_\red)\To{\delta_0(p_o)}\cdots\Omega^k_x(\log P_\red)\To{\delta_k(p_o)}\cdots\Big)\ar[d]^\wr\\
\displaystyle F_p\cO_x(*P_\red)\To{\delta_0(p_o)}\cdots F_{p+k}\Omega^k_x(*P_\red)\To{\delta_k(p_o)}\cdots
}
\]
where we now use the standard (increasing) pole order filtration on $\Omega^\cbbullet_x(*P_\red)$, and~$\sigma^{\geq\cbbullet}$ denotes the stupid filtration. We will show the quasi-isomorphism for all~$p$, and for that purpose it will be enough to show that the graded complexes are quasi-isomorphic. For the upper complex, the graded differential is zero, while for the lower complex, it is equal to $\rd\log x^{-\bme}$. We can then argue as in the proof of \cite[Prop.\,II.3.13]{Deligne70} (second reduction) to reduce to the case $\ell=1$, where the graded quasi-isomorphism is easy to check.
\end{proof}

\begin{remark}
Moreover (\cf\cite[Cor.\,1.4]{Yu12}), one can consider the sub-complex $F_0^\Yu\DR(E^{1/g}(*H))$ with the induced filtration. Then the natural inclusion
\[
F_{\min(\mu,0)}^\Yu\DR(E^{1/g}(*H))\hto F_\mu^\Yu\DR(E^{1/g}(*H)).
\]
is a quasi-isomorphism for each $\mu\in\QQ$. This reduces to considering $F_\mu^\Yu\DR(E^{1/g}(*H))$ with $\mu\leq0$, that is, $F^\lambda(\nabla)$ with $\lambda\geq0$.
\end{remark}

The identification of the Koszul complex $K^\cbbullet(E^{1/g}(*H),\partial_x,\partial_y,\partial_z)$ as the cokernel of the termwise injective morphism
\[
\partial_{t'}:K^\cbbullet(M'\otimes E^{1/t'},\partial_x,\partial_y,\partial_z)\to K^\cbbullet(M'\otimes E^{1/t'},\partial_x,\partial_y,\partial_z)
\]
gives a quasi-isomorphism
\begin{multline*}
\DR(M'\otimes E^{1/t'})\simeq K^\cbbullet(M'\otimes E^{1/t'},\partial_x,\partial_y,\partial_z,\partial_{t'})\\
\isom K^\cbbullet(E^{1/g}(*H),\partial_x,\partial_y,\partial_z)[-1]\simeq \DR(E^{1/g}(*H))[-1].
\end{multline*}

Given a filtered complex $F_\bbullet C^\bbullet$, we denote by $(F_\bbullet C^\cbbullet)[k]$ the filtered complex $F_{\bbullet-k}C^{\cbbullet+k}$. From Propositions \ref{prop:YuDel} and \ref{prop:strictshift} we get:

\begin{corollary}\label{cor:comparison}
We have
$$
F_\bbullet^\Yu\DR(E^{1/g}(*H))\simeq F_\bbullet^\Del\DR(E^{1/g}(*H))\simeq \big(F_\bbullet^\Del\DR(M'\otimes E^{1/t'})\big)[1].\eqno\qed
$$
\end{corollary}

\section{Preliminaries to a generalization of Deligne's~filtration}\label{sec:prelimDel}

\subsection{Filtered $\cD$-modules and $R_F\cD$-modules}\label{subsec:filtDRFD}

\subsubsection*{Filtered $\cD$-modules}\label{subsec:filtDmodRmod}
Let $Y$ be a complex manifold and let $F_\bbullet\cD_Y$ denote the filtration of $\cD_Y$ by the order of differential operators. Let $(\ccM,F_\bbullet\ccM)$ be a filtered holonomic $\cD_Y$-module, that is, a holonomic $\cD_Y$-module $\ccM$ equipped with a \emph{good} filtration, \ie $F_k\cD_YF_p\ccM\subset F_{k+p}\ccM$ with equality for $p$ sufficiently large (locally on~$Y$) and any $k$. Let $\hb$ denote a new variable and let $R_F\cD_Y:=\nobreak\bigoplus_{p\in\NN}F_p\cD_Y\hb^p$ denote the Rees ring of the filtered ring $(\cD_Y,F_\bbullet\cD_Y)$. This is a sheaf of $\cO_Y[\hb]$-algebras generated by $\hb\Theta_Y$. In any coordinate chart, the coordinate vector fields~$\hb\partial_y$ will be denoted by $\partiall_y$.

Given a $\cD_Y$-module $\ccM$ equipped with a $F_\bbullet\cD_Y$-filtration $F_\bbullet\ccM$, the Rees module $R_F\ccM:=\bigoplus_pF_p\ccM\cdot\hb^p$ is a graded\marginpars{Added.} $R_F\cD_Y$-module. The filtration $F_\bbullet\ccM$ is \emph{good} if and only if $R_F\ccM$ is $R_F\cD_Y$-coherent. Conversely, given a coherent $R_F\cD_Y$-module, it is of the form $R_F\ccM$ for some coherent $\cD_Y$-module equipped with a good filtration $(\ccM,F_\bbullet\ccM)$ if and only if it is a graded $R_F\cD_Y$-module and it has no $\CC[\hb]$\nobreakdash-torsion (the latter property is called \emph{strictness}). If we regard a $\cD_Y$-module as a $\cO_Y$-module equipped with an integrable connection $\nabla$, we can regard a $R_F\cD_Y$-module as a $\cO_Y[\hb]$-module equipped with an integrable $\hb$-connection~$\hb\nabla$.

\subsubsection*{Exponential twist}
In the following, $X$ will be a complex manifold and we will consider holonomic $\cD$-modules $\ccM$ on $Y=X\times\PP^1$. For example, given a holomorphic function $f:X\to\PP^1$ and a holonomic $\cD_X$-module~$\ccN$, we will consider the push-forward $\ccM=i_{f,+}\ccN$ of $\ccN$ by the graph inclusion \hbox{$i_f:X\hto X\times\PP^1$}. As in \S\ref{subsec:setup}, we will consider $\PP^1$ as covered by two charts $\Afu_t$ and $\Afu_{t'}$ in such a way that $\infty=\{t'=0\}$.

We denote by
$$
q:X\times\PP^1\to\PP^1
$$
the projection and we will simply denote by $\infty$ the divisor $X\times\{\infty\}$ in $X\times\PP^1$. Let us consider the localization $\wt\ccM=\cO_{X\times\PP^1}(*\infty)\otimes_{\cO_{X\times\PP^1}}\nobreak\ccM$ of $\ccM$, which is a holonomic $\cD_{X\times\PP^1}$-module by a theorem of Kashiwara. We also regard it as a $\cO_{X\times\PP^1}(*\infty)$-module with integrable connection~$\nabla$. We denote by $\ccM\otimes\ccE^q$ the $\cO_{X\times\PP^1}(*\infty)$-module~$\wt\ccM$ equipped with the integrable connection $\nabla+\rd q$ (\cf \S\ref{subsec:Dapproach} for the similar notation $E^{1/g}$). It is also a holonomic $\cD_{X\times\PP^1}$-module.

Let us now consider these constructions for a filtered $\cD_{X\times\PP^1}$-module $(\ccM,F_\bbullet\ccM)$. We set $F_p\wt\ccM=\cO_{X\times\PP^1}(*\infty)\otimes_{\cO_{X\times\PP^1}}F_p\ccM$ (this is not $\cO_{X\times\PP^1}$-coherent). We then have
\[
R_F\wt\ccM=\cO_{X\times\PP^1}(*\infty)[\hb]\otimes_{\cO_{X\times\PP^1}[\hb]}R_F\ccM=:(R_F\ccM)(*\infty).
\]

We consider the $\cO_{X\times\PP^1}[\hb]$-module $\cO_{X\times\PP^1}(*\infty)[\hb]$ equipped with the integrable $\hb$\nobreakdash-connec\-tion $\hb\rd+\hb\rd q$, that we still denote by $\ccE^q$ (although it is equal to $\ccE^q[\hb]$). We then define $R_F\ccM\otimes\ccE^q$ as $R_F\wt\ccM$ equipped with the integrable $\hb$-connection $\hb\nabla+\hb\rd q$.

\subsection{Strict specializability along a hypersurface}\label{subsec:strictspe}
The notion of $V$-filtration will play an important role for the construction of the Deligne filtration. We will distinguish two notions for a filtered $\cD_{X\times\PP^1}$-module $\ccM$: the notion of strict specializability of $(\ccM,F_\bbullet\ccM)$ as a filtered $\cD_{X\times\PP^1}$-module \cite{MSaito86}, and that of strict specializability of $R_F\ccM$ as a $R_F\cD_{X\times\PP^1}$-module \cite{Bibi01c}. If one uses the definition as stated in \cite[\S3.2]{MSaito86} for $(\ccM,F_\bbullet\ccM)$, one does not recover exactly that given in \cite[Def.\,3.3.8]{Bibi01c} for $R_F\ccM$. This is why we will strengthen that of \cite[\S3.2]{MSaito86}, and we will show that mixed Hodge modules in the sense of \cite{MSaito87} also satisfy the strengthened condition.

\subsubsection*{Specialization of a filtered $\cD_{X\times\PP^1}$-module}
Let $X$ be a complex manifold and let $(\ccM,F_\bbullet\ccM)$ be a filtered holonomic $\cD$-module on $X\times\PP^1$. Since $\ccM$ is holonomic, it admits a Kashiwara-Malgrange filtration $V_\bbullet\ccM$ along $X\times\{\infty\}$ indexed by $A+\ZZ$, for some finite set $A\subset\CC$ equipped with some total order. We will not care about the choice of such an order by assuming that $A\subset\RR$, and equipped with the induced order. This assumption will be enough for our purpose. We can extend in a trivial way the filtration as a filtration indexed by~$\RR$, with only possible jumps at $A+\ZZ$ at most. The normalization we use for the Kashiwara-Malgrange filtration is that $t'\partial_{t'}+\alpha$ is nilpotent on $\gr_\alpha^V\ccM:=V_\alpha\ccM/V_{<\alpha}\ccM$, for each $\alpha\in A+\ZZ$ (so there will be a shift with the convention in \cite{MSaito86} and \cite{Bibi01c}).

\begin{definition}[{\cf\cite[(3.2.1)]{MSaito86}}]\label{def:strictspe}
Let $(\ccM,F_\bbullet\ccM)$ be a filtered holonomic $\cD_{X\times\PP^1}$-module. We say that $(\ccM,F_\bbullet\ccM)$ is \emph{strictly specializable and regular} along $X\times\{\infty\}$~if
\begin{enumerate}
\item\label{def:strictspe2}
Compatibility conditions in \cite[\S3.2.1]{MSaito86}:
\begin{enumerate}
\item\label{def:strictspe2a}
for each $\alpha<1$ and each $p$, $t':F_p\ccM\cap V_\alpha\ccM\isom F_p\ccM\cap V_{\alpha-1}\ccM$,
\item\label{def:strictspe2b}
for each $\alpha>0$, $\partial_{t'}:F_p\gr_\alpha^V\ccM\isom F_{p+1}\gr_{\alpha+1}^V\ccM$.
\end{enumerate}
\item\label{def:strictspe1}
For each $\alpha\in\RR$, the filtration $F_\bbullet\ccM$ induces on each $\gr^V_\alpha\ccM$ a good filtration (with respect to $F_\bbullet\cD_{X\times\{\infty\}}$).
\end{enumerate}
\end{definition}

We refer to \cite[\S3.2]{MSaito86} for the consequences of \eqref{def:strictspe2}. By definition, for a polarizable Hodge module \cite{MSaito86} or more generally a (graded-polarizable) mixed Hodge module \cite{MSaito87}, $(\ccM,F_\bbullet\ccM)$ is strictly specializable and regular along $X\times\{\infty\}$ in the sense of Definition~\ref{def:strictspe}.

\subsubsection*{Specialization of a $R_F\cD_{X\times\PP^1}$-module}
Consider the increasing filtration $V_\bbullet(R_F\cD_{X\times\PP^1})$ indexed by $\ZZ$, which is constant equal to $R_F\cD_{X\times\Afu}$ in the chart with coordinate $t$, and for which, in the chart with coordinate $t'$, the function $t'$ has degree $-1$, the vector field $\partiall_{t'}$ has degree $1$, and $\cO_X[\hb]$ and $\hb\Theta_X$ have degree zero.

\begin{definition}[{\cf\cite[Def.\,3.3.8]{Bibi01c}}]\label{def:strictspeR}
Let $\cM$ be a coherent $R_F\cD_{X\times\PP^1}$-module (\eg $\cM=R_F\ccM$ as above).
\begin{enumerate}
\item\label{def:strictspeRa}
We say that $\cM$ is \emph{strictly specializable} along $X\times\{\infty\}$ if there exists a finite set $A\subset\RR$ and a good $V$-filtration of $\cM$ (good with respect to $V_\bbullet(R_F\cD_{X\times\PP^1})$) indexed by $A+\ZZ$, such that
\begin{enumerate}
\item\label{def:strictspeR1}
each graded module $\gr_\alpha^V\cM$ is strict, \ie it has no $\CC[\hb]$-torsion,
\item\label{def:strictspeR2}
On each $\gr_\alpha^V\cM$, the operator $t'\partiall_{t'}+\hb \alpha=\hb(t'\partial_{t'}+\alpha)$ is nilpotent (the normalization is shifted by one with respect to \cite{Bibi01c}, for later convenience),
\item\label{def:strictspeR3}
the map $t':V_{\alpha}\cM\to V_{\alpha-1}\cM$ is an isomorphism for $\alpha<1$,
\item\label{def:strictspeR4}
the map $\partiall_{t'}:\gr_\alpha^V\cM\to\gr_{\alpha+1}^V\cM$ is an isomorphism for $\alpha>0$.
\end{enumerate}
\item\label{def:strictspeRb}
We then say (\cf\cite[\S3.1.d]{Bibi01c}) that $\cM$ is \emph{regular} along $X\times\{\infty\}$ if, for any $\alpha\in\RR$, the restriction of $V_\alpha\cM$ to some neighbourhood of $X\times\{\infty\}$ is coherent over $R_F\cD_{X\times\PP^1/\PP^1}$ (and not only over $R_FV_0\cD_{X\times\PP^1}=V_0R_F\cD_{X\times\PP^1}$).
\end{enumerate}
\end{definition}

\begin{remarks}\mbox{}
\begin{enumerate}
\item
This definition gives a subcategory of that considered in \cite[\S3.3]{Bibi01c} (\cf also \cite[Chap.\,14]{Mochizuki07}, \cite[Chap.\,22]{Mochizuki08}), as we implicitly assume that $t'\partiall_{t'}$ acting on $\gr_\alpha^V\cM$ has the only eigenvalue $\hb\alpha$ with $\alpha\in\RR$, so that we will in fact implicitly assume that $\alpha\in\RR$ if $\gr_\alpha^V\cM\neq0$. However, this subcategory is enough for our purpose.
\item
Such a filtration is unique (\cf \cite[Lem.\,3.3.4]{Bibi01c}).
\item
Conditions \ref{def:strictspeR}\eqref{def:strictspeR3} and \eqref{def:strictspeR4} are not the conditions given in \cite[Def.\,3.3.8(1b,c)]{Bibi01c}, but are equivalent to them, according to \cite[Rem.\,3.3.9(2)]{Bibi01c}.
\end{enumerate}
\end{remarks}

\begin{proposition}\label{prop:Hodgestrictspe}
Assume that $(\ccM,F_\bbullet\ccM)$ underlies a polarizable Hodge module \cite{MSaito86} or more generally a (graded-polarizable) mixed Hodge module \cite{MSaito87}. Then $R_F\ccM$ is strictly specializable and regular along $X\times\{\infty\}$ in the sense of Definition~\ref{def:strictspeR}. Moreover, the $V$-filtration $V_\bbullet(R_F\ccM)$ of $R_F\ccM$ as a $R_F\cD_{X\times\PP^1}$-module is equal to $R_FV_\bbullet\ccM$, where we have set $R_FV_\alpha\ccM=\bigoplus_p(F_p\ccM\cap V_\alpha\ccM)\hb^p$.
\end{proposition}

\begin{proof}
We\marginpars{The proof given in Old App.\,\ref{app:B} is much simplified by using \cite[Cor.\,3.4.7]{MSaito86}.} first note that, according to \cite[Cor.\,3.4.7]{MSaito86}, if $(\ccM,F_\bbullet\ccM)$ is as in Definition \ref{def:strictspe}, then it also satisfies
\begin{enumerate}
\item[(\ref{def:strictspe1}$'$)]
for each $\alpha\in\RR$, the filtration $F_\bbullet\ccM$ induces in the neighbourhood of $X\times\{\infty\}$ on each $V_\alpha\ccM$ a good filtration with respect to $F_\bbullet\cD_{X\times\PP^1/\PP^1}$.
\end{enumerate}

For each $\alpha\in\RR$, let us set $U_\alpha R_F\ccM:=R_F(V_\alpha\ccM)$ with $F_pV_\alpha\ccM:=F_p\ccM\cap\nobreak V_\alpha\ccM$. This is a $V$-filtration since $R_FV_0\cD_{X\times\PP^1}=V_0(R_F\cD_{X\times\PP^1})$. Moreover, we have $\gr_\alpha^UR_F\ccM=R_F\gr_\alpha^V\ccM$, hence $\gr_\alpha^UR_F\ccM$ is strict, \ie has no $\CC[\hb]$-torsion. So \ref{def:strictspeR}\eqref{def:strictspeR1} is fulfilled.

Note that $R_FV_\alpha\ccM$ is left invariant by $\hb t'\partial_{t'}=:t'\partiall_{t'}$, and that $t'\partiall_{t'}+\alpha\hb$ is nilpotent on $R_F\gr_\alpha^V\ccM$ since $t'\partial_{t'}+\alpha$ is so on $\gr_\alpha^V\ccM$. Therefore, we get \ref{def:strictspeR}\eqref{def:strictspeR2}.

We also note that (\ref{def:strictspe1}$'$) implies that each $U_\alpha R_F\ccM$ is $V_0(R_F\cD_{X\times\PP^1})$-coherent. Moreover, \ref{def:strictspe}\eqref{def:strictspe2} implies the conditions \ref{def:strictspeR}\eqref{def:strictspeR3}--\eqref{def:strictspeR4} which, together with the coherence of $U_\alpha R_F\ccM$, implies the goodness property of this $V$-filtration. By the uniqueness of the $V$-filtration, we conclude that $U_\bbullet R_F\ccM=V_\bbullet R_F\ccM$.

Lastly, (\ref{def:strictspe1}$'$) means that the restriction of $V_\alpha R_F\ccM$ to a neighbourhood of $X\times\{\infty\}$ is $R_F\cD_{X\times\PP^1/\PP^1}$-coherent for each $\alpha$, which is the regularity property \ref{def:strictspeR}\eqref{def:strictspeRb}.
\end{proof}

\subsection{Partial Laplace exponential twist of $R_F\cD_{X\times\PP^1}$-modules}\label{subsec:partialLaplace}
Let $(\ccM,F_\bbullet\ccM)$ be a filtered $\cD_{X\times\PP^1}$-module. We will usually denote by $(M,F_\bbullet M)$ its restriction to $X\times\Afu_t$, that we regard as a filtered $\cD_X[t]\langle\partial_t\rangle$-module, and by $(M',F_\bbullet M')$ its restriction to $X\times\Afu_{t'}$, that we regard as a filtered $\cD_X[t']\langle\partial_{t'}\rangle$-module. The Laplace exponential twist $\cFcM$ of $R_F\ccM$ that we define below is an intermediate step to define the partial Laplace transform of $R_F\ccM$, but we will not need to introduce the latter.

We consider the affine line $\Afuh$ with coordinate $\tau$. The varieties $X\times\PP^1$ and $Z=X\times\PP^1\times\Afuh$ are equipped with a divisor (still denoted by) $\infty$. We denote by~$p$ the projection $Z\to X\times\PP^1$.

Let $\cM$ be a left $R_F\cD_{X\times\PP^1}$-module, \eg $\cM=R_F\ccM$. We denote by $\wt\cM$ the localized module $R_F\cD_{X\times\PP^1}(*\infty)\otimes_{R_F\cD_{X\times\PP^1}}\cM$, \eg $\wt\cM=R_F\wt\ccM$ as defined in \S\ref{subsec:filtDmodRmod}. Then $p^+\wt\cM$ is a left $R_F\cD_Z(*\infty)$-module. We denote by $p^+\wt\cM\otimes\cE^{t\tau/\hb}$ or, for short, by $\cFcM$, the $R_F\cO_Z(*\infty)$-module $p^+\wt\cM$ equipped with the twisted action of $R_F\cD_Z$ described by the exponential factor: the $R_F\cD_X$-action is unchanged, and, for any local section $m$ of~$\cM$,
\begin{itemize}
\item
on $X\times\Afu_t\times\Afuh$,
\begin{equation}\label{eq:ttau}
\begin{split}
\partiall_t(m\otimes\cE^{t\tau/\hb})&=[(\partiall_t+\tau)m]\otimes\cE^{t\tau/\hb},\\
\partiall_\tau(m\otimes\cE^{t\tau/\hb})&=tm\otimes\cE^{t\tau/\hb},
\end{split}
\end{equation}

\item
on $X\times\Afu_{t'}\times\Afuh$,
\begin{equation}\label{eq:tprimetau}
\begin{split}
\partiall_{t'}(m\otimes\cE^{t\tau/\hb})&=[(\partiall_{t'}-\tau/t^{\prime2})m]\otimes\cE^{t\tau/\hb},\\
\partiall_\tau(m\otimes\cE^{t\tau/\hb})&=m/t'\otimes\cE^{t\tau/\hb},
\end{split}
\end{equation}
\end{itemize}

\begin{lemma}\label{lem:injcFcM}
The left multiplication by $\tau-\hb$ is injective on $\cFcM$ and the cokernel is identified with $\cM\otimes\ccE^t$ (that is, $\cM\otimes_{\cO_X[\hb]}\ccE^t[\hb]$ equipped with its natural $\hb$-connection or, equivalently, $\wt\cM$ equipped with the twisted $\hb$-connection).
\end{lemma}

\begin{proof}
We can realize $p^+\wt\cM$ algebraically as $\CC[\tau]\otimes_\CC \wt\cM$ as a $\cO_X[\hb,\tau]$-module, with the twisted $\partiall_t,\partiall_{t'},\partiall_\tau$-action as above. Then the first statement is obvious by considering the filtration with respect to the degree in $\tau$, and the second statement is obtained by replacing $\tau$ with $\hb$ in \eqref{eq:ttau} and \eqref{eq:tprimetau}.
\end{proof}

\subsection{Partial Laplace exponential twist and specialization of $R_F\cD_{X\times\PP^1}$-modules}\label{subsec:partialLaplacespe}

Let $(\ccM,F_\bbullet\ccM)$ be a filtered holonomic $\cD_{X\times\PP^1}$-module which underlies a~mixed Hodge module and let $\cM$ be the $R_F\cD_{X\times\PP^1}$-module defined as $\cM=R_F\ccM$. Then $\cM$ is a strictly specializable and regular along $X\times\{\infty\}$ in the sense of Definition \ref{def:strictspeR}, according to Proposition \ref{prop:Hodgestrictspe}. It follows from \cite[Prop.\,4.1(ii)]{Bibi05b} that $\cFcM$ is strictly specializable and regular along $\tau=0$. We will denote by $V_\bbullet^\tau\cFcM$ the corresponding $V$-filtration. Let us recall the main steps for proving the strict specializability and the regularity in the present setting, which is simpler than the general one considered in \loccit, since the eigenvalues of the monodromy at infinity have absolute value equal to one (and more precisely, are roots of unity).

On the one hand, for $\alpha\in(0,1]$, one identifies $\gr_\alpha^{V^\tau}\!\cFcM$ (a $R_F\cD_{X\times\PP^1}$-module) with the push-forward, by the inclusion $i_\infty:X\times\{\infty\}\hto X\times\PP^1$, of $\gr_{\alpha-1}^VR_F\ccM$ (\cf \loccit, proof of (ii)(6) and (ii)(7)), whose strictness follows from the first part of Proposition \ref{prop:Hodgestrictspe}.

On the other hand, let us denote by $\ccM_{\min}$ the minimal extension of $\ccM$ along $X\times\{\infty\}$. This is also a mixed Hodge module: if $j:X\times\Afu\hto X\times\PP^1$ denote the inclusion, we have $\ccM_{\min}=\image[j_!j^*\ccM\to j_*j^*\ccM]$ in the category of mixed Hodge modules (\cf \cite{MSaito87}). Moreover, $R_F\ccM_{\min}$ corresponds to the minimal extension of $R_F\ccM$, in the sense of \cite[Def.\,3.4.7]{Bibi01c}. Then the proof of (ii)(8) in \loccit identifies $\gr_0^{V^\tau}\!\cFcM$ to a successive extension of the objects $\gr_1^{V^\tau}\!\cFcM$, $R_F\ccM_{\min}$, $i_{\infty,+}\ker\rN_{t'}$ and $i_{\infty,+}\coker\rN_{t'}$, where $\rN_{t'}$ is the nilpotent part of the monodromy on $\gr_0^VR_F\ccM_{\min}=\gr_0^VR_F\ccM$, and all these components are known to be strict, by the discussion of first part for the first one, and by the very definition of mixed Hodge modules for the rest.

Let us emphasize at this point that, according to the previous result and the second part of Proposition \ref{prop:Hodgestrictspe}, for each $\alpha\in(0,1]$, $\gr_\alpha^{V^\tau}\!\cFcM$ is identified with the Rees module of a filtered $\cD$-module underlying, up to a shift of the filtration, a direct summand of a mixed Hodge module (recall that, for $(\ccM,F_\bbullet\ccM)$ underlying a mixed Hodge module, $\bigoplus_{\alpha\in(0,1]}(\gr_\alpha^V\ccM,F_\bbullet\gr_\alpha^V\ccM)$ also underlies a mixed Hodge module). The same property holds if $\alpha=0$, as explained in Appendix \ref{app:B}.\marginpars{App.\,\ref{app:B} is changed and gives now the complement needed here.}

Lastly, the $V$-filtration of $\cFcM$ along $\tau=0$ is given by an explicit formula from the $V$-filtration of $\cM$ along $t'=0$ (see the proof of 4.1(ii) in \cite{Bibi05b}). For our purpose, we have the formula already used in the proof of \cite[Prop.\,6.10]{Bibi08}, when considering the chart $\Afu_{t'}$:
\[
\forall \alpha\in[0,1),\quad V^\tau_\alpha(\cFcM)_{|X\times\Afu_{t'}\times\Afuh}=\sum_{k\geq0}(1\otimes\hb\partial_{t'}-\tau\otimes t^{\prime-2})^k(\CC[\tau]\otimes_\CC t^{\prime-1}R_FV^{t'}_\alpha M'),
\]
where we have indicated as an exponent the variable with respect to which the $V$-filtration is taken. In the chart $\Afu_t$, we simply have
\[
\forall \alpha\in[0,1),\quad V^\tau_\alpha(\cFcM)_{|X\times\Afu_t\times\Afuh}=\cFcM_{|X\times\Afu_t\times\Afuh}.
\]

We note that, according to Lemma \ref{lem:injcFcM}, left multiplication by $\tau-\hb$ is injective on $V^\tau_\alpha(\cFcM)$.

\section{A generalization of Deligne's filtration}\label{sec:defDel}
We keep the notation as in \S\ref{subsec:filtDRFD}. Our purpose in this section is to prove:

\begin{theorem}\label{th:main2}
For each filtered holonomic $\cD_{X\times\PP^1}$-module $(\ccM,F_\bbullet\ccM)$ one can define canonically and functorially a $F_\bbullet\cD_{X\times\PP^1}$-filtration $F^\Del_\bbullet(\ccM\otimes\ccE^q)$.
\begin{enumerate}
\item\label{th:main21}
If $(\ccM,F_\bbullet\ccM)$ underlies a mixed Hodge module, then $F^\Del_\bbullet(\ccM\otimes\ccE^q)$ is a good $F_\bbullet\cD_{X\times\PP^1}$-filtration.
\item\label{th:main22}
For each morphism $\varphi:(\ccM_1,F_\bbullet\ccM_1)\to(\ccM_2,F_\bbullet\ccM_2)$ underlying a morphism of mixed Hodge modules, the corresponding morphism
\[
\varphi^q:(\ccM_1\otimes\ccE^q,F^\Del_{\cbbullet}(\ccM_1\otimes\ccE^q))\to(\ccM_2\otimes\ccE^q,F^\Del_{\cbbullet}(\ccM_2\otimes\ccE^q))
\]
is strictly filtered.
\item\label{th:main23}
For $(\ccM,F_\bbullet\ccM)$ underlying a mixed Hodge module, the spectral sequence attached to the hypercohomology of the filtered de~Rham complex $F^\Del_\bbullet\DR(\ccM\otimes\nobreak\ccE^q)$ degenerates at~$E_1$.
\end{enumerate}
\end{theorem}

\subsection{Definition of Deligne's filtration}\label{subsec:defDel}
Let $(\ccM,F_\bbullet\ccM)$ be a filtered holonomic $\cD_{X\times\PP^1}$-module. We recall that $\ccM\otimes\ccE^q=\wt\ccM$ as an $\cO_{X\times\PP^1}$-module. We will implicitly use the description of $(\ccM,F_\bbullet\ccM)$ as a pair consisting of a filtered $\cD_X[t]\langle\partial_t\rangle$-module $(M,F_\bbullet M)$ and a filtered $\cD_X[t']\langle\partial_{t'}\rangle$-module $(M',F_\bbullet M')$ with the standard identification. We define, for $\alpha\in[0,1)$ and $p\in\ZZ$,
\begin{equation}\label{eq:FDel}
\begin{split}
F^\Del_{\alpha+p}(\ccM\otimes\ccE^q)_{|X\times\Afu_t}&=F_p\ccM_{|X\times\Afu_t}=F_pM,\\[3pt]
F^\Del_{\alpha+p}(\ccM\otimes\ccE^q)_{|X\times\Afu_{t'}}&=\sum_{k\geq0}\partial_{t'}^k\tpm \big[(F_{p-k}M'\cap V_\alpha M')\otimes E^{1/t'}\big],
\end{split}
\end{equation}
and the last term is included in $F_pM'(*\infty)$: recall from the general properties of the Kashiwara-Malgrange filtration (\cf\eg \cite{MSaito86,Bibi87}) that, for $\alpha\in[0,1)$, the restriction of the localization morphism $M'\to M'(*\infty)$ to $V_\alpha M'$ is injective. We can therefore regard each $F_{p-k}M'\cap V_\alpha M'$ ($k\geq0$) as being contained in $F_{p-k}M'(*\infty)$, hence its image by the operator on its left is contained in $F_pM'(*\infty)$.

The Deligne filtration satisfies properties similar to those of its special case \eqref{eq:Fdelalpha}. It is a $F_\bbullet\cD_{X\times\PP^1}$-filtration (this is clear on $X\times\Afu_t$ and this is proved as for \eqref{eq:Fdelalpha} on $X\times\Afu_{t'}$). Similarly, \eqref{eq:FdF} holds on $X\times\Afu_{t'}$. Note also that, since $F_pM'=0$ for $p\ll0$, each $F_{\alpha+p}^\Del(\ccM\otimes\ccE^q)$ is $\cO_{X\times\PP^1}$-coherent.

If we set $F_{<\mu}^\Del(\ccM\otimes\ccE^q)=\sum_{\mu'<\mu}F_{\mu'}^\Del(\ccM\otimes\ccE^q)$ and $\gr_\mu^{F^\Del}(\ccM\otimes\ccE^q)=F_\mu^\Del/F_{<\mu}^\Del$, then $\gr_\mu^{F^\Del}(\ccM\otimes\ccE^q)$ is supported on $X\times\{\infty\}$ if $\mu\not\in\ZZ$.

\begin{proposition}\label{prop:RFV}
Assume that $(\ccM,F_\bbullet\ccM)$ underlies a mixed Hodge module. Then, for each $\alpha\in[0,1)$, the Rees module $R_{F^\Del_{\alpha+\bbullet}}(\ccM\otimes\ccE^q)$ is obtained (up to forgetting the grading)\marginpars{This is added to make the argument clearer.} by the formula
\[
R_{F^\Del_{\alpha+\bbullet}}(\ccM\otimes\ccE^q)=V^\tau_\alpha(\cFcM)/(\tau-\hb)V^\tau_\alpha(\cFcM).
\]
\end{proposition}

(Recall that we denote by $\cM$ the Rees module $R_F\ccM$ and by $\cFcM$ its Laplace exponential twist, \cf\S\ref{subsec:partialLaplace}.)

\begin{proof}
We will use the expression of $V^\tau_\alpha(\cFcM)$ given in \S\ref{subsec:partialLaplacespe}. The equality is easy in the chart $X\times\Afu_t\times\Afuh$, and we will consider the chart $X\times\Afu_{t'}\times\Afuh$. Then it follows from the expression of $V^\tau_\alpha(\cFcM)$ that
\begin{align*}
V^\tau_\alpha\cFcM/\big[(\tau-\hb)\cFcM\cap V^\tau_\alpha\cFcM\big]
&=
\sum_{k\geq0}(\partial_{t'}-t^{\prime-2})^kt^{\prime-1}\hb^k\CC[\hb]R_FV_\alpha^{t'}\ccM\\
&=R_{F^\Del_{\alpha+\bbullet}}(\ccM\otimes\ccE^q).
\end{align*}
It remains to show that $V^\tau_\alpha\cFcM\cap(\tau-\hb)\cFcM=(\tau-\hb)V^\tau_\alpha\cFcM$. This is a consequence of the strictness of $\gr_\gamma^{V^\tau}\!\cFcM$ for any $\gamma$, as recalled in \S\ref{subsec:partialLaplacespe}. Indeed, assume that $m\in V^\tau_\gamma\cFcM$ is such that $(\tau-\hb)m\in V^\tau_\alpha\cFcM$. If $\gamma>\alpha$ and the class of $m$ in $\gr_\gamma^{V^\tau}\!\cFcM$ is not zero, then the class of $(\tau-\hb)m$ is zero in $\gr_\gamma^{V^\tau}\!\cFcM$, and this is nothing but the class of $-\hb m$. By strictness, the multiplication by $\hb$ is injective on $\gr_\gamma^{V^\tau}\!\cFcM$, which leads to a contradiction.
\end{proof}

\begin{remark}
The natural inclusion
\[
R_{F^\Del_{\alpha+\bbullet}}(\ccM\otimes\ccE^q)=V^\tau_\alpha\cFcM/\big[(\tau-\hb)\cFcM\cap V^\tau_\alpha\cFcM\big]\hto\cFcM/(\tau-\hb)\cFcM
\]
can be understood as follows. Recall that $\cFcM/(\tau-\hb)\cFcM$ is identified with $R_F\wt\ccM$ (forgetting its grading)\marginpars{Added.} with twisted action of $\partiall_t,\partiall_{t'}$ (Lemma \ref{lem:injcFcM})\marginpars{Reference corrected.}, where the filtration $F_p\wt\ccM$ is defined as $(F_p\ccM)(*\infty)$. We then remark that, if $\alpha\in[0,1)$ is fixed, we have a natural inclusion $F_{\alpha+p}^\Del\wt\ccM\subset (F_p\ccM)(*\infty)$, since both coincide on $X\times\Afu_t$. This is the natural inclusion above.
\end{remark}

\begin{proof}[Proof of \ref{th:main2}\eqref{th:main21}]
As already mentioned in \S\ref{subsec:partialLaplacespe}, $\cFcM$ is strictly specializable and regular along $\tau=0$. This implies that $V_\alpha^\tau\cFcM$ is $R_F\cD_{Z/\Afuh}$-coherent. By Proposition \ref{prop:RFV}, $R_{F^\Del_{\alpha+\bbullet}}(\ccM\otimes\ccE^q)$ is then $R_F\cD_{Z/\Afuh}/(\tau-\hb)R_F\cD_{Z/\Afuh}$-coherent. In order to conclude, it remains to identify the latter quotient with $R_F\cD_{X\times\PP^1}$, which is straightforward.
\end{proof}

\begin{proof}[Proof of \ref{th:main2}\eqref{th:main22}]
Let\marginpars{A few changes here in order to make clear that the role of the grading.} us denote by $\ccK,\ccI,\ccC$ the kernel, image and cokernel of $\varphi$, with the induced filtration~$F_\bbullet$, which underlie mixed Hodge modules, according to \cite{MSaito87} (due to the strictness of $\varphi$, both natural filtrations on $\ccI$ coincide). By the strictness of $\varphi$, $R_F\ccK,R_F\ccI,R_F\ccC$ are the kernel, image and cokernel of $R_F\varphi$. On the other hand, $\varphi$ is strictly compatible with the filtration $V_\bbullet^{t'}$. It is therefore compatible with the filtrations $F^\Del_{\alpha+\bbullet}$, and for each $\alpha\in[0,1)$ we get sequences of graded $R_F\cD_{X\times\PP^1}$-modules
\begin{equation}\label{eq:exactRF}
\begin{gathered}
0\to R_{F^\Del_{\alpha+\bbullet}}(\ccK\otimes\ccE^q)\to R_{F^\Del_{\alpha+\bbullet}}(\ccM_1\otimes\ccE^q)\to R_{F^\Del_{\alpha+\bbullet}}(\ccI\otimes\ccE^q)\to0,\\
0\to R_{F^\Del_{\alpha+\bbullet}}(\ccI\otimes\ccE^q)\to R_{F^\Del_{\alpha+\bbullet}}(\ccM_2\otimes\ccE^q)\to R_{F^\Del_{\alpha+\bbullet}}(\ccC\otimes\ccE^q)\to0,
\end{gathered}
\end{equation}
that we will prove to be exact, in order to get the desired exact sequence of graded $R_F\cD_{X\times\PP^1}$-modules
\[
0\ra R_{F^\Del_{\alpha+\bbullet}}(\ccK\otimes\ccE^q)\ra R_{F^\Del_{\alpha+\bbullet}}(\ccM_1\otimes\ccE^q)\ra R_{F^\Del_{\alpha+\bbullet}}(\ccM_2\otimes\ccE^q)\ra R_{F^\Del_{\alpha+\bbullet}}(\ccC\otimes\ccE^q)\ra0.
\]
In order to prove exactness, we can now forget the grading and express the $R_F\cD_{X\times\PP^1}$-module $R_{F_{\alpha+\bbullet}^\Del}\ccM_i$ as in Proposition \ref{prop:RFV}.

For each $\alpha\in[0,1]$, the morphism $\gr_\alpha^{V^{t'}}\!\!\varphi:\gr_\alpha^{V^{t'}}(\ccM_1,F_\bbullet\ccM_1)\to\gr_\alpha^{V^{t'}}(\ccM_2,F_\bbullet\ccM_2)$ is strict, since it underlies a morphism of mixed Hodge modules, according to \cite{MSaito87}. Using \cite[Prop.\,4.1]{Bibi05b}, we find that $\varphi$ induces a strictly specializable morphism (in the sense of \cite[Def.\,3.3.8(2)]{Bibi01c}) $\cFvarphi:\cFcM_1\to\cFcM_2$. Applying \cite[Lem.\,3.3.10]{Bibi01c}, we obtain that $\cFvarphi$ is $V$-strict.

 Since the construction~$\cF$ (pull-back $p^+$ followed by the twist $\cE^{t\tau/\hb}$) is exact, $\cFcK,\cFcI,\cFcC$ are the kernel, image and cokernel of $\cFvarphi$.

Then, by the $V$-strictness of $\cFvarphi$, we get exact sequences for each $\alpha\in[0,1)$,
\begin{gather*}
0\to V_\alpha^\tau\cFcK\to V_\alpha^\tau\cFcM_1\to V_\alpha^\tau\cFcI\to0,\\
0\to V_\alpha^\tau\cFcI\to V_\alpha^\tau\cFcM_2\to V_\alpha^\tau\cFcC\to0.
\end{gather*}
Since $\tau-\hb$ is injective on each term, we can apply Proposition \ref{prop:RFV} to get the desired exactness of \eqref{eq:exactRF}.\end{proof}

\subsection{Behaviour with respect to push-forward}\label{subsec:push-fwd}
Let $h:X\to Y$ be a projective morphism. Let us also denote by $h$ the projective morphism $h\times\id:X\times\PP^1\to Y\times\PP^1$ and by $r$ the projection $Y\times\PP^1\to\PP^1$, so that $q=r\circ h$. Let $(\ccM,F_\bbullet\ccM)$ be a filtered holonomic $\cD_{X\times\PP^1}$-module which underlies a mixed Hodge module \cite{MSaito87}. Its push-forward $h_+(\ccM,F_\bbullet\ccM)$ is then strict (see Appendix \ref{sec:degenerationstrictness}), that~is,
\begin{equation}\label{eq:imdirstrict}
\forall j,\quad\cH^jh_+R_F\ccM=R_F\cH^jh_+\ccM,
\end{equation}
which is equivalent to asking that the left-hand term has no $\CC[\hb]$-torsion. Moreover, each $(\cH^jh_+\ccM,F_\bbullet\cH^jh_+\ccM)$ underlies a mixed Hodge module. If $(\ccM,F_\bbullet\ccM)$ is pure (and polarizable), then each $(\cH^jh_+\ccM,F_\bbullet\cH^jh_+\ccM)$ is also pure (and polarizable).

On the other hand, we clearly have $h_+(\ccM\otimes\ccE^q)=(h_+\ccM)\otimes\ccE^r$. Therefore, $\cH^jh_+(\ccM\otimes\ccE^q)=\cH^jh_+(\ccM)\otimes\ccE^r$. Similarly, considering the twisted objects by $\cE^{t\tau/\hb}$ (\cf\S\ref{subsec:partialLaplace}) and according to \eqref{eq:imdirstrict}, we have, with obvious notation,
\begin{equation}\label{eq:imdirLaplacestrict}
\forall j,\quad\cH^jh_+\cFcM={}^{{}^{\mathcal{F}}\!\!}(R_F\cH^jh_+\ccM).
\end{equation}

\begin{proposition}\label{prop:F+strict}
If $(\ccM,F_\bbullet\ccM)$ underlies a mixed Hodge module then, for each \hbox{$\alpha\in[0,1)$}, we have strictness of the push-forward of the Deligne filtration:
\[
\cH^j\big(h_+R_{F^\Del_{\alpha+\bbullet}}(\ccM\otimes\ccE^q)\big)=R_{F^\Del_{\alpha+\bbullet}}\cH^jh_+(\ccM)\otimes\ccE^r,\quad\forall j.
\]
\end{proposition}

\begin{proof}
As we have seen in \S\ref{subsec:partialLaplacespe}, $\cFcM$ is strictly specializable along $\tau=0$ and each $\gr_\alpha^{V^\tau}\!\cFcM$ is a direct summand of the Rees module of a mixed Hodge \hbox{module} (up~to a shift of the filtration). \hbox{It follows} from \cite{MSaito87} that $h_+\gr_\alpha^{V^\tau}\!\cFcM$ is strict for each such~$\alpha$ (\cf Appendix \ref{sec:degenerationstrictness}). According to \cite[Th.\,3.1.8]{Bibi01c}, each $\cH^jh_+(\cFcM)$ is strictly specializable along $\tau=0$ and we have
\[
\cH^j(h_+V_\alpha^\tau\cFcM)=V_\alpha^\tau\cH^jh_+(\cFcM).
\]
According to \eqref{eq:imdirLaplacestrict}, we can apply the results of \S\ref{subsec:partialLaplacespe} to $\cH^j(h_+\cFcM)$. Therefore, for each $j$, $\tau-\hb$ is then injective on $\cH^jh_+(V_\alpha^\tau\cFcM)$ since it is injective on $V_\alpha^\tau\cH^jh_+(\cFcM)$. Arguing by decreasing induction on $j$, we find that for each $j$ the sequence
\[
0\ra\cH^jh_+(V_\alpha^\tau\cFcM)\To{\tau-\hb}\cH^jh_+(V_\alpha^\tau\cFcM)\to\cH^jh_+\big(V_\alpha^\tau\cFcM/(\tau-\hb)V_\alpha^\tau\cFcM\big)\ra0
\]
is exact and is identified with the sequence
\begin{multline*}
0\ra V_\alpha^\tau\cH^jh_+(\cFcM)\To{\tau-\hb}V_\alpha^\tau\cH^jh_+(\cFcM)\\
\to V_\alpha^\tau\cH^jh_+(\cFcM)\big/\big[(\tau-\hb)V_\alpha^\tau\cH^jh_+(\cFcM) \big]\ra0.
\end{multline*}
The proposition now follows from Proposition \ref{prop:RFV} that we apply to both $\cFcM$ and $\cH^jh_+\cFcM$, according to \eqref{eq:imdirLaplacestrict}.
\end{proof}

\subsection{$E_1$-degeneration}
We keep the setting as in \S\ref{subsec:push-fwd}, and we consider the twisted de~Rham complex $\DR(\ccM\otimes\ccE^q)$. For each $\alpha\in[0,1)$, the filtration $F^\Del_{\alpha+\bullet}(\ccM\otimes\ccE^q)$ is a $F_\bbullet\cD_{X\times\PP^1}$-filtration, hence the twisted de~Rham complex is filtered by
\[
F^\Del_{\alpha+\bullet}\DR(\ccM\otimes\ccE^q)=\big\{0\ra F^\Del_{\alpha+\bullet}(\ccM\otimes\ccE^q)\To{\nabla}\Omega^1_{X\times\PP^1}\otimes F^\Del_{\alpha+\bullet+1}(\ccM\otimes\ccE^q)\ra\cdots\big\}.
\]

\begin{theorem}\label{th:E1degen}
For each $\alpha\in[0,1)$, the filtered complex
\[
\bR\Gamma\big(X\times\PP^1,F^\Del_{\alpha+\bullet}\DR(\ccM\otimes\ccE^q)\big)
\]
is strict.
\end{theorem}

\begin{remark}
From \cite[Prop.\,(1.3.2)]{DeligneHII} we deduce that the corresponding spectral sequence degenerates at $E_1$, hence \ref{th:main2}\eqref{th:main23}.
\end{remark}

\begin{lemma}\label{lem:E1degen}
The theorem holds if $(\ccM,F_\bbullet\ccM)$ is a mixed Hodge module on $\PP^1$, that~is, if we assume $X=\{\mathrm{pt}\}$ in the theorem.
\end{lemma}

\begin{proof}
According to \cite[Th.\,6.1]{Bibi08}, the theorem holds if $X=\{\mathrm{pt}\}$ and $(\ccM,F_\bbullet\ccM)$ is a polarizable Hodge module. Let $a_{\PP^1}:\PP^1\to\mathrm{pt}$ denote the constant map. More precisely, it follows from the proof of \cite[Prop.\,6.10]{Bibi08} that
\begin{equation}\label{eq:Hia}
H^ia_{\PP^1,+}\big(R_{F^\Del_{\alpha+\bbullet}}(\ccM\otimes\ccE^q)\big)=0\quad \text{for }i\neq0,
\end{equation}
and
\begin{equation}\label{eq:H0a}
H^0a_{\PP^1,+}\big(R_{F^\Del_{\alpha+\bbullet}}(\ccM\otimes\ccE^q)\big)=R_{F^\Del_{\alpha+\bbullet}}H^0a_{\PP^1,+}(\ccM\otimes\ccE^q).
\end{equation}
Let us prove by induction on the length of the weight filtration that \eqref{eq:Hia} and \eqref{eq:H0a} hold for a mixed Hodge module. If the length is $>1$, we find a short exact sequence of mixed Hodge modules whose underlying filtered exact sequence is
\[
0\to(\ccM_1,F_\bbullet\ccM_1)\to(\ccM,F_\bbullet\ccM)\to(\ccM_2,F_\bbullet\ccM_2)\to0,
\]
and \eqref{eq:Hia} and \eqref{eq:H0a} hold for the extreme terms. Theorem \ref{th:main2}\eqref{th:main22} gives an exact sequence
\[
0\to R_{F^\Del_{\alpha+\bbullet}}(\ccM_1\otimes\ccE^q)\to R_{F^\Del_{\alpha+\bbullet}}(\ccM\otimes\ccE^q)\to R_{F^\Del_{\alpha+\bbullet}}(\ccM_2\otimes\ccE^q)\to 0.
\]
Then the long exact sequence for $H^*a_{\PP^1,+}$ shows that \eqref{eq:Hia} holds for the middle term and we have an exact sequence
\begin{multline*}
0\to H^0a_{\PP^1,+}\big(R_{F^\Del_{\alpha+\bbullet}}(\ccM_1\otimes\ccE^q)\big)\to H^0a_{\PP^1,+}\big(R_{F^\Del_{\alpha+\bbullet}}(\ccM\otimes\ccE^q)\big)\\
\to H^0a_{\PP^1,+}\big(R_{F^\Del_{\alpha+\bbullet}}(\ccM_2\otimes\ccE^q)\big)\to 0.
\end{multline*}
According to \eqref{eq:H0a}, the extreme terms have no $\CC[\hb]$-torsion, that is, ``$R_{F^\Del_{\alpha+\bbullet}}$ commutes with $ H^0a_{\PP^1,+}$''. Then the same property holds for the middle term.
\end{proof}

\begin{proof}[\proofname\ of Theorem \ref{th:E1degen}]
We regard $\bR\Gamma\big(X\times\PP^1, R_{F^\Del_{\alpha+\bbullet}}\DR(\ccM\otimes\ccE^q)\big)$ as the complex $a_+R_{F^\Del_{\alpha+\bbullet}}(\ccM\otimes\ccE^q)$ up to a shift, where $a:X\times\PP^1\to\mathrm{pt}$ is the constant map. The theorem is a consequence of the strictness (\ie no $\CC[\hb]$-torsion) of this complex. Let us set $a=a_{\PP^1}\circ F$. Then, according to Lemma \ref{lem:E1degen}, \eqref{eq:Hia} and \eqref{eq:H0a} hold for $R_{F^\Del_{\alpha+\bbullet}}\cH^jF_+(\ccM\otimes\ccE^q)$ for each $j$. Using now the strictness given by Proposition \ref{prop:F+strict}, we have for each $j$
\[
H^j\big(a_+R_{F^\Del_{\alpha+\bbullet}}(\ccM\otimes\ccE^q)\big)=H^0a_{\PP^1,+}\big(R_{F^\Del_{\alpha+\bbullet}}\cH^jF_+(\ccM\otimes\ccE^q)\big),
\]
and therefore
\begin{align*}
H^j\big(a_+R_{F^\Del_{\alpha+\bbullet}}(\ccM\otimes\ccE^q)\big)&=R_{F^\Del_{\alpha+\bbullet}}H^0a_{\PP^1,+}\big(\cH^jF_+(\ccM\otimes\ccE^q)\big)\\
&=R_{F^\Del_{\alpha+\bbullet}}H^j\big(a_+(\ccM\otimes\ccE^q)\big),
\end{align*}
which is the desired result.
\end{proof}

\subsection{Proof of Theorem \ref{th:main}}\label{subsec:proofE1degen}

It follows from \cite{MSaito87} that the filtered $\cD_X$\nobreakdash-module $(\cO_X(*D),F_\bbullet\cO_X(*D))$ underlies a mixed Hodge module. We first remark that Corollary \ref{cor:comparison}, which is proved for $(\cO_X(*D),F_\bbullet\cO_X(*D))$ in the chart $X\times\Afu_{t'}$, also holds in the chart $X\times\Afu_t$ in a standard way. Therefore, in the statement of Theorem \ref{th:main}, we can replace $F^{\Yu,\lambda}(\nabla)$ with $F_{-\lambda}^\Del\DR(\ccE^f(*H))$, due to the first isomorphism in Corollary \ref{cor:comparison}.

Using now the second quasi-isomorphism in Corollary \ref{cor:comparison}, we are reduced to proving the injectivity of
\[
\bH^q\big(X\times\PP^1,F_{-\lambda-1}^\Del\DR(i_{f,+}\cO_X(*D)\otimes \ccE^q)\big)\to\bH^q\big(X\times\PP^1,\DR(i_{f,+}\cO_X(*D)\otimes \ccE^q)\big)
\]
for each $q$, \ie the strictness of $\bR\Gamma\big(X\times\PP^1,F_\bbullet^\Del\DR(i_{f,+}\cO_X(*D)\otimes \ccE^q)\big)$. This follows from Theorem \ref{th:E1degen} applied to $i_{f,+}\cO_X(*D)$.\qed

\appendix

\numberwithin{equation}{section}

\section{$E_1$-degeneration and strictness}\label{sec:degenerationstrictness}
Let $X$ be a complex manifold and let $(\ccM,F_\bbullet\ccM)$ be a coherent $\cD_X$-module equipped with a good filtration. It defines a coherent graded $R_F\cD_X$-module $R_F\ccM$ and this correspondence induces an equivalence of categories consisting of the corresponding objects.

In \cite[\S2]{MSaito86} (\cf also \cite{Laumon83}) is constructed the bounded derived category $D^\rb F(\cD_X)$ of filtered complexes of $\cD_X$-modules together with an equivalence $R_F$ with the bounded derived category $D^\rb(\gr\text{-}R_F\cD_X)$ of graded $R_F\cD_X$-modules. The subcategory $D^\rb_\coh F(\cD_X)$ is by definition $R_F^{-1}$ of the subcategory $D^\rb_\coh(\gr\text{-}R_F\cD_X)$ (graded $R_F\cD_X$-coherent cohomology). We have a commutative diagram of functors:
\[
\xymatrix{
D^\rb_\coh(\gr^F\cD_X)&\ar[l]_-{\gr} D^\rb_\coh F(\cD_X)\ar[r]^-{\text{forget}}\ar[d]^{\wr}&D^\rb_\coh(\cD_X)\\
&\ar[ul]^{\bL i_{\hb=0}^*}D^\rb_\coh(\gr\text{-}R_F\cD_X)\ar[ur]_{\bL i_{\hb=1}^*}&
}
\]
A bounded complex $(K^\cbbullet,F_\bbullet)$ of filtered $\cD_X$-modules is said to be \emph{strict} if for each $j$ and each $p$, the morphism $\cH^j(F_pK^\cbbullet)\to\cH^j(F_{p+1}K^\cbbullet)$ is injective (this corresponds to the definition given in \cite[\S1.2.1, p.\,865, line $-5$]{MSaito86} with $\#I=1$). Equivalently, $\cH^j(R_FK^\cbbullet)$ has no $\hb$-torsion for all $j$ and, since this is a graded module, this is equivalent to $\cH^j(R_FK^\cbbullet)$ having no $\CC[\hb]$-torsion (\ie being $\CC[\hb]$-flat) for all $j$.

For an object $(K^\cbbullet,F_\bbullet)$ of $D^\rb_\coh F(\cD_X)$, strictness implies that $\cH^j(R_FK^\cbbullet)$ is graded $R_F\cD_X$-coherent and without $\CC[\hb]$-torsion, hence takes the form $R_F\cH^j(K^\cbbullet)$ for some good filtration on $\cH^j(K^\cbbullet)$. This filtration is nothing but the filtration
\[
F_p\cH^j(K^\cbbullet)=\image\Big[\cH^j(F_pK^\cbbullet)\to\cH^j(K^\cbbullet)\Big].
\]

Let $h:X\to Y$ be a proper morphism. The direct image functor $h_+:D^\rb_\coh F(\cD_X)\to D^\rb_\coh F(\cD_Y)$ (we use right $\cD$-modules here) is constructed in \cite{MSaito86} by using the equivalence of categories with induced filtered $\cD$\nobreakdash-modules. Through the $R_F$ functor, it corresponds to the direct image $h_+:D^\rb_\coh(\gr\text{-}R_F\cD_X)\to D^\rb_\coh(\gr\text{-}R_F\cD_Y)$ constructed by using the equivalence of categories with induced graded $R_F\cD_X$-modules. On the other hand, a functor $h_+:D^\rb_\coh(\gr\text{-}R_F\cD_X)\to D^\rb_\coh(\gr\text{-}R_F\cD_Y)$ can be defined directly as $\bR h_*\big(\cbbullet\otimes_{R_F\cD_X}R_F(\cD_{X\to Y})\big)$ and both coincide since $h$ is proper (\cf\cite{MSaito89b} for the analogue for $\cD$-modules).

As a consequence, given a coherent $\cD_X$-module with good filtration $(\ccM,F_\bbullet\ccM)$, the push-forward $h_+(\ccM,F_\bbullet\ccM)$ is strict if and only if $\cH^jh_+(R_F\ccM)$ is strict (\ie has no $\CC[\hb]$-torsion) for any $j$, and in such a case $\cH^jh_+(R_F\ccM)=R_F\cH^jh_+(\ccM)$, where $F_\bbullet\cH^jh_+(\ccM)$ is the filtration defined as $\image(\cH^jF_\bbullet h_+\ccM\to\cH^jh_+\ccM)$, and these morphisms are injective because, $h$ being proper, $\cH^jh_+\ccM=\varinjlim_p\cH^jF_ph_+\ccM$. It also follows that the corresponding spectral sequence degenerates at~$E_1$.

\section{A complement to \cite{Bibi05b}}\label{app:B}
We\marginpars{Appendix \ref{app:B} has been changed} keep the notation of \S\ref{subsec:partialLaplacespe}. Since the construction of $\cFcM$ only depends on $j_*j^*(\ccM,F_\bbullet\ccM)$, we may assume that $(\ccM,F_\bbullet\ccM)=j_*j^*(\ccM,F_\bbullet\ccM)$. Then $\cM$ is the $R_F\cD_{X\times\PP^1}$-submodule of $\wt\cM$ generated by $V_0\wt\cM$. Using the notation of \cite{Bibi05b} up to the shift already indicated in \S\ref{subsec:strictspe}, we have an exact sequence
\begin{equation}\label{eq:tau}
0\to\gr_1^{V^\tau}\!\cFcM\To{\tau}\gr_0^{V^\tau}\!\cFcM\to \cM\to0,
\end{equation}
which is essentially the horizontal line of \cite[(4.7)]{Bibi05b}, together with the identification $U_0/\tau U_1\simeq\cM$ (notation of \loccit), identification which is obtained through the equality $U_0\cap\tau\cFcM=\tau U_1$ proved exactly in the same way as \cite[(4.12)]{Bibi05b}. In the exact sequence above, we have an identification of the extreme terms as the Rees modules of filtered $\cD_{X\times\PP^1}$-modules underlying mixed Hodge modules (up to a shift of the filtration), but we have to make precise why the intermediate term is of the same kind, as stated in \S\ref{subsec:partialLaplacespe}.

Let us introduce a new factor $\PP^1$, with coordinates $t_1,t'_1$ and let us consider the diagonal inclusion $i:X\times\PP^1\hto X\times\PP^1\times\PP^1$ sending $(x,t)$ to $(x,t,t)$, and similarly in the chart $t'$. Denote by $j_1:X\times\PP^1\times\Afu_{t_1}\hto X\times\PP^1\times\PP^1$ the open inclusion, and by $p_1:X\times\PP^1\times\PP^1\to X\times\PP^1$ the projection on the first two factors. Set $(\ccM_1,F_\bbullet\ccM_1)=i_*(\ccM,F_\bbullet\ccM)$ (push-forward in the sense of mixed Hodge modules). Then there exists a mixed Hodge module $(\ccM'_1,F_\bbullet\ccM'_1)$ such that 
\begin{itemize}
\item
$(\ccM'_1,F_\bbullet\ccM'_1)=j_{1,*}j_1^*(\ccM'_1,F_\bbullet\ccM'_1)$,
\item
$p_{1,*}(\ccM'_1,F_\bbullet\ccM'_1)=0$,
\item
there is an exact sequence of mixed Hodge modules
\[
0\to j_{1,*}j_1^*p_1^*(\ccM,F_\bbullet\ccM)\to(\ccM'_1,F_\bbullet\ccM'_1)\to(\ccM_1,F_\bbullet\ccM_1)\to0.
\]
\end{itemize}

Such an object is classically obtained by the convolution operation of $(\ccM_1,F_\bbullet\ccM_1)$ with the mixed Hodge module on $\PP^1$ obtained by extending the constant mixed Hodge module on $\PP^1\setminus\{0,\infty\}$ by the functors $j_{\infty,*}$ and $j_{0,!}$, with $j_0:\PP^1\setminus\{0,\infty\}\hto \PP^1\setminus\{\infty\}$ and $j_\infty:\PP^1\setminus\{0,\infty\}\hto \PP^1\setminus\{0\}$.

From the exact sequence $0\to \cH^{-1}i_{0,+}i_0^+\cO_{\Afu}\to j_{0,!}j_0^*\cO_{\Afu}\to\cO_{\Afu}\to0$ of $\cD_{\Afu}$-modules and the corresponding sequence of mixed Hodge modules (see \cite[(3.5.8.1)]{MSaito86}) we deduce the exact sequence above, and a corresponding exact sequence of associated Rees modules
\[
0\to\cM''_1\to\cM'_1\to\cM_1\to0.
\]
We consider the partial Laplace transformation with respect to the variables $t_1,\tau$, and we get an exact sequence
\[
0\to\cFcM''_1\to\cFcM'_1\to\cFcM_1\to0
\]
of strictly specializable modules along $\tau=0$, according to \cite[Prop.\,4.1]{Bibi05b} and, by uniqueness of the $V$-filtration, the following sequence is also exact for any $\alpha$:
\begin{equation}\label{eq:gralphaM1}
0\to\gr_\alpha^{V^\tau}\!\cFcM''_1\to\gr_\alpha^{V^\tau}\!\cFcM'_1\to\gr_\alpha^{V^\tau}\!\cFcM_1\to0.
\end{equation}

That $\gr_0^{V^\tau}\!\cFcM$ is the Rees module of a filtered $\cD$-module underlying a mixed Hodge module (up to a shift of the filtration) now follows from the following lemma.

\begin{lemma}\mbox{}
\begin{enumerate}
\item
For each $\alpha\in[0,1]$, we have an identification $\gr_\alpha^{V^\tau}\!\cFcM=\cH^0p_{1,+}\gr_\alpha^{V^\tau}\!\cFcM_1$.
\item
For $\alpha\in[0,1)$ and any $j$ we have $\cH^jp_{1,+}\gr_\alpha^{V^\tau}\!\cFcM''_1=0$.
\end{enumerate}
\end{lemma}

Indeed, from the lemma and the exact sequence \eqref{eq:gralphaM1} we obtain
\[
\cH^0p_{1,+}\gr_0^{V^\tau}\!\cFcM'_1\isom\gr_0^{V^\tau}\!\cFcM.
\]
According to \eqref{eq:tau} for $\cFcM'_1$ and the vanishing of $p_{1,+}\cM'_1$, we obtain that $\tau$ induces an isomorphism
\[
\cH^0p_{1,+}\gr_1^{V^\tau}\!\cFcM'_1\isom\cH^0p_{1,+}\gr_0^{V^\tau}\!\cFcM'_1,
\]
hence we have an isomorphism
\[
\cH^0p_{1,+}\gr_1^{V^\tau}\!\cFcM'_1\isom\gr_0^{V^\tau}\!\cFcM.
\]
The argument of \cite[Prop.\,4.1]{Bibi05b} now applies to $\gr_1^{V^\tau}\!\cFcM'_1$, which is shown to underlie a mixed Hodge module (up to a shift of the filtration), and therefore $\cH^0p_{1,+}\gr_1^{V^\tau}\!\cFcM'_1$ also, by \cite{MSaito87}, hence $\gr_0^{V^\tau}\!\cFcM$ too, as wanted.\qed

\begin{proof}[Sketch of proof of the lemma]
For the first point one checks that, similarly to \eqref{eq:iotaE}, we have $\cFcM_1=i_+\cFcM$ (with an obvious extension of the meaning of $i$) and also $\gr_\alpha^{V^\tau}\!\cFcM_1=i_+\gr_\alpha^{V^\tau}\!\cFcM$ for each $\alpha$, so the result follows by using $p_1\circ i=\id$.

For the second point, we can regard $\cM''_1$ as the external product of $\cM$ with the Rees module $j_{1,*}R_F\cO_{\Afu_{t_1}}$ so that the operation $\cF$ only concerns the latter, and the assertion relies on the property that the Fourier transform of $\cO_{\Afu_{t_1}}$ as a $\cD_{\Afu_{t_1}}$-module is supported on $\tau=0$.
\end{proof}

\section{Proof of Proposition \ref{prop:Kontsevich2b}}\label{app:proof152}
Let $\varpi_X:\wt X\to X$ (\resp $\varpi:\wt\PP^1\to\PP^1$) be the real oriented blowing up of the irreducible components of $D$ (\resp of $\infty$). It induces an isomorphism $\wt X\setminus\nobreak\varpi_X^{-1}(D)\isom X \setminus D=U$. Recall that one can construct $\wt X$ by gluing local charts as follows. Let~$X_\alpha$ be charts of $X$ in which $D$ is equal to a union of hyperplane coordinates.

In the local setting of \S\ref{subsec:setup}, we set $x_i=\rho_i\exp(\ri\theta_i)$ ($i=1,\dots,\ell$), $y_j=\eta_j\exp(\ri\tau_j)$ ($j=1,\dots,m$). Then
\begin{equation}\label{eq:realblowup}
\wt X_\alpha=(S^1)^\ell\times(\RR_{\geq0})^\ell\times\wt\Delta^m\times\Delta^p,\quad\text{where }\wt\Delta^m=(S^1)^m\times(\RR_{\geq0})^m.
\end{equation}
Any holomorphic gluing between $X_\alpha$ and $X_\beta$ which is compatible with $D$ induces a holomorphic gluing between $X_\alpha\setminus D$ and $X_\beta\setminus D$ which extends in a unique way as a real analytic gluing between $\wt X_\alpha$ and $\wt X_\beta$. It satisfies therefore the cocycle condition, from which we obtain the real oriented blow-up map $\varpi_X$.

In a similar way one checks that the morphism $f:X\to\PP^1$ induces a map $\wt f:\wt X\to\wt\PP^1$. Set $S^1_\infty=\varpi^{-1}(\infty)$ and $\wt X_\infty=\wt f{}^{-1}(S^1_\infty)$. In the neighbourhood of $P_\red$ we will replace $X$ with the inverse image $X'$ of a disc $\Delta\subset\PP^1$ centered at $\infty$, with coordinate $t'$. We then denote by $g:X'\to\Delta$ the map induced by $1/f$, so that $P=(g)$. This map $g$ can be lifted as a map $\wt g:\wt X'\to\wt\Delta$, where $\wt\Delta$ has coordinates $(\exp\ri\arg t',\rho')$. In the local setting of \S\ref{subsec:setup}, we have
\[
\wt X_\infty=(S^1)^\ell\times\{\prod_{i=1}^\ell\rho_i=0\}\times\wt\Delta^m\times\Delta^p,
\]
and if $g=x^{\bme}$, then $|\wt g|=\rho^{\bme}$ and $\arg\wt g=\sum_{i=1}^\ell e_i\theta_i$. It follows that \hbox{$\wt g_{|\wt X_\infty}:\wt X_\infty\!\to\! S^1_\infty$} is a topological fibration (since the natural stratification of $\wt X_\infty$ is obviously Whitney and $\wt g_{|\wt X_\infty}$ is smooth on each stratum, and proper).

Denote by $Z\subset\wt X_\infty$ the closed subset whose complement consists of points in the neighbourhood of which $e^{1/g}$ has moderate growth (\ie in the neighbourhood of which $\text{Re}(g)<0$) and let $Z^0\subset Z$ be the closed subset $\ov{\wt g^{-1}(\RR_{\geq0})}\cap\wt X_\infty$. We have $Z=\wt X_\infty\cap\{\arg \wt g\in[-\pi/2,\pi/2]\}$ and $Z^0=\wt X_\infty\cap\{\arg \wt g=0\}$. Since $\wt g_{|\wt X_\infty}$ is a topological fibration, $Z^0$ is a deformation retract of $Z$. We consider the inclusions
\[
\xymatrix{
U\ar@{=}[d]\ar@<-2pt>@{^{ (}->}[r]^-{\alpha}&\wt X\setminus Z\ar@{^{ (}->}[d]\ar@<-2pt>@{^{ (}->}[r]^-{\beta}&\wt X\ar@{=}[d]&\ar@<2pt>@{_{ (}->}[l]Z\\
U\ar@<-2pt>@{^{ (}->}[r]^-{~\alpha^\pi}&\wt X\setminus Z^0\ar@<-2pt>@{^{ (}->}[r]^-{~\beta^\pi}&\wt X&\ar@<2pt>@{_{ (}->}[l]Z^0\ar@{_{ (}->}[u]
}
\]
and the exact sequences
\begin{gather*}
0\to \cF=\beta_!\alpha_*\CC_U\to \CC_{\wt X}=(\beta\circ\alpha)_*\CC_U\to \CC_Z\to 0,\\
0\to \cF'=\beta^\pi_!\alpha^\pi_*\CC_U\to \CC_{\wt X}=(\beta^\pi\circ\alpha^\pi)_*\CC_U\to \CC_{Z^0}\to 0.
\end{gather*}
In these exact sequences, it can be seen that the $*$ push-forwards are equal to the corresponding derived push-forward, that is, for example,  $R^k\alpha_*\CC_U=0$ for $k>0$.

\begin{lemma}
There is a natural quasi-isomorphism \hbox{$(\Omega^\cbbullet_X(*D),\nabla\!:=\!\rd\!+\!\rd f)\simeq\bR\varpi_*\cF$}.
\end{lemma}

\begin{proof}[Sketch of proof]
We can argue in two ways. Either we use the theorems of multi-variable asymptotic analysis of Majima \cite{Majima84} as in \cite[Prop.\,1]{Hien09}, or we factorize~$\varpi$ through the real blow-up space $\wh\varpi:\wh X\to X$ of the single divisor $P_\red$. Let us sketch the latter method. Using notation as above for $\wh X$, \cite[Th.\,5.1]{Bibi94} gives a quasi-isomorphism $(\Omega^\cbbullet_X(*D),\nabla)\simeq\bR\wh\varpi_*\bR\wh\beta_!\bR\wh\alpha_*\bR j_*\CC_U$, where $j$ denotes here the inclusion $U\hto X\setminus P_\red$. Let $p:\wt X\to\wh X$ denote the natural map, so that $\wh\alpha\circ j=p\circ\alpha$ and $\wh\beta\circ p=p\circ\beta$. Then the right-hand term above is isomorphic to $\bR\varpi_*\cF$.
\end{proof}

As a consequence, $H^k_{\DR}(U,\nabla)\simeq H^k(\wt X,\cF)$. Since $Z$ retracts to $Z^0$, the natural map $H^k(Z,\CC)\to H^k(Z^0,\CC)$ is an isomorphism, and therefore so is the map $H^k(\wt X,\cF')\to H^k(\wt X,\cF)$ for each $k$. The proof now reduces to finding a morphism $(\Omega_f^\cbbullet,\rd)\to\bR\varpi_*\cF'$ (in the derived category $D^b(\CC_X)$) and to proving that it is an isomorphism. Equivalently, we should find a morphism $(\Omega^\cbbullet_X(\log D)/\Omega_f^\cbbullet,\rd)\to\bR\varpi_*\CC_{Z^0}$ which should be an isomorphism, and should make the following diagram commutative:
\begin{equation}\label{eq:diagOmegaf}
\begin{array}{c}
\xymatrix{
(\Omega^\cbbullet_X(\log D),\rd)\ar[r]\ar[d]_{\wr}&(\Omega^\cbbullet_X(\log D)/\Omega_f^\cbbullet,\rd)\ar@{-->}[d]^{?}\\
\bR j_*\CC_U=\bR\varpi_*\CC_{\wt X}\ar[r]&\bR\varpi_*\CC_{Z^0}
}
\end{array}
\end{equation}

The question is now local around $P_\red$ and we can work with $g:X'\to\Delta$ already considered above. We will also denote by $(\Omega_g^\cbbullet,\rd)$ the complex $(\Omega_f^{\cbbullet,\an},\rd)_{|X'}$. We can describe it as follows. Working on~$\Delta$, we denote by $(\Omega_{t'}^\cbbullet,\rd)$ the complex $\{t'\cO_\Delta\To{\rd}\Omega^1_\Delta(\log0)\}$. Then $\Omega_g^\cbbullet=g^*\Omega_{t'}^\cbbullet\otimes\nobreak\Omega^\cbbullet_{X'}(\log D)$, according to \eqref{eq:Omegafloc}, and  the quotient complex $\Omega^\cbbullet_{X'}(\log D)/\Omega_g^\cbbullet$ can be obtained from the relative logarithmic de~Rham complex $\Omega^\cbbullet_{X'/\Delta}(\log D)$ by the formula
\[
\Omega^\cbbullet_{X'}(\log D)/\Omega_g^\cbbullet=\Omega^\cbbullet_{X'/\Delta}(\log D)/g\Omega^\cbbullet_{X'/\Delta}(\log D)
\]
(\cf Appendix \ref{app:saito}). Recall that $\Omega^\cbbullet_{X'/\Delta}(\log D)$ was defined by Steenbrink in \cite{Steenbrink76}:
\begin{align*}
\Omega^p_{X'/\Delta}(\log D)&=\bigwedge^p\big(\Omega^1_{X'}(\log D)/g^*\Omega^1_\Delta(\log0)\big)\\
&=\Omega^p_{X'}(\log D)\Big/\frac{\rd g}{g}\wedge\Omega^{p-1}_{X'}(\log D).
\end{align*}

The proof now decomposes in three steps, in order to treat the case of a non reduced divisor $P$. We first analyze the behaviour of the various objects by a ramification of the value of $g$, following \cite{Steenbrink77}. We then treat the case when the pole divisor is reduced but within the framework of V-manifolds. We finally treat the general case by pushing forward along the ramification morphism and taking invariants with respect to the corresponding group action.

\subsubsection*{Step one: ramification}
Inspired by the approach of M.\,Saito in Appendix \ref{app:saito}, we will argue as in \cite{Steenbrink77}. Let $e$ be a common multiple of the numbers $e_i$ ($e_i$ is the multiplicity of the $i$-th component of $P$) and let us consider the commutative diagram
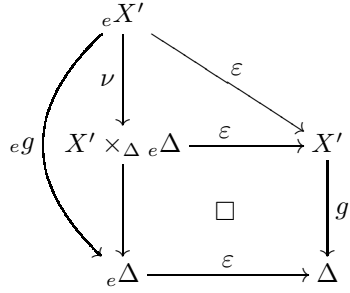
\[
\xymatrix@C=1.5cm@R=1.2cm{
{}_eX'\ar@/_2.5pc/[dd]_{{}_eg}\ar[rd]^-{\epsilon}\ar[d]_\nu&\\
X'\times_\Delta{}_e\Delta\ar[d]\ar[r]^\epsilon\ar@{}[dr]|\square&X'\ar[d]^g\\
{}_e\Delta\ar[r]^-{\epsilon}&\Delta
}
\]
where ${}_e\Delta$ has coordinate $u$, $\epsilon$ is defined by $\epsilon(u)=u^e$, and $\nu:{}_eX'\to X'\times_\Delta{}_e\Delta$ is the normalization morphism. Then ${}_eP=({}_eg)$ is a reduced divisor with V-normal crossings in the V-manifold ${}_eX'$. Set ${}_eX^{\prime*}={}_eX'\setminus {}_eP$. Then $\epsilon:{}_eX^{\prime*}\to X^{\prime*}$ is a covering of group $G=\ZZ/e\ZZ$, which also acts on ${}_eX'$ above the corresponding action on ${}_e\Delta$. Recall that $D=P_\red\cup H$. Then ${}_eD:={}_eP\cup\epsilon^{-1}H$ is also a reduced divisor with V-normal crossings.

We have the following local description of $({}_eX',{}_eP)$ (\cf\cite[Proof of Lem.\,2.2]{Steenbrink77}). We keep the notation of \S\ref{subsec:setup}. Set $d=\gcd(e_1,\dots,e_\ell)$, $e'=e/d$, $e'_i=e_i/d$, $c_i=e/e_i=e'/e'_i$ ($i=1,\dots,\ell$). Then ${}_eX'$ is the disjoint union of $d$ copies of the space $Y'$ obtained as the quotient of the space $Y$ having coordinates $((x'_i)_{i=1,\dots,\ell},(y_j)_{j=1,\dots,m},(z_k)_{k=1,\dots,p})$ by the subgroup $G'$ of $G'':=\ZZ/c_1\ZZ\times\cdots\times\ZZ/c_\ell\ZZ$ consisting of the elements $\alphag=(\alpha_1,\dots,\alpha_\ell)$ such that $\exp(\sum2\pi\ri\alpha_i/c_i)=1$, acting as
\begin{equation}\label{eq:actionGprime}
\alphag\cdot(\bmx',\bmy,\bmz)=((e^{2\pi\ri\alpha_i/c_i}x'_i)_{i=1,\dots,\ell},\bmy,\bmz).
\end{equation}
Each component $Y'$ is identified with the normalization of the space with equation $u^{e'}=\prod_{i=1}^\ell x_i^{e'_i}$ and the composed map $\nu\circ\pi$ from $Y$ to the latter is given by $x_i=x^{\prime c_i}_i$ ($i=1,\dots,\ell$). Lastly, the composed map $h:Y\To\pi Y'\To{{}_e g}{}_e\Delta$ is given by $u=h(\bmx',\bmy,\bmz)=\prod_{i=1}^\ell x'_i$ and, by definition, the action of $G'$ preserves the fibres of~$h$. We visualize this in the following diagram:
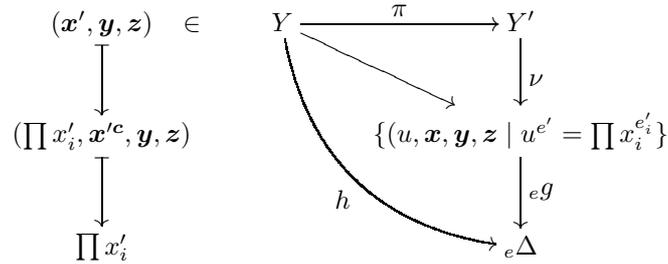
\[
\xymatrix{
(\bmx',\bmy,\bmz)\ar@{|->}[d]\ar@{}[r]|{\in}&Y\ar[r]^-{\pi}\ar[dr]\ar@/_2pc/[ddr]_h&Y'\ar[d]^\nu\\
(\prod x'_i,\bmx^{\prime\bmc},\bmy,\bmz)\ar@{|->}[d]&&\{(u,\bmx,\bmy,\bmz\mid u^{e'}=\prod x_i^{e'_i}\}\ar[d]^{{}_eg}\\
\prod x'_i&&{}_e\Delta
}
\]
As a consequence, the pull-back of the local component $P_i$ of $P_\red$ defined by~$x_i$ is, in each local connected component $Y'$ of ${}_eX'$, the V-smooth component ${}_eP_i$ which is the quotient of $\{x'_i=0\}$ by the induced action of $G'$.

\begin{lemma}
The natural morphisms of complexes
\begin{multline*}
(\Omega^\cbbullet_{X'}(\log D),\rd)\to(\Omega^\cbbullet_{X'/\Delta}(\log D),\rd)\\
\to(\Omega^\cbbullet_{X'/\Delta}(\log D)/g\Omega^\cbbullet_{X'/\Delta}(\log D),\rd)=(\Omega^\cbbullet_{X'}(\log D)/\Omega^\cbbullet_g,\rd)
\end{multline*}
can be obtained by taking $G$-invariants of the morphisms of complexes
\begin{multline*}
\epsilon_*(\Omega^\cbbullet_{{}_eX'}(\log{}_eD),\rd)\to\epsilon_*(\Omega^\cbbullet_{{}_eX'/{}_e\Delta}(\log{}_eD),\rd)
\\
\to\epsilon_*(\Omega^\cbbullet_{{}_eX'/{}_e\Delta}(\log{}_eD)/{}_eg\Omega^\cbbullet_{{}_eX'/{}_e\Delta}(\log{}_eD),\rd)=\epsilon_*(\Omega^\cbbullet_{{}_eX'}(\log{}_eD)/\Omega^\cbbullet_{{}_eg},\rd).
\end{multline*}
\end{lemma}

\begin{proof}
According to \cite[Rem.\,2.3]{Steenbrink77}, $\epsilon_*\cO_{{}_eX'}$ is $\cO_{X'}$ locally free of rank $e$ and the~$G$\nobreakdash-action is induced by the natural action $u\mto u\cdot\exp(2\pi\ri k/e)$, so that $\cO_{X'}=(\epsilon_*\cO_{{}_eX'})^G$. This action is compatible with the $G$-action on $\epsilon_*\cO_{X'\times_\Delta{}_e\Delta}$ and the induced action on $\epsilon_*(\cO_{X'\times_\Delta{}_e\Delta}/u\cO_{X'\times_\Delta{}_e\Delta})=\cO_{X'}/g\cO_{X'}$ is trivial. The same holds then for
\begin{align*}
\epsilon_*(\cO_{{}_eX'}/u\cO_{{}_eX'})&=\epsilon_*(\nu_*(\cO_{{}_eX'}/u\cO_{{}_eX'}))\\
&=\epsilon_*(\nu_*\cO_{{}_eX'}/u\nu_*\cO_{{}_eX'}))=\epsilon_*(\cO_{X'\times_\Delta{}_e\Delta}/u\cO_{X'\times_\Delta{}_e\Delta}).
\end{align*}

The sheaves $\Omega^k_{{}_eX'}(\log{}_eD)$ (\resp $\Omega^k_{{}_eX'/{}_e\Delta}(\log{}_eD)$) are $\cO_{{}_eX'}$-locally free and are locally identified with $\epsilon^*\Omega^k_{X'}(\log D)$ (\resp $\epsilon^*\Omega^k_{X'/\Delta}(\log D)$), \cf \cite[Rem.\,2.5]{Steenbrink77}, so that we have natural identifications
\begin{align*}
(\Omega^\cbbullet_{X'}(\log D),\rd)&=\big[\epsilon_*(\Omega^\cbbullet_{{}_eX'}(\log{}_eD),\rd)\big]^G,\\
(\Omega^\cbbullet_{X'/\Delta}(\log D),\rd)
&=\big[\epsilon_*(\Omega^\cbbullet_{{}_eX'/{}_e\Delta}(\log{}_eD),\rd)\big]^G.\qedhere
\end{align*}
\end{proof}

Consider the real blow-up space $\wt Y$ of~$Y$ along the components of the divisor $(\prod_{i=1}^\ell x'_i\prod_{j=1}^m y_j)$. The action \eqref{eq:actionGprime} of $G'$ on $Y$ extends to an action on $\wt Y$ (in the presentation like \eqref{eq:realblowup}, $G'$ only acts on the arguments $\theta'_i$) and the quotient space is by definition the real blow-up space~$\wt Y'$ of~$Y'$ along the components of the pull-back~${}_eD$ of $D$ in $Y'$, which is a divisor with V\nobreakdash-normal crossings. By the gluing procedure described above one defines a global map $\varpi_{{}_eX'}:\wt{{}_eX'}\to{}_eX'$. Note that the map $\wt\pi:\wt Y\to \wt Y'$ is a covering with group $G'$, and so the local charts $\wt Y'$ can also be described by a formula like \eqref{eq:realblowup}.

The map $\epsilon:{}_eX'\to X'$ lifts as a map $\wt\epsilon:\wt{{}_eX'}\to\wt{X'}$, which is a covering map of degree $e$ with group $G$, and ${}_eg$ lifts as a map $\wt{{}_eg}:\wt{{}_eX'}\to\wt{{}_e\Delta}$ giving rise to an obvious commutative diagram. It induces therefore a homeomorphism ${}_eZ^0:=\wt{{}_eg}^{-1}(\arg u=\nobreak0,|u|=\nobreak0)\isom Z^0$.

More precisely, $\wt Y$ has the form \eqref{eq:realblowup} with coordinates $((\arg x'_i)_i,(|x'_i|)_i)$ on the first two factors, and the map $\wt\epsilon$ is given by $\arg x_i=c_i\arg x'_i$ (a covering map of group $G''$) and $|x_i|=|x'_i|^{c_i}$ (a homeomorphism). The factorization through $\wt Y'$ consists in taking the quotient by $G'$ first, and then by $G''/G'$ on the arguments.

\subsubsection*{Step two: the case of a reduced divisor with $V$-normal crossings}
In order to simplify notation we will denote by $X',g,Z^0,S$ the objects previously denoted by ${}_eX',{}_eg,{}_eZ^0,{}_eD$. Let $\cC^\infty_{\wt X'}$ denote the sheaf of $C^\infty$ functions on $\wt X'$ (well defined in each chart~$\wt Y'$ as above, due to the local form \eqref{eq:realblowup}). Since $G'$ acts on~$\wt Y$ through the only factor $(S^1)^\ell$, the functions $\rho'_i=|x'_i|$ descend to $\wt Y'$, and we have $\cC^\infty_{\wt Y'}=(\wt\pi_*\cC^\infty_{\wt Y})^{G'}$, where $\wt\pi:\wt Y\to\wt Y'$ is the quotient map (covering map with group~$G'$). Recall also that the action of~$G'$ preserves the fibres of the map $\wt h:\wt Y\to\wt\Delta$.

Let $\cC^\infty_{\wt X'}(\log D)$ be the subsheaf of $j_*\cC^\infty_{X'{}^*}$ locally generated by $\cC^\infty_{\wt X'}$, $\log\rho'_i$ (\hbox{$i=1,\dots,\ell$}) and $\log\eta_j$ ($j=1,\dots,m$) in the local setting above. The logarithmic $1$-forms $\cA_{\wt X'}^1(\log D)$ are the linear combination with coefficients in $\cC^\infty_{\wt X'}(\log D)$ of the forms $\rd\rho'_i/\rho'_i$, $\rd\theta'_i$, $\rd\eta_j/\eta_j$, $\rd\tau_j$, $\rd\bmz$, $\rd\ov\bmz$ and we set $\cA_{\wt X'}^p(\log D)=\bigwedge^p\cA_{\wt X'}^1(\log D)$. We therefore get a logarithmic de~Rham complex $(\cA_{\wt X'}^\cbbullet(\log D),\rd)$. We have $(\cA_{\wt Y'}^\cbbullet(\log D),\rd)=(\wt\pi_*\cA_{\wt Y}^\cbbullet(\log D),\rd)^{G'}$ (where we still denote by $D$ the pull-back of $D\subset Y'$ in $Y$).

\begin{lemma}\label{lem:resolA}
The complex $(\cA_{\wt X'}^\cbbullet(\log D),\rd)$ is a resolution of $\CC_{\wt X'}$.
\end{lemma}

\begin{proof}
One first shows the result in the charts like $Y$, where it is proved in a standard way, and then one takes the $G'$-invariants.
\end{proof}

The sheaf $\cA_{\wt X',Z^0}^p(\log D)$ of logarithmic $p$-forms vanishing on $Z^0$ is the subsheaf of $\cA_{\wt X'}^p(\log D)$ locally defined as:
\begin{multline*}
\cA_{\wt X',Z^0}^p(\log D):=\big(|\wt g|,\log|\wt g|,(\exp(2\pi\ri\arg\wt g)-1)\big)\cA_{\wt X'}^p(\log D)\\[-3pt]
+\frac{\rd|\wt g|}{|\wt g|}\wedge\cA_{\wt X'}^{p-1}(\log D)
+\rd\arg\wt g\wedge\cA_{\wt X'}^{p-1}(\log D).
\end{multline*}
We will therefore set $\cA_{Z^0}^p(\log D)=\cA_{\wt X'}^p(\log D)/\cA_{\wt X',Z^0}^p(\log D)$.

Given a chart $Y'$ as above, let us set $Z^0_{Y'}=Z^0\cap Y'$ and $Z^0_Y=\wt\pi^{-1}(Z^0_{Y'})\subset\partial\wt Y$. Since $G'$ preserves the fibres of $\wt h$, it induces a covering $Z^0_Y\to Z^0_{Y'}$. As a consequence, if we define the sheaf $\cA_{\wt Y,Z^0_Y}^p(\log D)$ by the same formula as above, where we only replace $\wt X'$ with $\wt Y$ and $\wt g$ with $\wt h$, we have $(\cA_{\wt Y',Z^0_{Y'}}^\cbbullet(\log D),\rd)=(\wt\pi_*\cA_{\wt Y,Z^0_Y}^\cbbullet(\log D),\rd)^{G'}$. Defining $\cA_{Z^0_Y}^\cbbullet(\log D)$ similarly, we then also have $(\cA_{Z^0_{Y'}}^\cbbullet(\log D),\rd)=(\wt\pi_*\cA_{Z^0_Y}^\cbbullet(\log D),\rd)^{G'}$.

\begin{lemma}\label{lem:resolF}
The exact sequence of complexes
\[
0\to(\cA_{\wt X',Z^0}^\cbbullet(\log D),\rd)\to(\cA_{\wt X'}^\cbbullet(\log D),\rd)\to(\cA_{Z^0}^\cbbullet(\log D),\rd)\to0
\]
is a resolution of the exact sequence of sheaves
\[
0\to\cF'\to\CC_{\wt X'}\to\CC_{Z^0}\to0.
\]
\end{lemma}

\begin{proof}
In view of Lemma \ref{lem:resolA}, it is enough to prove that $(\cA_{Z^0}^\cbbullet(\log D),\rd)$ is a resolution of $\CC_{Z^0}$, and by the same argument as above, it is enough to show the result in charts like $Y$. On each octant $\rho'_i=0$ ($i=1,\dots,\ell$) of $\rho'_1\cdots\rho'_\ell=0$, one identifies~$Z^0$ with $(S^1)^{\ell-1}\times(\RR_{\geq0})^{\ell-1}\times\wt\Delta^m\times\Delta^p$ with coordinates $e^{2\pi\ri\theta_{\neq i}},\rho'_{\neq i}$ on the first two factors, and the restriction of $(\cA_{Z^0}^\cbbullet(\log D),\rd)$ to this subset is equal to the complex defined as above for $\wt Y$ with the corresponding variables. We can then apply Lemma \ref{lem:resolA}.
\end{proof}

We have a natural morphism of complexes $\varpi^{-1}(\Omega^\cbbullet_{X'}(\log D),\rd)\!\to\!(\cA_{\wt X'}^\cbbullet(\log D),\rd)$.

\begin{lemma}
The image of $\varpi^{-1}(\Omega_g^\cbbullet,\rd)$ is contained in $(\cA_{\wt X',Z^0}^\cbbullet(\log D),\rd)$.
\end{lemma}

\begin{proof}
This follows immediately by expressing \eqref{eq:Omegafloc} in polar coordinates.
\end{proof}

We conclude that we have a commutative diagram:
\[
\xymatrix{
\varpi^{-1}(\Omega^\cbbullet_{X'}(\log D),\rd)\ar[d]\ar[r]&
\varpi^{-1}(\Omega^\cbbullet_{X'}(\log D)/\Omega_g^\cbbullet,\rd)\ar[d]\\
(\cA_{\wt X'}^\cbbullet(\log D),\rd)\ar[r]&
(\cA_{Z^0}^\cbbullet(\log D),\rd)
}
\]
and by using the adjunction $\id\to\varpi_*\varpi^{-1}$, we get the desired commutative diagram:
\begin{equation}\label{eq:diagpreviousstep}
\begin{array}{c}
\xymatrix@R=.3cm{
(\Omega^\cbbullet_{X'}(\log D),\rd)\ar[dd]\ar[r]&
(\Omega^\cbbullet_{X'}(\log D)/\Omega_g^\cbbullet,\rd)\ar[dd]\\ &\\
\varpi_*(\cA_{\wt X'}^\cbbullet(\log D),\rd)\ar[r]\ar@{-}[d]^-\wr&
\varpi_*(\cA_{Z^0}^\cbbullet(\log D),\rd)\ar@{-}[d]^-\wr\\
\bR\varpi_*\CC_{\wt X'}\ar[r]&\bR\varpi_*\CC_{Z^0}
}
\end{array}
\end{equation}

That the right vertical morphism is a quasi-isomorphism can now be checked fiberwise at points of $P_\red$. We will thus check this at the center of each chart~$Y'$. Since the variables~$\bmz$ do not play any role, we will simply forget them. Moreover, we can work in the corresponding chart $Y$ with the function $h$, and we take $G'$\nobreakdash-invariants to obtain the desired isomorphism in the chart $Y'$.

We have $\varpi^{-1}(0)=(S^1)^\ell\times(S^1)^m$ and $Z^0_o:=Z^0\cap\varpi^{-1}(0)$ is the fibre of the map $(e^{\ri\theta'_1},\dots,e^{\ri\theta'_\ell},e^{\ri\tau_1},\dots,e^{\ri\tau_m})\mto e^{\ri\sum_i\theta'_i}$ above~$1$. If we represent $H^p(\varpi^{-1}(0),\CC)$ as $\bigwedge^p\big\langle\rd\theta,\rd\tau\big\rangle$ (where $\langle\cbbullet\rangle$ denotes the $\CC$\nobreakdash-vector space generated by $\cbbullet$), then the map $H^p(\varpi^{-1}(0),\CC)\to H^p(Z^0_o,\CC)$ is represented by the quotient map
\[
\bigwedge^p\big\langle\rd\theta',\rd\tau\big\rangle\mto \bigwedge^p\big\langle\rd\theta',\rd\tau\big\rangle\Big/({\textstyle\sum_i\rd\theta'_i})\wedge\bigwedge^{p-1}\big\langle\rd\theta',\rd\tau\big\rangle.
\]

Let us now denote by $\Omega^\cbbullet(\log D),\Omega_h^\cbbullet$ the germs at the origin of the corresponding complexes. Then $(\Omega^\cbbullet(\log D),\rd)$ is quasi-isomorphic to the complex $(\bigwedge^\cbbullet\langle\rd x'/x',\rd y/y\rangle,0)$ and the identification $H^p(\Omega^\cbbullet(\log D),\rd)\simeq H^p(\varpi^{-1}(0),\CC)$ is by the isomorphism $\rd x'/x'\mto\ri\,\rd\theta'$, $\rd y/y\mto\ri\,\rd\tau$. We can now conclude thanks to the following lemma:

\begin{lemma}\label{lem:Omegafinjects}
For each $p$ we have
\[
H^p(\Omega^\cbbullet(\log D)/\Omega_h^\cbbullet,\rd)=\bigwedge^p\langle\rd x'/x',\rd y/y\rangle\big/\big(\textstyle\sum_i\rd x'_i/x'_i\big)\wedge\bigwedge^{p-1}\langle\rd x'/x',\rd y/y\rangle.
\]

\end{lemma}

\begin{proof}
For $\omega\in\Omega^p(\log D)$ such that $\rd\omega\in\Omega_h^{p+1}$, let us write $\omega$ as a power series $\sum_{\bma,\bmb}\omega_{\bma,\bmb}x^{\prime\bma} y^{\bmb}$ with $\omega_{\bma,\bmb}\in\bigwedge^p\langle\rd x'/x',\rd y/y\rangle$. According to \eqref{eq:Omegafloc}, we can restrict the sum to $\bma\not\geq(1,\dots,1)$. Then the condition $\rd\omega\in\Omega_h^{p+1}$ reads
\[
\frac{\rd(x^{\prime\bma} y^{\bmb})}{x^{\prime\bma} y^{\bmb}}\wedge\omega_{\bma,\bmb}\in\frac{\rd(\prod_ix'_i)}{\prod_ix'_i}\wedge\bigwedge^{p-1}\Big\langle\frac{\rd x'}{x'},\frac{\rd y}{y}\Big\rangle,\quad\forall\bma\not\geq(1,\dots,1),\forall\bmb.
\]
Since $\sfrac{\rd(x^{\prime\bma} y^{\bmb})}{x^{\prime\bma} y^{\bmb}}$ and $\sum_i\rd x'_i/x'_i$ are $\CC$-linearly independent in $\langle\sfrac{\rd x'}{x'},\sfrac{\rd y}{y}\rangle$ whenever $(\bma,\bmb)\neq(0,0)$ and $\bma\not\geq(1,\dots,1)$, we also have $\omega_{\bma,\bmb}\in(\sum_i\rd x'_i/x'_i)\wedge\bigwedge^{p-1}\langle\sfrac{\rd x'}{x'},\sfrac{\rd y}{y}\rangle$. As a consequence, $\omega-\omega_{0,0}$ belongs to the image of~$\rd$. Since
\begin{align*}
\bigwedge^p\Big\langle\frac{\rd x'}{x'},\frac{\rd y}{y} \Big\rangle\cap\rd\Omega^{p-1}(\log D)&=0\\
\tag*{and}
\bigwedge^p \Big\langle\frac{\rd x'}{x'},\frac{\rd y}{y} \Big\rangle\cap\Omega_h^p&=\frac{\rd(\prod_ix'_i)}{\prod_ix'_i}\wedge\bigwedge^{p-1}\Big\langle\frac{\rd x'}{x'},\frac{\rd y}{y}\Big\rangle,
\end{align*}
we obtain the desired identification of $H^p(\Omega^\cbbullet(\log D)/\Omega_h^\cbbullet,\rd)$ with $H^p(Z^0_o,\CC)$.
\end{proof}

\subsubsection*{Step three: end of the proof}
We now go back to the notation of the beginning of the proof. The group $G$ acts on $\wt\epsilon_*\cA_{\wt{{}_eX'}}^\cbbullet(\log{}_eD)$ in the following way. Let $\gamma:{}_eX'\to{}_eX'$ be induced by $u\mto\zeta u$ for some $\zeta$ with $\zeta^e=1$ and let $\wt\gamma$ be the corresponding lifting on $\wt{{}_eX'}$. Then, for a local section $\varphi$ of $\wt\epsilon_*\cA_{\wt{{}_eX'}}^\cbbullet(\log{}_eD)$, the correspondence $\varphi\mto\varphi\circ\wt\gamma$ induces an isomorphism $\wt\gamma^*:\wt\epsilon_*\cA_{\wt{{}_eX'}}^\cbbullet(\log{}_eD)\isom\wt\epsilon_*\wt\gamma_*\cA_{\wt{{}_eX'}}^\cbbullet(\log{}_eD)$. Since $\wt\epsilon\circ\wt\gamma=\wt\epsilon$, we get an action of $G$ on $\wt\epsilon_*\cA_{\wt{{}_eX'}}^\cbbullet(\log{}_eD)$ which satisfies $(\wt\epsilon_*\cA_{\wt{{}_eX'}}^\cbbullet(\log{}_eD))^G=\cA_{\wt X'}^\cbbullet(\log D)$. This action induces an action of $G$ on $\varpi_{X',*}\wt\epsilon_*\cA_{\wt{{}_eX'}}^\cbbullet(\log{}_eD)=\epsilon_*\varpi_{{}_eX',*}\cA_{\wt{{}_eX'}}^\cbbullet(\log{}_eD)$, and this action is compatible with the action of $G$ on $\epsilon_*\Omega^\cbbullet_{{}_eX'}(\log{}_eD)$ through the natural morphism considered in the previous step.

The diagram \eqref{eq:diagpreviousstep} of the previous step gives rise to a commutative diagram
\[
\xymatrix{
\epsilon_*\Omega^\cbbullet_{{}_eX'}(\log{}_eD)\ar[r]\ar[d]_\wr&\epsilon_*(\Omega^\cbbullet_{{}_eX'}(\log{}_eD)/\Omega^\cbbullet_{{}_eg})\ar[d]^\wr\\
\epsilon_*\varpi_{{}_eX',*}\cA_{\wt{{}_eX'}}^\cbbullet(\log{}_eD)\ar[r]&\epsilon_*\varpi_{{}_eX',*}\cA_{{}_eZ^0}^\cbbullet(\log{}_eD)
}
\]
Since $\wt\epsilon_*\cA_{{}_eZ^0}^\cbbullet(\log{}_eD)=\cA_{Z^0}^\cbbullet(\log D)$, the right vertical morphism gives an isomorphism
\[
(\Omega^\cbbullet_{X'}(\log D)/\Omega^\cbbullet_{g})\isom\varpi_{X',*}\cA_{Z^0}^\cbbullet(\log D)\simeq\bR\varpi_{X',*}\CC_{Z^0}.
\]
On the other hand, restricting the left vertical isomorphism to $G$-invariants induces an isomorphism
\[
\Omega^\cbbullet_{X'}(\log D)\isom\varpi_{X',*}\cA_{\wt{X'}}^\cbbullet(\log D)\simeq\bR\varpi_{X',*}\CC_{\wt X'}
\]
We thus have completed the diagram \eqref{eq:diagOmegaf}.\qed

\numberwithin{equation}{subsection} \section{Proof of \eqref{eq:Kontsevich}, after M.\,Kontsevich}\label{sec:Kontsevich}

Under a restricted condition, we proceed the Deligne-Illusie approach as in \cite{E-V92} for the Kontsevich complex $(\Omega_f^\cbbullet,\rd)$ by listing the necessary modification. We shall follow closely the notations therein. \smallskip

Let $\kappa$ be a perfect field of characteristic $p>0$. Fix $f:U\to\AA^1$ over $\kappa$ and a compactification $f:X\to\PP^1$ such that $D:=X\setminus U$ is a normal crossing divisor and the pole divisor $P$ of $f$ on $X$ has multiplicity one (\ie $P=P_{\mathrm{red}}$). We consider the sheaves $\Omega_f^\cbbullet \subset \Omega^\cbbullet(\log D)$ on $X$ defined as before.

\subsection{The Cartier isomorphism}
Consider the commutative diagram with Cartesian square
\[
\xymatrix{
X \ar[r]^F\ar[dr] & X' \ar[r]\ar[d] & X \ar[d] \\ & \Spec\kappa \ar[r] & \Spec\kappa
}
\]
defining the relative Frobenius $F$. Here the lower arrow is the absolute Frobenius. We have the Cartier isomorphism
\begin{equation}\label{Eq:Cartier} C^{-1}: \bigoplus_a \Omega^a_{X'}(\log D') \to \bigoplus_a \cH^a\big(F_*(\Omega^\cbbullet_X(\log D),\rd) \big)
\end{equation}
of $\cO_{X'}$-algebras. If $x$ is a local section of $\cO_X$, then
\begin{equation}\label{Eq:Cartier-map}
C^{-1}:
\begin{aligned}
x' &\mto x^p \\
\rd x' &\mto x^{p-1}\rd x.
\end{aligned}
\end{equation}
Here $x'=1\otimes x$ is the pullback in $\cO_{X'} = \kappa\otimes_{\kappa}\cO_X$.

\begin{lemma}\label{Lemma:HKont-HLog}
The inclusion $(\Omega_{X,f}^\cbbullet,\rd) \to (\Omega^\cbbullet_X(\log D),\rd)$ induces an inclusion
\[
\cH^a\big(F_*(\Omega_{X,f}^\cbbullet,\rd)\big) \to \cH^a\big(F_*(\Omega^\cbbullet_X(\log D),\rd)\big).
\]
\end{lemma}

\begin{proof}
We need to show that
\begin{equation}\label{Eq:Kont-Om-closed} \rd\Omega^a_X(\log D)\cap\Omega_{X,f}^{a+1} = \rd \Omega_{X,f}^a.
\end{equation}
Let $Z$ be the singular locus of the divisor $D$ and $j: X\setminus Z \to X$ the inclusion. Since~$\Omega_{X,f}^{a+1}$ is locally free and $Z$ is of codimension at least two, we have $j_*j^*(\Omega_{X,f}^{a+1}) = \Omega_{X,f}^{a+1}$. Thus one only needs to prove \eqref{Eq:Kont-Om-closed} on $X\setminus Z$ and hence we may assume that $D$ is smooth.

Let $\{x_i\}_{i=1}^n$ be local coordinates and assume that $f = x_1^{-1}$. Any element $\omega\in\Omega^a_X(\log D)$ can be written as
\[
\omega = \frac{\rd x_1}{x_1}\,\alpha + \beta
\]
with $\alpha\in\Omega_X^{a-1}$, $\beta\in\Omega_X^a$. To see the divisibility of $\beta$ by $x_1$, we may pass to the completion along $x_1$ and write
\[
\beta = \gamma + x_1\delta
\]
where $\gamma$ is not divisible by $x_1$. Now if $\rd\omega\in\Omega_{X,f}^{a+1}$, then $(\sfrac{\rd x_1}{x_1^2})\wedge\rd\beta \in \Omega^{a+2}_X(\log D)$. The latter implies that $\rd\gamma = 0$. We obtain $\rd\omega = \rd\eta$ with $\eta = \frac{\rd x_1}{x_1}\alpha+x_1\delta \in \Omega_{X,f}^a$.
\end{proof}

\begin{proposition}
The Cartier isomorphism \eqref{Eq:Cartier} sends $\Omega_{X',f'}^a$ to $\cH^a\big(F_*(\Omega_{X,f}^\cbbullet,\rd)\big)$ and induces an isomorphism
\[
C^{-1}: \bigoplus_a \Omega_{X',f'}^a \to \bigoplus_a \cH^a\left(F_*(\Omega_{X,f}^\cbbullet,\rd)\right)
\]
of $\cO_{X'}$-algebras.
\end{proposition}

\begin{proof}
Locally we have an explicit lifting
\[
\widetilde{C}^{-1}: \Omega^a_{X'}(\log D') \to Z^a(\Omega^\cbbullet_X(\log D),\rd)
\]
of \eqref{Eq:Cartier} given by the formula \eqref{Eq:Cartier-map} in the chain level. Here $Z^a(\Omega^\cbbullet_X(\log D),\rd)$ denotes the $\cO_{X'}$-module of cocycles. It is then clear that $\widetilde{C}^{-1}$ sends $\Omega_{X',f'}^a$ to $Z^a(\Omega_{X,f}^\cbbullet,\rd)$. As by Lemma \ref{Lemma:HKont-HLog}, the natural map $\cH^a\big(F_*(\Omega_{X,f}^\cbbullet,\rd)\big) \to \cH^a\big(F_*(\Omega^\cbbullet_X(\log D),\rd)\big)$ is injective, $C^{-1}$ restricts to a well-defined map $\Omega^a_{X',f'}\to\cH^a\big(F_*(\Omega_{X,f}^\cbbullet,\rd)\big)$, so one has the commutative diagram
\begin{equation}\label{Eq:CartierSq}
\xymatrix{
\Omega_{X',f'}^a \ar@{^{ (}->}[r]\ar[d] & \Omega^a_{X'}(\log D') \ar[d]^{C^{-1}}_\simeq \\
\cH^a\big(F_*(\Omega_{X,f}^\cbbullet,\rd) \big) \ar@{^{ (}->}[r] & \cH^a \big(F_*(\Omega^\cbbullet_X(\log D),\rd) \big)
}		
\end{equation}
Therefore the problem reduces to showing that the left vertical arrow is surjective.

We regard the involved sheaves as coherent $\cO_{X'}$-modules. The statement is local, so we may assume that there exists a Cartesian diagram
\[
\xymatrix{ X \ar[r]^F\ar[d]_\pi & X' \ar[d] \\ \AA^n \ar[r]^F & (\AA^n)' }
\]
with \'etale vertical morphisms such that $f = \pi^*(x_1\cdots x_\ell)^{-1}$ for some $\ell\leq n$. Also notice that
\[
\Omega_{X,f}^a = \pi^*\Omega_{\AA^n,(x_1\cdots x_\ell)^{-1}}^a \quad\text{and}\quad \Omega^a_X(\log D) = \pi^*\Omega_{\AA^n}^a(\log (x_1\cdots x_\ell)).
\]
Thus to prove the statement, we may assume $X = \AA^n$ and work with global sections of the sheaves. Moreover by the K\"unneth formula for the complex $(\Omega_f^\cbbullet,\rd)$ and the classical Cartier isomorphism for $(\Omega^\cbbullet(\log D),\rd)$, we only need to consider the case where $\ell=n$, \ie $f = (x_1\cdots x_n)^{-1}$ and $D=(x_1\cdots x_n)$.

The sheaf $\cH^a\big(F_*(\Omega_{X,f}^\cbbullet,\rd)\big)$ is equal to the image of $Z^a(\Omega_{X,f}^\cbbullet,\rd)$ into the sheaf in the lower right of the diagram \eqref{Eq:CartierSq}. Via the isomorphism in the right side of \eqref{Eq:CartierSq}, the $\cO_{X'}$-module $\cH^a\big(F_*(\Omega_{X,f}^\cbbullet,\rd)\big)$ corresponds to the intersection of $\kappa[x^p]$-modules\enlargethispage{\baselineskip}%
\[
\kappa[x]\cdot\left\{\frac{df}{f}\bigwedge^{a-1} \left\{\frac{\rd x_1}{x_1},\dots,\frac{\rd x_n}{x_n}\right\}, \frac{1}{f}\bigwedge^a \left\{\frac{\rd x_1}{x_1},\dots,\frac{\rd x_n}{x_n} \right\}\right\} \bigcap \kappa[x^p]\cdot\bigwedge^a \left\{\frac{\rd x_1}{x_1},\dots,\frac{\rd x_n}{x_n} \right\}
\]
where the left and right modules correspond to $\Omega_{X,f}^a$ and $\Omega^a_{X'}(\log D')$, respectively. Since the pole orders of $f$ are equal to one, the intersection equals
\[
\kappa[x^p]\cdot\left\{\frac{df}{f}\bigwedge^{a-1} \left\{\frac{\rd x_1}{x_1},\dots,\frac{\rd x_n}{x_n}\right\}, \frac{1}{f^p}\bigwedge^a \left\{\frac{\rd x_1}{x_1},\dots,\frac{\rd x_n}{x_n} \right\}\right\}
\]
and this completes the proof.
\end{proof}

\subsection{The lifting and the splitting}

Let $W_2$ be the ring of Witt vectors of length two of $\kappa$. We now make the assumption that $(U\subset X,f)$ has a lifting $(\wt{U}\subset\wt{X},\wt{f})$ to~$W_2$. Again $\wt{X}',\wt{f}',\dots$ will denote the base-change of $\wt{X},\wt{f},\dots$ under the absolute Frobenius $W_2\to W_2$.

\begin{proposition}\label{Prop:LiftFrob}
Locally in the Zariski topology there is a lifting $\wt{F}: \wt{X}\to\wt{X}'$ of the relative Frobenius $F$ such that
\[
\wt{F}^*\cO_{\wt{X}'}(-\wt{D}') = \cO_{\wt{X}}(-p\wt{D}) \quad\text{and}\quad \wt{F}^*(\wt{f}') = \wt{f}^p.
\]
\end{proposition}

\begin{proof}
This is shown in the proof of \cite[Prop.9.7]{E-V92}. Indeed locally there is an \'etale morphism
\[
\wt{\pi}: \wt{X} \to \AA^n = \Spec W_2[\wt{t}_1,\dots,\wt{t}_n]
\]
with $\wt{x}_i := \wt{\pi}^*\wt{t}_i$ such that
\[
\wt{f} = \frac{1}{\wt{x}_1\cdots \wt{x}_\ell} \quad\text{and}\quad \wt{D} = (\wt{x}_1\cdots \wt{x}_m)
\]
for some $\ell\leq m\leq n$. The morphism $\wt{F}:\wt{X} \to \wt{X}'$ defined by $\wt{F}^*(\wt{x}_i') = \wt{x}_i^p$ then has the desired property.
\end{proof}

\begin{theorem}\label{Thm:Splitting-I}
Fix a positive integer $i$ with $i < p$. The lifting $(\wt{U}\subset\wt{X},\wt{f})$ defines a splitting
\[
\bigoplus_{a=0}^i \cH^a\left(F_*(\Omega_{X,f}^\cbbullet,\rd)\right)[-a] \isom \tau_{\leq i}\left(F_*(\Omega_{X,f}^\cbbullet,\rd)\right)
\]
of the $i$-th truncation of $F_*(\Omega_{X,f}^\cbbullet,\rd)$ in the derived category of $\cO_{X'}$-modules.
\end{theorem}

By the standard thickening and base change arguments (\cf\cite[Cor.10.23]{E-V92}), we obtain the following.

\begin{corollary}
Let $U$ be a smooth quasi-projective variety defined over a field of characteristic zero and $f\in\cO_U(U)$. Suppose that $U$ has a compactification $X$ such that $X\setminus U$ is a normal crossing divisor and $f$ extends to $f:X\to\PP^1$ with only simple poles on $X$. Then the spectral sequence
\[
E_1^{pq} = H^q\big(X,\Omega_f^p\big) \Longrightarrow \bH^{p+q}\left(X,(\Omega_f^\cbbullet,\rd)\right)
\]
degenerates at $E_1$.
\end{corollary}

In the rest, we prove the above theorem by showing that with the choice of local liftings of the Frobenius given by Proposition \ref{Prop:LiftFrob}, the splitting
\begin{equation}\label{Eq:Splitting}
\bigoplus_{a=0}^i \cH^a \big(F_*(\Omega^\cbbullet_X(\log D),\rd) \big)[-a] \isom \tau_{\leq i} \big(F_*(\Omega^\cbbullet_X(\log D),\rd) \big)
\end{equation}
constructed in \cite[\S 10]{E-V92} actually induces the desired splitting. \smallskip

We thus fix a collection $\{X_\alpha,\wt{F}_\alpha\}$ where $\{X_\alpha\}$ is a covering of $X$ and $\wt{F}_\alpha: \wt{X}_\alpha \to\nobreak \wt{X}'_\alpha$ is a lifting of the relative Frobenius such that
\begin{equation}\label{Eq:Frob-fS}
\begin{aligned}
\wt{F}_\alpha^*(\wt{f}') &= \wt{f}^p && \text{on each $X_\alpha$ with $X_\alpha\cap P \neq \emptyset$} \\
\wt{F}_\alpha^*\cO_{\wt{X}_\alpha}(-\wt{D}_\alpha') &= \cO_{\wt{X}_\alpha}(-p\wt{D}_\alpha) && \text{on each $X_\alpha$ with $D_\alpha := X_\alpha\cap D \neq \emptyset$}.
\end{aligned}
\end{equation}

Attached to the covering $\{X_\alpha\}$, let
\[
Z\cC^j \subset \cC^j\big(\tau_{\leq i}F_*\big(\Omega^\cbbullet_X(\log D),\rd\big) \big),\quad Z\cC_f^j \subset \cC^j \big(\tau_{\leq i}F_*\big(\Omega_{X,f}^\cbbullet,\rd\big) \big)
\]
be the corresponding sheaves of \v{C}ech cocycles contained in the sheaves of \v{C}ech cochains at degree $j$. We trivially have $Z\cC_f^j = Z\cC^j \cap \cC^j\big(\tau_{\leq i}F_*\big(\Omega_{X,f}^\cbbullet,\rd\big)\big)$. Let $\mathfrak{S}_j$ be the symmetric group of $j$ letters. For $j<p$ consider the $\cO_{X'}$-linear map
\begin{align*}
\delta^j: \Omega^j_{X'}(\log D') &\to \big(\Omega^1_{X'}(\log D') \big)^{\otimes j} \\
\omega_1\wedge\cdots\wedge\omega_j &\mto \frac{1}{j!}\sum_{\sigma\in\mathfrak{S}_j}\mathrm{sign}(\sigma)\cdot \omega_{\sigma(1)}\otimes\cdots\otimes\omega_{\sigma(j)}.
\end{align*}
One defines a map $(\phi,\psi)^{\otimes j}$ (recalled below) sitting in the factorization of $C^{-1}$
\[
\xymatrix@C=1.5cm{
\big(\Omega^1_{X'}(\log D') \big)^{\otimes j} \ar[r]^-{(\phi,\psi)^{\otimes j}} & Z\cC^j \ar[d]^{\text{natural quotient}} \\
\Omega^j_{X'}(\log D') \ar[u]^{\delta^j}\ar[r]^-{C^{-1}} & \cH^j \big(F_*(\Omega^\cbbullet_X(\log D),\rd) \big).
}
\]
The splitting \eqref{Eq:Splitting} is then given by the collection $\theta:=\{ (\phi,\psi)^{\otimes j}\circ\delta^j\circ C \}_{j=0}^i$. We now show that $(\phi,\psi)^{\otimes j}\circ\delta^j$ sends $\Omega_{X',f'}^j$ to $Z\cC_f^j$ and thus $\theta$ induces the desired splitting for $\tau_{\leq i}F_*\big(\Omega_{X,f}^\cbbullet,\rd\big)$.

The map $(\phi,\psi)^{\otimes 0}$ is just the pullback $F^*$; while $(\phi,\psi)^{\otimes 1}$ is given by the pair
\[
\Omega^1_{X'}(\log D') \To{(\phi_{\alpha\beta},\psi_\alpha)} \bigoplus (F_*\cO_X)_{\alpha\beta} \oplus \bigoplus (F_*\Omega^1_X(\log D))_\alpha
\]
defined as follows. Take a system of local coordinates $\{\wt{x}_1,\dots,\wt{x}_n\}$ on $\wt{X}$ such that
\begin{equation}\label{Eq:f-S}
\wt{f} = \frac{1}{\wt{x}_1\cdots \wt{x}_\ell} \quad\text{and}\quad \wt{D} = (\wt{x}_1\cdots \wt{x}_m)
\end{equation}
for some $\ell\leq m\leq n$, and write $\wt{F}_\alpha^*(\wt{x}_j') = \wt{x}_j^p + pv_{\alpha,j}$
for some $v_{\alpha,j} \in \cO_{X_\alpha}$.
The second equation in \eqref{Eq:Frob-fS} says that there exists a unit $v$ on $\wt{X}_\alpha$ such that
$\wt{F}_\alpha^*\big(\prod_{j=1}^m \wt{x}_j'\big)
	= v\prod_{j=1}^m \wt{x}_j^p$,
which implies that
\[ (v-1)\prod_{j=1}^m \wt{x}_j^p
	= p\sum_{j=1}^m v_{\alpha,j}\prod^{i\neq j}_{1\leq i\leq m}\wt{x}_i^p. \]
Reducing mod $p$ and computing in the domain $\cO_{X_\alpha}$
implies that $v = 1+ pv'$ for some $v'\in\cO_{X_\alpha}$.
Thus the above identity reduces to
\[ v'\prod_{j=1}^m x_j^p
	= \sum_{j=1}^m v_{\alpha,j}\prod^{i\neq j}_{1\leq i\leq m}x_i^p \]
on $X_\alpha$
(cf.,~\cite[\S 8.7]{E-V92}).
Since $X_\alpha$ is smooth and the $x_i$ forms local coordinates,
one concludes that $x_j^p$ divides $v_{\alpha,j}$ for $1\leq j\leq m$.
We obtain that, for all $\alpha$,
\[
\wt{F}_\alpha^*(\wt{x}_j') =
\begin{cases}
\wt{x}_j^p(1+pu_{\alpha,j}) & \text{for $1\leq j\leq m$} \\
\wt{x}_j^p + pu_{\alpha,j} & \text{for $m< j\leq n$}
\end{cases}
\]
for some $u_{\alpha,j} \in\cO_{X_\alpha}$. Then the pair $(\phi_{\alpha\beta},\psi_\alpha)$ is defined by
\begin{align*}
\phi_{\alpha\beta}\Big(\frac{\rd x_j'}{x_j'} \Big) &= u_{\alpha,j}-u_{\beta,j},
& \psi_\alpha \Big(\frac{\rd x_j'}{x_j'} \Big) &= \frac{\rd x_j}{x_j} + \rd u_{\alpha,j} \\
\phi_{\alpha\beta}\left(\rd x_k'\right) &= u_{\alpha,k}-u_{\beta,k},
& \psi_\alpha\left(\rd x_k'\right) &= x_k^{p-1}\rd x_k + \rd u_{\alpha,k}
\end{align*}
for $1\leq j\leq m<k\leq n$, and it lands in $Z\cC^1$.

On the other hand, conditions \eqref{Eq:Frob-fS} and \eqref{Eq:f-S} imply that for all $\alpha$,
\[
\sum_{j=1}^r u_{\alpha,j} = 0.
\]
With this equation and the explicit description of the generators, a direct computation reveals that $(\phi,\psi)^{\otimes 1}$ indeed sends $\Omega_{X',f'}^1$ to $Z\cC_f^1$.

In general, $(\phi,\psi)^{\otimes j}$ is constructed as a product of $(\phi,\psi)^{\otimes 1}$ using the product structure on $\bigoplus_k Z\cC^k$ (see \cite[p.116]{E-V92}). In particular, $(\phi,\psi)^{\otimes j}$ is a sum of certain $j$-term products of $\phi_{\alpha\beta}$ and $\psi_\alpha$, which send $\Omega_{X',f'}^1$ to $F_*\Omega_{X,f}^0$ and $F_*\Omega_{X,f}^1$, respectively. On the other hand, notice that $\delta^j$ sends $\Omega_{X',f'}^j$ to the subspace
\[
\sum_{a=1}^j \left\{\omega_1\otimes\cdots\otimes\omega_j \mid \omega_a \in\Omega_{X',f'}^1 \right\} \subset \big(\Omega^1_{X'}(\log D') \big)^{\otimes j}.
\]
Since for any $k$ the exterior product $(\Omega^1_X(\log D))^{\otimes k} \to \Omega^k_X(\log D)$ sends both $\Omega_{X,f}^0\cdot(\Omega^1_X(\log D))^{\otimes k}$ and $\Omega_{X,f}^1\otimes(\Omega^1_X(\log D))^{\otimes k-1}$ to $\Omega_{X,f}^k$, one obtains that $(\phi,\psi)^{\otimes j}\circ\delta^j$ maps $\Omega_{X',f'}^j$ to $Z\cC_f^j$, which completes the proof.\qed

\subsection{The case $\dim X = p$}
\begin{theorem}
If $\dim X=p$,
the splitting of Theorem \ref{Thm:Splitting-I} extends to $i=p$.
\end{theorem}

\begin{proof}
Let $n=\dim X$.
We set $D=P+H$, where $H$ is the horizontal divisor of $f$. Recall that $P$ is assumed to be reduced. The wedge product of forms
\[
F_*\Omega^{n-i}_{X,f} \otimes_{\cO_{X'}} F_*\Omega^{i}_{X,f}(-H)
	\to F_*\Omega^{n}_{X}
\]
followed by the projection to the cohomology sheaf and then the Cartier operator
\[ F_*\Omega^{n}_{X} \to F_*\Omega^n_{X}/d\Omega^{n-1}_{X}
	\To{C} \Omega^n_{X'} \]
induces a perfect duality
(see proof of \cite[Lemma 9.20]{E-V92},
which adapts word by word to the sheaves here).

For $0\leq j <p$ we constructed $\cO_{X'}$-linear maps
\[ \Omega^j_{X',f'} \to \cC^j
	\big(\tau_{\leq j}(F_*(\Omega^\cbbullet_{X,f},\rd) \big), \]
which dualize to
\[ \tau_{\geq n-j} F_*\Omega^\cbbullet_{X,f}(-H)[n-j] \to
	 \cC^\cbbullet \big(\Omega^{n-j}_{X',f'}(-H')\big) \]
and induce a quasi-isomorphism
\begin{equation}\label{Eq:Trunc>1}
\tau_{\geq n-p+1} F_*\Omega^\cbbullet_{X,f}(-H) \to
	\bigoplus_{i=n-p+1}^n \cC^\cbbullet\big(\Omega^i_{X',f'}(-H')\big)[-i]
\end{equation}
(see \cite[p.\,119]{E-V92}).
We now use $n=p$ to conclude that
the kernel of the $\cO_{X'}$-linear surjective map
\begin{equation}\label{Eq:To-Trunc>0}
F_*\Omega^\cbbullet_{X,f}(-H) \to \tau_{\geq 1} F_*\Omega^\cbbullet_{X,f}(-H)
\end{equation}
is the single $\cO_{X'}$-coherent cohomology sheaf $\cH^0$
concentrated in degree $0$.
Thus, by cohomological dimension of coherent sheaves,
\eqref{Eq:To-Trunc>0} induces a surjection
\[ \bH^p(X', F_*\Omega^\cbbullet_{X,f}(-H) \otimes_{\cO_{X'}} \mathcal{M}) \to \bH^p(X', \tau_{\geq 1}
	F_*\Omega^\cbbullet_{X,f}(-H) \otimes_{\cO_{X'}}\mathcal{M}) \]
for any coherent sheaf $\mathcal{M}$ on $X'$.
Therefore by \eqref{Eq:Trunc>1} a surjection
\[
\bH^p(X', F_*\Omega^\cbbullet_{X,f}(-H) \otimes_{\cO_{X'}} \mathcal{M}) \to
H^0(X', \Omega^p_{X'}(\log P)\otimes_{\cO_{X'}} \mathcal{M}).
\]
Taking for $\mathcal{M}$ the $\cO_{X'}$-dual
of the last cohomology sheaf $\cH^p$ of $F_*\Omega^p_{X,f}(-H)$
yields the Cartier operator $C$ as a non-vanishing global section of
$\Omega^p_{X'}(\log P)\otimes_{\cO_{X'}} (\cH^p)^{\vee}$.
Any lifting
$\tilde{C} \in
	\bH^p(X',  F_*\Omega^\cbbullet_{X,f}(-H) \otimes_{\cO_{X'}} (\cH^p)^\vee)$
of $C$ defines then a splitting
of the natural $\cO_{X'}$-linear surjection
$F_*\Omega^\cbbullet_{X,f}(-H) \to \cH^p[-p]$.
\end{proof}

\numberwithin{equation}{section}
\let\oldthefootnote\thefootnote
\def\thefootnote{(*)}
\section{On the Kontsevich-de Rham complexes and Beilinson's maximal extensions, by Morihiko Saito\protect\footnotemark}\label{app:saito}
\markboth{\MakeUppercase{M. Saito}}{\MakeUppercase{Appendix} \ref{app:saito}. \MakeUppercase{On the Kontsevich-de Rham complexes}}
\makeatletter
\g@addto@macro\authors{Morihiko Saito}%
\g@addto@macro\addresses{\author{}}%
\g@addto@macro\shortauthors{Morihiko Saito}%
\g@addto@macro\addresses{\address{M.~Saito}{RIMS, Kyoto University, Kyoto 606-8502, Japan}}%
\makeatother

We will use the same notation as in the main part of the article.\footnotetext{This work is partially supported by Kakenhi 24540039.}
\let\thefootenote\oldthefootnote

Let $f:X\to S$ be a proper morphism of smooth complex algebraic varieties with $S=\PP^1$. Let $U$ be a Zariski open subset of $X$ such that $D:=X\setminus U$ is a divisor with simple normal crossings which contains $P:=X_{\infty}$. Here $X_s:=f^{-1}(s)$ for $s\in S$. M.\,Kontsevich defined a subcomplex $(\Omega_f^\cbbullet,\rd)$ of the logarithmic de Rham complex $(\Omega_X^\cbbullet(\log D),\rd)$ by
$$\Omega_f^j:=\ker\bigl(\rd f\wedge:\Omega_X^j(\log D)\to \Omega_X^{j+1}(*P)/\Omega_X^{j+1}(\log D)\bigr),$$
where $f$ is identified with a meromorphic function on $X$. (In this paper, we denote by $\Omega_X^j(\log D)$ the analytic sheaf on the associated analytic space. The reader can also use the algebraic sheaf in the main theorems by GAGA.) Kontsevich considered more generally the differentials $u\,\rd+v\,\rd f\wedge$ for $u,v\in\CC$, although we consider only the case $v=0$ in this Appendix. This is the reason for which we call $(\Omega_f^\cbbullet,\rd)$ the Kontsevich-de Rham complex (instead of the Kontsevich complex). Note that $\Omega_f^\cbbullet$ coincides with $\Omega_X^\cbbullet(\log D)$ on the complement of $P$, and we have $\rd f=-\rd g/g^2$ by setting $g:=f^*t'=1/f$, which is holomorphic on a neighborhood of $P$, where $t':=1/t$ with $t$ the affine coordinate of $\CC\subset\PP^1$.

Kontsevich showed that the filtered complex $\bR\Gamma(X,(\Omega_f^\cbbullet,F))$ is strict by using a method of Deligne-Illusie \cite{D-I87} in case $P$ is reduced (with $v=0$), where the Hodge filtration $F^p$ is defined by $\sigma_{\geq p}$ as in \cite{DeligneHII}. In this Appendix~\ref{app:saito} we prove this assertion without assuming $P$ reduced by using the theory of relative logarithmic de Rham complexes in \cite{Steenbrink76}, \cite{Steenbrink77}, \cite{S-Z85}. (This is quite different from the method in the main part of this paper.)

Set
$$\ov\Omega{}_{X/S}^\cbbullet(\log D):=\Omega_{X/S}^\cbbullet(\log D)/ g\,\Omega_{X/S}^\cbbullet(\log D)|_P,$$
where $\Omega_{X/S}^\cbbullet(\log D)$ is the relative logarithmic de Rham complex, see \cite{Steenbrink76}, \cite{S-Z85}. Let $j:U\hto X$ be the canonical inclusion. We have the following.

\begin{thsaito}\label{th:C1}
There is a short exact sequence of filtered complexes
\begin{equation}\label{eq:C1}
0\to(\Omega_f^\cbbullet,F)\to(\Omega_X^\cbbullet(\log D),F)\To\rho (\ov\Omega{}_{X/S}^\cbbullet(\log D),F)\to 0, 
\end{equation}
where the filtration $F^p$ is defined by $\sigma_{\geq p}$ as above. Moreover, there is a decreasing filtration $V$ on $\ov\Omega{}_{X/S}^\cbbullet(\log D)$ indexed discretely by $\QQ\cap[0,1]$ such that the Hodge filtration of the mixed Hodge complex calculating the nearby cycle sheaf $\psi_gj_*\QQ_U$ is given by
$$\bigoplus_{\alpha\in[0,1)}\gr^{\alpha}_V(\ov\Omega{}_{X/S}^\cbbullet(\log D),F),$$
where $\gr^{\alpha}_V\ov\Omega{}_{X/S}^\cbbullet(\log D)$ corresponds to $\psi_{g,\lambda}j_*\QQ_U$ with $\lambda:=\exp(-2\pi i\alpha)$.
\end{thsaito}

Here $\psi_{g,\lambda}$ denotes the $\lambda$-eigenspace of the monodromy on the nearby cycle functor~$\psi_g$, see \cite{Deligne73}. A variant of the second assertion of Theorem \ref{th:C1} is noted in \cite[Prop.\,2.1]{MSaito83c}. The proof uses as in \loccit the normalization of the unipotent base change together with the theory of logarithmic forms on $V$-manifolds in \cite{Steenbrink77}, and the relation between the $V$-filtration and the multiplier ideals is not used here.

\begin{thsaito}\label{th:C2}
After taking the cohomology over $X$ or $P$, the filtration $V$ in Theorem \ref{th:C1} splits by the action of $t'\partial_{t'}$ on the relative logarithmic de Rham cohomology groups forgetting the filtration $F$, and the image of the morphism $\rho_j$ between the $j$-th cohomology groups induced by $\rho$ is contained in the unipotent monodromy part so that $\rho_j$ is identified with the natural morphism $H^j(U,\CC)\to\psi_{t',1} R^j(f_U)_*\CC_U$ where $f_U:U\to S$ is the restriction of $f$ to $U$.
\end{thsaito}

This implies the following.

\begin{corsaito}\label{cor:C1}
The Hodge filtration $F$ on $\bR\Gamma(X,(\Omega_f^\cbbullet,F))$ is strict.
\end{corsaito}

We denote by $\QQ_{h,U}[n]$ the pure Hodge module of weight $n$ whose underlying perverse sheaf is $\QQ_U[n]$, where $n:=\dim X$. Let $\Xi_g(j'_*\QQ_{h,U}[n])$ be Beilinson's maximal extension \cite{Beilinson87} as a mixed Hodge module. This is defined by generalizing the definition in loc.~cit., so that we have a short exact sequence of mixed Hodge modules
\begin{equation}\label{eq:C2}
0\to\psi_{g,1}j_*\QQ_{h,U}[n-1]\to\Xi_g(j'_*\QQ_{h,U}[n]) \to j_*\QQ_{h,U}[n]\to 0,
\end{equation}
where $j':U\hto X\setminus P$ is the canonical inclusion. (In this paper, $\psi_g$ for mixed Hodge modules is compatible with the one for the underlying perverse sheaves without any shift of complex; hence $\psi_g[-1]$ preserves perverse sheaves and also mixed Hodge modules.)

\begin{thsaito}\label{th:C3}
Assume $P$ reduced so that $\psi_{g,1}j_*\QQ_{h,U}=\psi_gj_*\QQ_{h,U}$, and $V$ in Theorem \ref{th:C1} is trivial. By the filtered de Rham functor $\DR_X$, the short exact sequence \eqref{eq:C2} corresponds to the associated distinguished triangle of the short exact sequence \eqref{eq:C1} up to a shift of triangles. More precisely,
$$\psi_{g,1}j_*\QQ_{h,U}[n-1],\quad\Xi_g(j'_*\QQ_{h,U}[n]),\quad j_*\QQ_{h,U}[n]$$
in \eqref{eq:C2} respectively correspond to
$$\ov\Omega{}_{X/S}^\cbbullet(\log D)[n-1],\quad(\Omega_f^\cbbullet,F)[n],\quad (\Omega_X^\cbbullet(\log D),F)[n],$$
so that the extension class of the short exact sequence \eqref{eq:C2} corresponds to $\rho$ in~\eqref{eq:C1}.
\end{thsaito}

We have the inverse functor $\DR_X^{-1}$ which associates a complex of filtered $\cD$-module to a filtered differential complex (see \cite{MSaito86}). By this functor, the surjective morphism $\rho$ between the filtered differential complexes in \eqref{eq:C1} corresponds to an extension class between the corresponding filtered $\cD$-modules or mixed Hodge modules in \eqref{eq:C2} so that the {\it kernel} of $\rho$ corresponds to an {\it extension} of filtered $\cD$-modules or mixed Hodge modules (because of the difference in $t$-structures).

We thank C.\,Sabbah for useful discussions about the Kontsevich-de Rham complexes.

\medskip
In Section \ref{subsec:C2} we recall some basic facts from the theory of relative logarithmic de Rham complexes and Beilinson's maximal extension. In Section \ref{subsec:C3} we give the proofs of the main theorems and their corollary.

\subsection{Relative logarithmic de Rham complexes}\label{subsec:C2}
\numberwithin{equation}{subsubsection}

In this section we recall some basic facts from the theory of relative logarithmic de Rham complexes and Beilinson's maximal extension.

\subsubsection{Kontsevich-de Rham complexes}\label{subsubsec:C2.1}
With the notation of the introduction, set
$$\Omega_X^j\lDr:=\Omega_X^j(\log D),\quad \ov\Omega{}_X^j\lDr:=\Omega_X^j\lDr/g\,\Omega_X^j\lDr\big|_P.$$
The Kontsevich-de Rham complex is defined by
$$\aligned \Omega_f^j:=\,\,&\ker\bigl(\rd g/g^2\wedge:\Omega_X^j\lDr\to \Omega_X^{j+1}\lDr(*P)/\Omega_X^{j+1}\lDr\bigr)\\ =\,\,&\ker\bigl(\rd g/g\wedge:\Omega_X^j\lDr\to \ov\Omega{}_X^{j+1}\lDr\bigr).\endaligned$$
As for the morphism in the last term, we have
\begin{equation}\label{eq:C.2.1.1}
\Imm\big(\rd g/g\wedge:\Omega_X^j\lDr\ra\ov\Omega{}_X^{j+1}\lDr\big)= \coker\big(\rd g/g\,\wedge:\ov\Omega^{j-1}_X\lDr\ra\ov\Omega^j_X\lDr\big).
\end{equation}
Indeed, $(\Omega^\cbbullet_X\lDr,\rd g/g\wedge)\big|_P$ is {\it acyclic}, and so is $(\ov\Omega^\cbbullet_X\lDr,\rd g/g\,\wedge)$.

Let $\ov\Omega^j_{X/S}\lDr$ denote the right-hand side of the isomorphism \eqref{eq:C.2.1.1}. Then
\begin{equation}\label{eq:C.2.1.2}
\ov\Omega^j_{X/S}\lDr=\Omega^j_{X/S}\lDr/g\,\Omega^j_{X/S}\lDr\big|_P,
\end{equation}
where $\Omega^j_{X/S}\lDr:=\coker\big(\rd g/g\,\wedge:\Omega^{j-1}_X\lDr\to \Omega^j_X\lDr\big)$ as in \cite{Steenbrink76}, \cite{S-Z85}. Indeed, \eqref{eq:C.2.1.2} follows from the diagram of the snake lemma by applying it to the action of $\rd g/g\wedge$ on the short exact sequences
$$0\to\Omega^j_X\lDr\big|_P\To{g}\Omega^j_X\lDr\big|_P\to\ov\Omega^j_X\lDr\to 0.$$
So we get the short exact sequence of complexes \eqref{eq:C1} by \eqref{eq:C.2.1.1}-\eqref{eq:C.2.1.2}.

\subsubsection{Unipotent base change}\label{subsubsec:C2.2}
Let $P_i$ be the irreducible components of $P$ with $e_i$ the multiplicity of $P$ along $P_i$. Set $e:={\rm LCM}\{e_i\}$. Let $\Delta$ be a sufficiently small open disk around $\infty\in\PP^1$. Let $\pi:\wt\Delta\to\Delta$ be the $e$-fold ramified covering such that $\pi^*t'=\wt t^e$, where $t'=1/t$ with $t$ the affine coordinate of $\CC\subset\PP^1$, and $\wt t$ is an appropriate coordinate of $\wt\Delta$. Let $\wt X$ be the normalization of the base change $X\times_S\wt\Delta$, with $\wt D\subset\wt X$ the pull-back of $D\subset X$ by the natural morphism $\pi_{\wt X}:\wt X\to X$. Set
$$\wt\Omega_{\wt X}^\cbbullet\lDtr:=\wt\Omega_{\wt X}^\cbbullet(\log\wt D),\quad \wt\Omega_{\wt X/\wt\Delta}^\cbbullet\lDtr:=\wt\Omega_{\wt X/\wt\Delta}^\cbbullet(\log\wt D).$$
These are the logarithmic de Rham complex and the relative logarithmic de Rham complex defined in \cite{Steenbrink77}. Let $X'\subset X$ be the inverse image of $\Delta\subset\PP^1$, and $G$ be the covering transformation group of $\pi_{X'}:\wt X\to X'$ which is isomorphic to $\ZZ/e\ZZ$. Then
\begin{equation}
\big((\pi_{X'})_*\wt\Omega_{\wt X}^\cbbullet\lDtr\big)^G=\Omega_X^\cbbullet\lDr\big|_{X'},\quad \big((\pi_{X'})_*\wt\Omega_{\wt X/\wt\Delta}^\cbbullet\lDtr\big)^G= \Omega_{X/S}^\cbbullet\lDr\big|_{X'}. \label{eq:C.2.2.1}
\end{equation}
This can be shown by taking a local coordinate system $(x_1,\dots,x_n)$ of a unit polydisk $\Delta^n$ in $X'$ such that $g$ is locally written as
$$c\,\prod_{j=1}^\ell\,x_j^{e_j}\quad\text{with}\,\,\,c\in\CC^*,\,e_j\geq 1,\, \ell\in[1,n],$$
and then taking a finite ramified covering $\Delta^n\to\Delta^n$ such that the pull-back of $x_j$ is $x_j^{e/e_j}$ if $j\leq \ell$, and $x_j$ otherwise.

\subsubsection{Koszul complexes}\label{subsubsec:C2.3}
Assume $g=\prod_{i=1}^nx_i^{e_i}$ with local coordinates $x_1,\dots,x_n$, where $e_i\geq 1$ for any $i$. The logarithmic de Rham complex $(\Omega^\cbbullet_X\lDr,\rd)_0$ at $0\in X$ is isomorphic to the Koszul complex associated with $x_i\partial/\partial x_i$ ($i\in[1,n]$) acting on $\CC\{x_1,\dots,x_n\}$. This is quasi-isomorphic to the subcomplex associated with the zero actions on $\CC\subset\CC\{x_1,\dots,x_n\}$.

Setting $\delta_i:=dx_i/x_i$, this Koszul complex has a basis defined by $\delta_J:=\bigwedge_{i\in J}\delta_i$ for $J\subset[1,n]$. The relative logarithmic de Rham complex $(\Omega^\cbbullet_{X/S}\lDr,\rd)$ is a quotient complex defined by the relation (see \cite{Steenbrink76})
$$\sum_i\,e_i\delta_i=0.$$
This is quasi-isomorphic to a subcomplex associated to the zero actions on the subspace\enlargethispage{1.5\baselineskip}%
$$\CC\{z\}=\big\{\textstyle\sum_{\nu}c_{\nu}x^{\nu}\,\big|\,c_{\nu}=0\,\,\,\text{if} \,\,\,e_i\nu_j\ne e_j\nu_i\,\,\,\text{for some}\,\,\,i,j\big\}\subset \CC\{x_1,\dots,x_n\},$$
where $z=\prod_{i=1}^nx_i^{e_i/d}$ with $d:={\rm GCD}\{e_i\}$, see \loccit

\subsubsection{External products}\label{subsubsec:C2.4}
Let $D_i$ be a divisor with simple normal crossings on a complex manifold $X_i$ ($i=1,2$). Set $X:=X_1\times X_2$, $D:=D_1\times X_2\cup X_1\times D_2$. Then
\begin{equation}\label{eq:C.2.4.1}
\Omega_X^\cbbullet(\log D)=\Omega_{X_1}^\cbbullet(\log D_1)\boxtimes \Omega_{X_2}^\cbbullet(\log D_2).
\end{equation}

Let $g_1$ be a (possible non-reduced) defining equation of $D_1$. Let $g$ be the pull-back of $g_1$ by the projection. These can be viewed as morphisms to $S=\CC$. So we have the relative logarithmic de Rham complexes $\Omega_{X_1/S}^\cbbullet(\log D_1)$, $\Omega_{X/S}^\cbbullet(\log D)$, and
\begin{equation}\label{eq:C.2.4.2}
\Omega_{X/S}^\cbbullet(\log D)=\Omega_{X_1/S}^\cbbullet(\log D_1)\boxtimes \Omega_{X_2}^\cbbullet(\log D_2).
\end{equation}
These can be shown by using a basis as in \eqref{subsubsec:C2.3}.

\subsubsection{Beilinson's maximal extension}\label{subsubsec:C2.5}
For $a,b\in\ZZ$ with $a\leq b$, let $E_{a,b}$ be the variation of mixed $\QQ$-Hodge structure of rank $b-a+1$ on $S':=\CC^*$ having an irreducible monodromy and such that
$$\gr^W_{2i}E_{a,b}=\QQ_{S'}(-i)\quad\text{if $i\in[a,b]$, and $0$ otherwise.}$$
There are natural inclusions
$$E_{a,b}\hto E_{a,b'}\quad(a\leq b\leq b'),$$
and $\{E_{a,b}\}_{b\geq a}$ is an inductive system for each fixed $a$.

Let $g:X\to S$ be a function on a complex algebraic variety, where $S=\CC$ in this subsection. Set $X':=X\setminus g^{-1}(0)$ with $j:X'\hto X$ the natural inclusion. Let $g':X'\to S'$ be the morphism induced by $g$. For a mixed Hodge module $M'$ on~$X'$, it is known that
\begin{equation}\label{eq:C.2.5.1}
^p\psi_{g,1}M'=\ker\bigl(j_!(M'\otimes g'{}^*E_{0,b})\to j_*(M'\otimes g'{}^*E_{0,b})\bigr)\quad\text{for }b\gg 0, 
\end{equation}
where $^p\psi_{g,1}M':=\psi_{g,1}M'[-1]\,(:=\psi_{g,1}j_*M'[-1])$, which is a mixed Hodge module on $X$. More precisely, the kernels for $b\gg 0$ form a constant inductive system, and the images of the morphisms form an inductive system $\{I_b\}$ whose inductive limit vanishes, \ie the image of $I_b$ in $I_{b'}$ vanishes for $b'\gg b$, see the proof of \cite[Prop.\,1.5]{MSaito90b}.

Beilinson's maximal extension functor $\Xi_g$ (see \cite{Beilinson87}) can be defined for mixed Hodge modules $M'$ on $X'$ by
\begin{equation}\label{eq:C.2.5.2}
\Xi_gM':=\ker\bigl(j_!(M'\otimes g'{}^*E_{0,b})\to j_*(M'\otimes g'{}^*E_{1,b})\bigr)\quad\text{for}\,\,\,b\gg 0,
\end{equation}
so that there is a a short exact sequence of mixed Hodge modules on $X$
\begin{equation}\label{eq:C.2.5.3}
0\to \psi_{g,1}M'\to\Xi_gM'\to j_*M'\to 0.
\end{equation}
In fact, the kernel of the morphism in the right-hand side of \eqref{eq:C.2.5.2} is the same as $H^{-1}$ of the mapping cone of this morphism, and is identified with the kernel of the morphism
$$j_!(M'\otimes g'{}^*E_{0,b})\oplus j_*M'\to j_*(M'\otimes g'{}^*E_{0,b}).$$
So we get \eqref{eq:C.2.5.3} by using the diagram of the snake lemma, except for the surjectivity of the last morphism. The latter is shown by passing to the underlying $\QQ$-complexes, where the inductive limit $E_{a,+\infty}$ exists and the kernel in \eqref{eq:C.2.5.1}-\eqref{eq:C.2.5.2} can be replaced by the mapping cone with $b=+\infty$ as in \cite{Beilinson87}. This argument implies that the extension class defined by \eqref{eq:C.2.5.3} is induced by the natural inclusion
\begin{equation}\label{eq:C.2.5.4}
j_*M'\hto C\bigl(j_!(M'\otimes g'{}^*E_{0,b})\to j_*(M'\otimes g'{}^*E_{0,b})\bigr)\quad(b\gg 0).
\end{equation}

\subsection{Proof of the main theorems}\label{subsec:C3}

In this section we give the proofs of the main theorems and their corollary.

\subsubsection{Proof of Theorem \ref{th:C1}}\label{subsubsec:C3.1}
By \eqref{eq:C.2.1.1}-\eqref{eq:C.2.1.2} we get the short exact sequence of complexes \eqref{eq:C1}, which implies the strict compatibility with the filtration \hbox{$F^p=\sigma_{\geq p}$} (\ie we get short exact sequences after taking $\gr_F^p$). So the first assertion follows.

The second assertion follows from \eqref{eq:C.2.2.1} by using \cite{Steenbrink77}. Here the filtration~$V$ is induced by the $\wt t$-adic filtration on $(\pi_{X'})_*\wt\Omega_{\wt X/\wt\Delta}^\cbbullet\lDtr$, see also \cite[Prop.\,2.1]{MSaito83c}. This finishes the proof of Theorem \ref{th:C1}.\qed

\subsubsection{Proof of Theorem \ref{th:C2}}\label{subsubsec:C3.2}
The first assertion follows from the splitting of the $V$-filtration in the case $n=1$ by using a coordinate. So it remains to show that the composition
$$\overline{\rho}:\Omega_X^\cbbullet\lDr\To{\rho} \overline{\Omega}_{X/S}^\cbbullet\lDr\to \gr_V^0\overline{\Omega}_{X/S}^\cbbullet\lDr$$
represents the canonical morphism
$$\bR j_*\CC_U\to\psi_{g.1}\bR j_*\CC_U\quad\text{in}\,\,\,D^b_c(X,\CC).$$

We have the short exact sequence of differential complexes
\begin{equation}\label{eq:C.3.2.1}
0\to\Omega_X^\cbbullet\lDr(-P_\red)\to \Omega_X^\cbbullet\lDr\to \Omega_X^\cbbullet\lDr_{P_\red}\to 0,
\end{equation}
where
$$\Omega_X^\cbbullet\lDr(-P_\red):=\Omega_X^\cbbullet \lDr\otimes_{\cO_X}{\mathcal I}_{P_\red},\quad \Omega_X^\cbbullet\lDr_{P_\red}:=\Omega_X^\cbbullet \lDr\otimes_{\cO_X}\cO_{P_\red},$$
with ${\mathcal I}_{P_\red}$ the ideal of $P_\red$. Note that
$$\gr_V^0\overline{\Omega}_{X/S}^\cbbullet\lDr= \Omega_{X/S}^\cbbullet\lDr_{P_\red}.$$
Locally the three complexes in \eqref{eq:C.3.2.1} are the Koszul complexes of the action of $x_i\partial/\partial x_i$ on
$${\mathcal I}_{P_\red},\quad\cO_X,\quad\cO_X/{\mathcal I}_{P_\red},$$
if $D=\bigcup_{i=1}^n\{x_i=0\}$ with $x_1,\dots,x_n$ local coordinates of $X$.

It is well-known that \eqref{eq:C.3.2.1} represents the distinguished triangle
$$j''_!\bR j'_*\CC_U\to\bR j_*\CC_U\to i^*\bR j_*\CC_U \To{+1}\quad\text{in}\,\,\,D^b_c(X,\CC),$$
where $i:P\hto X$, $j':U\hto X\setminus P$, $j'':X\setminus P\hto X$ denote the inclusions so that $j=j''\circ j'$.

Set $\tau:=\log g$. Let $\Omega_X^\cbbullet\lDr[\tau]$ be the tensor product of $\Omega_X^\cbbullet\lDr$ with $\CC[\tau]$ over $\CC$, where the differential of the complex is given by
$$d(\omega\tau^k)=(d\omega)\tau^k+k(\rd g/g)\wedge\omega\tau^{k-1}\quad(k\geq 0),$$
and $\Omega_X^\cbbullet\lDr[\tau]$ can be identified with a double complex whose differentials are given by $d$ and $\rd g/g\wedge$. (This construction corresponds to the tensor product of the corresponding $\QQ$-complex with $E_{0,+\infty}$ in \ref{subsubsec:C2.5}.) We can define $\Omega_X^\cbbullet\lDr(-P_\red)[\tau]$, $\Omega_X^\cbbullet\lDr_{P_\red}[\tau]$ similarly, and get the quasi-isomorphism
\begin{multline}\label{eq:C.3.2.2}
\eta:C\bigl(\Omega_X^\cbbullet\lDr(-P_\red)[\tau] \to\Omega_X^\cbbullet\lDr[\tau]\bigr)\isom \Omega_X^\cbbullet\lDr_{P_\red}[\tau]\\
\isom \Omega_{X/S}^\cbbullet\lDr_{P_\red},
\end{multline}
where the last morphism sends $\sum_{i\geq 0}\omega_i\tau^i$ to the class of $\omega_0$. In fact, the last morphism is a morphism of complexes by the double complex structure explained above. It is a quasi-isomorphism in case $D=P$ by \cite[2.6]{Steenbrink76} (see also \cite{Deligne73}). The general case is reduced to this case by using the compatibility with the external product as in \ref{subsubsec:C2.4}. Note also that the compatibility of nearby cycles with external products implies that
\begin{equation}\label{eq:C.3.2.3}
\psi_{g,1}\bR j_*\CC_U=\bR(j_{X\setminus H})_*j_{X\setminus H}^* (\psi_g\bR j_*\CC_U),
\end{equation}
where $H$ is the closure of $D\setminus P$ in $X$ with $j_{X\setminus H}:X\setminus H\hto X$ the natural inclusion.

By the argument in \ref{subsubsec:C2.5} the extension class given by \eqref{eq:C.2.5.4} for $M'=j'_*\QQ_U[n]$ corresponds to the natural inclusion
\begin{equation}\label{eq:C.3.2.4}
\iota:\Omega_X^\cbbullet\lDr\hto C\bigl(\Omega_X^\cbbullet\lDr(-P_\red)[\tau]\to \Omega_X^\cbbullet\lDr[\tau]\bigr),
\end{equation}
and the composition of \eqref{eq:C.3.2.4} with \eqref{eq:C.3.2.2} coincides with the canonical morphism
$$\rho:\Omega_X^\cbbullet\lDr\to\Omega_{X/S}^\cbbullet \lDr_{P_\red},$$
\ie we have the commutative diagram
\begin{equation}\label{eq:C.3.2.5}
\begin{array}{c}
\xymatrix{
\Omega_X^\cbbullet\lDr\ar[r]^-{\iota}\ar@{=}[d]& C\bigl(\Omega_X^\cbbullet\lDr(-P_\red)[\tau]\to \Omega_X^\cbbullet\lDr[\tau]\bigr)\ar[d]^\eta\\
\Omega_X^\cbbullet\lDr\ar[r]^-{\rho}&\Omega_{X/S}^\cbbullet\lDr_{P_\red}
}
\end{array}
\end{equation}
So the assertion follows. This finishes the proof of Theorem \ref{th:C2}.\qed

\subsubsection{Proof of Corollary \ref{cor:C1}}\label{subsubsec:C3.3}
By \cite{Steenbrink76}, \cite{Steenbrink77}, \cite{S-Z85}, we have the filtered relative logarithmic de Rham cohomology sheaves
$$\cH^jf_*(\Omega_{X/S}^\cbbullet\lDr,F),$$
which are locally free sheaves forgetting $F$. Moreover the graded quotients $\gr_F^\cbbullet$ of the Hodge filtration $F$ commute with the cohomological direct images $\cH^j$ (\hbox{\ie $F$ is} strict), and give also locally free sheaves. (Indeed, these can be reduced to the unipotent monodromy case by using a unipotent base change together with logarithmic forms on $V$-manifolds as in \cite{Steenbrink77}.) These imply the strictness of the Hodge filtration $F$ on
$$\bR\Gamma(P,(\ov\Omega{}_{X/S}^\cbbullet\lDr,F)),$$
by using the short exact sequence of filtered complexes
$$0\to(\Omega_{X/S}^\cbbullet\lDr,F)\To{g}(\Omega_{X/S}^\cbbullet\lDr,F)\to(\ov\Omega{}_{X/S}^\cbbullet\lDr,F) \to 0.$$
Moreover, we have the strictness of
$$\bR\Gamma(P,\gr^0_V(\ov\Omega{}_{X/S}^\cbbullet\lDr,F)),$$
and $H^j(P,\gr^0_V(\ov\Omega{}_{X/S}^\cbbullet\lDr,F))$ gives the Hodge filtration on the unipotent monodromy part of the limit mixed Hodge structure of the variation of mixed Hodge structure
$$\cH^jf_*(\Omega_{X/S}^\cbbullet\lDr,F)|_{S'},$$
where $S'$ is the Zariski-open subset of $S$ over which $\cH^jf_*\CC_U$ is a local system. Hence the induced morphism $\rho_j$ in Theorem \ref{th:C2} is strictly compatible with $F$.

Corollary \ref{cor:C1} then follows from the assertion that the filtration on a mapping cone of filtered complexes $C((K,F)\to(K',F))$ is strict if $(K,F),(K',F)$ are strict and the morphisms $H^i(K,F)\to H^i(K',F)$ are strict. (This is a special case of a result of Deligne for bifiltered complexes \cite{DeligneHII}, since the mapping cone has a filtration $G$ such that $\gr_G^0=K$, $\gr_G^1=K'$ and the associated spectral sequence degenerates at $E_2$.) This finishes the proof of Corollary \ref{cor:C1}.\qed

\subsubsection{Proof of Theorem \ref{th:C3}}\label{subsubsec:C3.4}
Since $(\Omega_f^\cbbullet,F)$ is defined by the mapping cone of $\rho$, it is enough to show that the extension class defined by \eqref{eq:C2} corresponds to the morphism $\rho$ by the de Rham functor $\DR$, \ie the commutative diagram \eqref{eq:C.3.2.5} is compatible with $F$ in an appropriate sense. (Here we cannot define the filtration~$F$ on each term of \eqref{eq:C.3.2.5} since the condition $F_p=0$ for $p\ll 0$ cannot be satisfied.) Note, however, that the isomorphism between the mapping cones may be non-unique.

We will use the inverse functor $\DR^{-1}$ of $\DR$ which gives an equivalence of categories (see \cite[Prop.\,2.2.10]{MSaito86}):
\begin{equation}\label{eq:C3.4.1}
\DR^{-1}:D^bF(\cO_X,{\rm Diff})\cong D^bF(\cD_X),
\end{equation}
where the left-hand side is the bounded derived category of filtered differential complexes in the sense of \loccit, and the right-hand side is that of filtered left $\cD$-modules (by using the transformation between left and right $\cD$-modules).

For the proof of the assertion we may assume $D=P$ by the compatibility with the external product as in \ref{subsubsec:C2.4}. In fact, the direct image of filtered regular holonomic $\cD$-modules for the open inclusion of the complement of each irreducible component of the closure of $D\setminus P$ can be defined by the argument as in \cite[Prop.\,2.8]{MSaito87}. We apply this to each term of the short exact sequences associated with the given extension classes by using the exactness of the direct image functor in this case. Here the argument is much simpler than in loc.~cit., since we can show that the direct image in this case is analytic-locally isomorphic to an external product with the open direct image of a constant sheaf.

Set
$$(C^\cbbullet_k,F):=C\bigl((\Omega_X^\cbbullet\lDr(-P) [\tau]^{\leq k},F)\to(\Omega_X^\cbbullet\lDr[\tau]^{\leq k},F) \bigr)\quad(k\gg 0),$$
where $[\tau]^{\leq k}$ means the tensor product over $\CC$ with $\CC[\tau]^{\leq k}$ which is the subspace spanned by monomials of degree $\leq k$, and the Hodge filtration $F$ on $\CC[\tau]^{\leq k}$ is defined in a compatible way with $F$ on $E_{0,k}$ in \ref{subsubsec:C2.5}. By the same argument as in \ref{subsubsec:C3.2} there are canonical morphisms
\begin{equation}\label{eq:C.3.4.2}
\eta_k:(C^\cbbullet_k,F)\to(\Omega_{X/S}^\cbbullet \lDr_P,F),
\end{equation}
together the commutative diagram
\begin{equation}\label{eq:C.3.4.3}
\begin{array}{c}
\xymatrix{
(\Omega_X^\cbbullet\lDr,F)\ar[r]^-{\iota_k}\ar@{=}[d]&(C^\cbbullet_k,F)\ar[d]^{\eta_k}\\
(\Omega_X^\cbbullet\lDr,F)\ar[r]^-{\rho}&(\Omega_{X/S}^\cbbullet\lDr_P,F)
}
\end{array}
\end{equation}
By the argument in \ref{subsubsec:C2.5}, $\eta_k$ induces a morphism of filtered regular holonomic $\cD_X$-modules
\begin{equation}\label{eq:C.3.4.4}
\cH^{n-1}(\DR^{-1}\eta_k):\psi^{\cD}_{g,1}(\cO_X,F)\to\DR^{-1} (\Omega_{X/S}^\cbbullet\lDr_P,F)[n-1],
\end{equation}
where the source denotes the underlying filtered $\cD_X$-module of $\psi_{g,1}\QQ_{h,X}[n-1]$ (and it will be shown soon that the target is a filtered regular holonomic $\cD_X$-module). By \ref{subsubsec:C2.5} and \ref{subsubsec:C3.2}, it is enough to show that \eqref{eq:C.3.4.4} is a filtered isomorphism. In fact, this implies that the Hodge filtration on $\Omega_{X/S}^\cbbullet\lDr_P$ is the correct one. Moreover we have the vanishing of
$$\cH^n(\DR^{-1}\iota_k):\DR^{-1}(\Omega_X^\cbbullet\lDr,F)[n]\to \cH^n\DR^{-1}(C_k,F)\quad(k\gg 0),$$
by the compatibility with the transition morphism of the inductive system $\{(C_k,F)\}_{k\gg0}$. Hence $\DR^{-1}\iota_k$ defines an extension class of filtered regular holonomic $\cD$-modules in
$$\mathrm{Ext}^1\bigl(\DR^{-1}(\Omega_X^\cbbullet\lDr,F)[n], \cH^{n-1}\DR^{-1}(C_k,F)\bigr),$$
where $\DR^{-1}(\Omega_X^\cbbullet\lDr,F)[n]$ is isomorphic to the underlying filtered $\cD$-modules of $j_*\QQ_{h,U}[n]$ as is well-known. So the assertion follows from the commutativity of the diagram \eqref{eq:C.3.4.3} if \eqref{eq:C.3.4.4} is a filtered isomorphism.

By the argument in \ref{subsubsec:C3.2}, \eqref{eq:C.3.4.4} is an isomorphism if $F$ is forgotten. So the source and the target of \eqref{eq:C.3.4.4} have the common $\QQ$-structure $\psi_{g,1}\QQ_X[n-1]$. Let $W$ be the monodromy filtration on $\psi_{g,1}\QQ_X[n-1]$ shifted by $n-1$. The source of \eqref{eq:C.3.4.4} with this $\QQ$-structure and this weight filtration is the mixed Hodge module $\psi_{g,1}\QQ_{h,X}[n-1]$ by definition. By the construction in \cite{Steenbrink76}, the target with this $\QQ$-structure and this weight filtration belongs to ${\rm MHW}(X)$ where the latter category consists of successive extensions of pure Hodge modules without assuming any conditions on the extensions. Then \eqref{eq:C.3.4.4} is a filtered isomorphism by \cite[Prop.\,5.1.14]{MSaito86}. So Theorem \ref{th:C3} follows.\qed

\subsubsection{Example}\label{subsubsec:C3.5}
Assume
$$U=U_1\times U_2\subset X=\PP^1\times\PP^1,$$
with $U_1,U_2$ Zariski-open subsets of $\CC\subset\PP^1$, and $f:X\to S=\PP^1$ the second projection. Let $j_1:U_1\hto\PP^1$, $j_2:U_2\hto\PP^1$ be the natural inclusions. We assume $U_2=\CC$ since the assertion is only on a neighborhood of $P=X\times\{\infty\}$. We have
$$\bR j_*\CC_U=\bR(j_1)_*\CC_{U_1}\boxtimes\bR(j_2)_*\CC_{U_2},\quad \psi_g\bR j_*\CC_U=\bR(j_1)_*\CC_{U_1}\boxtimes\CC_{\{\infty\}},$$
and hence
$$\Xi_g\bR j'_*\CC_U[2]=\bR(j_1)_*\CC_{U_1}[1]\boxtimes\Xi_{g''}\CC_{U_2}[1].$$
Here $g''$ is a local coordinate of $\PP^1$ at $\infty$, and there is a non-splitting short exact sequence of perverse sheaves on $\PP^1$
$$0\to\CC_{\{\infty\}}\to\Xi_{g''}\CC_{U_2}[1]\to\bR(j_2)_*\CC_{U_2}[1]\to 0.$$
In this case we have
$$\aligned\mathrm{Ext}^1\bigl(\bR j_*\CC_U[2],\psi_g\bR j_*\CC_U[1]\bigr)&= \Hom\bigl(\bR(j_1)_*\CC_{U_1}\oplus\bR(j_1)_*\CC_{U_1}[-1], \bR(j_1)_*\CC_{U_1}\bigr)\\ &=H^0(U_1,\CC)\oplus H^1(U_1,\CC),\endaligned$$
and one cannot prove the main theorems of Appendix \ref{app:saito} by using this extension group.

\backmatter
\def\og{}\def\fg{}
\def\smfedsname{eds.}
\def\smfedname{ed.}
\newcommand{\SortNoop}[1]{}\newcommand{\eprint}[1]{\href{http://arxiv.org/abs/#1}{\texttt{arXiv\string:#1}}}
\providecommand{\smfandname}{\&}

\end{document}